\documentclass{amsart}

\usepackage{amssymb,amsmath,amsthm,mathtools,amsfonts}

\usepackage{graphicx}
\usepackage[usenames,dvipsnames]{xcolor}
\usepackage{cite}
\usepackage{url}
\usepackage{hyperref}
\usepackage{xcolor}

\hypersetup{
    bookmarks=true,         	
    unicode=false,          		
    pdftoolbar=true,        		
    pdfmenubar=true,        	
    pdffitwindow=false,     		
    pdfstartview={FitH},    		
   pdftitle={Long-Time Asymptotic Behavior of Solutions to the Derivative Nonlinear Schr\"{o}dinger Equation for Soliton-Free Initial Data},    					
   pdfauthor={Jiaqi Liu, Peter Perry, Catherine Sulem},     	
    linktocpage=true,			
    pdfnewwindow=true,      	
    colorlinks=true,       			
    linkcolor=red,          		
    citecolor=PineGreen,    	
    filecolor=magenta,      		
    urlcolor=cyan           		
}

\usepackage{tikz}
\usetikzlibrary{decorations.pathmorphing}

\usepackage{float}
\usepackage[section]{placeins}		

\theoremstyle{plain}
\newtheorem{theorem}{Theorem}[section]

\newtheorem{lemma}[theorem]{Lemma}
\newtheorem{proposition}[theorem]{Proposition}

\theoremstyle{definition}

\newtheorem{problem}[theorem]{Problem}

\theoremstyle{remark}
\newtheorem{remark}[theorem]{Remark}

\numberwithin{figure}{section}
\numberwithin{equation}{section}

\DeclareMathOperator{\ad}{ad}

\DeclareMathOperator{\Real}{Re}
\DeclareMathOperator{\Imag}{Im}

\allowdisplaybreaks

\newenvironment{doublecases}
{
	\left\{ 
			\begin{array}{lllll}
}
{			
			\end{array} 
			\right.
}

\begin{document}

\title[Long-Time Behavior of DNLS]{Long-Time Behavior of Solutions to the Derivative Nonlinear Schr\"{o}dinger Equation for Soliton-Free Initial Data}
\author{Jiaqi Liu}
\author{Peter A. Perry}
\author{Catherine Sulem}
\address[Liu]{Department of Mathematics, University of Kentucky, Lexington, Kentucky 40506--0027}
\address[Perry]{ Department of Mathematics, University of Kentucky, Lexington, Kentucky 40506--0027}
\address[Sulem]{Department of Mathematics, University of Toronto, Toronto, Ontario M5S 2E4, Canada }
\thanks{P. Perry supported in part by a Simons Research and Travel Grant.}
\thanks {C. Sulem supported in part by NSERC Grant 46179-13}
\date{\today}

\begin{abstract}
The large-time behavior of  solutions to  the derivative nonlinear Schr\"{o}dinger equation is established for initial 
conditions in  some weighted Sobolev spaces under the assumption that the initial conditions  do not support 
solitons. Our approach uses the inverse scattering setting and the  nonlinear steepest descent method of Deift and Zhou as  
recast by Dieng  and  McLaughlin.

\medskip

\centerline{\bf{ Comportement aux temps longs des solutions de l' \'equation }}
\centerline{\bf{de Schr\"odinger nonlin\'eraire
 avec d\'eriv\'ee en l'absence de solitons}}

\medskip
On \'etablit le comportement au temps long des solutions de l'\'equation de Schr\"odinger nonlin\'eraire avec d\'eriv\'ee
dans des espaces de Sobolev \`a poids, sous l'hypoth\`ese que les conditions initiales ne supportent pas de solitons.
Notre approche utlise  l'inverse scattering et  la m\'ethode  de la plus grande pente (``steepest descent'') de Deift  et Zhou revisit\'ee par
 Dieng  et McLaughlin.

\end{abstract}

\maketitle
\tableofcontents

%
%

\newcommand{\eps}{\varepsilon}
\newcommand{\lam}{\lambda}

\newcommand{\bfN}{\mathbf{N}}
\newcommand{\calbR}{\mathcal{ \breve{R}}}
\newcommand{\rhobar}{\overline{\rho}}
\newcommand{\zetabar}{\overline{\zeta}}

\newcommand{\rarr}{\rightarrow}
\newcommand{\darr}{\downarrow}

\newcommand{\dee}{\partial}
\newcommand{\dbar}{\overline{\partial}}

\newcommand{\dint}{\displaystyle{\int}}

\newcommand{\dotarg}{\, \cdot \, }

%
%

\newcommand{\RHP}{\mathrm{PC}}			
\newcommand{\PC}{\mathrm{PC}}

%
%

\newcommand{\zbar}{\overline{z}}

\newcommand{\bbC}{\mathbb{C}}
\newcommand{\bbR}{\mathbb{R}}

\newcommand{\calB}{\mathcal{B}}
\newcommand{\calC}{\mathcal{C}}
\newcommand{\calR}{\mathcal{R}}
\newcommand{\calS}{\mathcal{S}}

\newcommand{\ba}{\breve{a}}
\newcommand{\bb}{\breve{b}}

\newcommand{\balpha}{\breve{\alpha}}
\newcommand{\brho}{\breve{\rho}}

\newcommand{\tPhi}{{\widetilde{\Phi}}}

\newcommand{\bfe}{\mathbf{e}}
\newcommand{\bfn}{\mathbf{n}}

\newcommand{\bphi}{\breve{\Phi}}
\newcommand{\bN}{\breve{N}}
\newcommand{\bV}{\breve{V}}
\newcommand{\bR}{\breve{R}}
\newcommand{\bdelta}{\breve{\delta}}
\newcommand{\bzeta}{\breve{\zeta}}
\newcommand{\bbeta}{\breve{\beta}}
\newcommand{\bm}{\breve{m}}
\newcommand{\br}{\breve{r}}
\newcommand{\bnu}{\breve{\nu}}
\newcommand{\bbfN}{\breve{\mathbf{N}}}

\newcommand{\One}{\mathbf{1}}

%
%

\newcommand{\bigO}[2][ ]
{
\mathcal{O}_{#1}
\left(
{#2}
\right)
}

\newcommand{\littleO}[1]{{o}\left( {#1} \right)}

\newcommand{\norm}[2]
{
\left\Vert		{#1}	\right\Vert_{#2}
}

%
%

\newcommand{\rowvec}[2]
{
\left(
	\begin{array}{cc}
		{#1}	&	{#2}	
	\end{array}
\right)
}

\newcommand{\uppermat}[1]
{
\left(
	\begin{array}{cc}
	0		&	{#1}	\\
	0		&	0
	\end{array}
\right)
}

\newcommand{\lowermat}[1]
{
\left(
	\begin{array}{cc}
	0		&	0	\\
	{#1}	&	0
	\end{array}
\right)
}

\newcommand{\offdiagmat}[2]
{
\left(
	\begin{array}{cc}
	0		&	{#1}	\\
	{#2}	&	0
	\end{array}
\right)
}

\newcommand{\diagmat}[2]
{
\left(
	\begin{array}{cc}
		{#1}	&	0	\\
		0		&	{#2}
		\end{array}
\right)
}

\newcommand{\Offdiagmat}[2]
{
\left(
	\begin{array}{cc}
		0			&		{#1} 	\\
		\\
		{#2}		&		0
		\end{array}
\right)
}

\newcommand{\twomat}[4]
{
\left(
	\begin{array}{cc}
		{#1}	&	{#2}	\\
		{#3}	&	{#4}
		\end{array}
\right)
}

\newcommand{\unitupper}[1]
{	
	\twomat{1}{#1}{0}{1}
}

\newcommand{\unitlower}[1]
{
	\twomat{1}{0}{#1}{1}
}

\newcommand{\Twomat}[4]
{
\left(
	\begin{array}{cc}
		{#1}	&	{#2}	\\[10pt]
		{#3}	&	{#4}
		\end{array}
\right)
}

%
%
%

\newcommand{\JumpMatrixFactors}[6]
{
	\begin{equation}
	\label{#2}
	{#1} =	\begin{cases}
					{#3} {#4}, 	&	z \in (-\infty,\xi) \\
					\\
					{#5}{#6},	&	z \in (\xi,\infty)
				\end{cases}
	\end{equation}
}


%
%
%

\newcommand{\RMatrix}[9]
{
\begin{equation}
\label{#1}
\begin{aligned}
\left. R_1 \right|_{(\xi,\infty)} 	&= {#2} &	\qquad\qquad		
\left. R_1 \right|_{\Sigma_1}		&= {#3} 
\\[5pt]
\left. R_3 \right|_{(-\infty,\xi)} 	&= {#4} 	&	
\left. R_3 \right|_{\Sigma_2} 	&= {#5}
\\[5pt]
\left. R_4 \right|_{(-\infty,\xi)} 	&= {#6} &	
\left. R_4 \right|_{\Sigma_3} 	&= {#7} 
\\[5pt]
\left. R_6 \right|_{(\xi,\infty)}  	&= {#8} &	
\left. R_6 \right|_{\Sigma_4} 	&= {#9}
\end{aligned}
\end{equation}
}

%
%

%
%
%
%
%
%

\newcommand{\SixMatrix}[6]
{
\begin{figure}
\centering
\caption{#1}
\vskip 15pt
\begin{tikzpicture}
%
%
\draw[thick]	 (-4,0) -- (4,0);
\draw[thick] 	(-4,4) -- (4,-4);
\draw[thick] 	(-4,-4) -- (4,4);
\path[fill=gray,opacity=0.25] (0,0) -- (-4,4) -- (4,4) -- (0,0);
\path[fill=gray,opacity=0.25] (0,0) -- (-4,-4) -- (4,-4) -- (0,0);
%
%
\draw	[fill]		(0,0)						circle[radius=0.075];
\node[below] at (0,-0.1) 				{$\xi$};
%
%
\node[above] at (3.5,2.5)				{$\Omega_1$};
\node[below]  at (3.5,-2.5)			{$\Omega_6$};
\node[above] at (0,3.25)				{$\Omega_2$};
\node[below] at (0,-3.25)				{$\Omega_5$};
\node[above] at (-3.5,2.5)			{$\Omega_3$};
\node[below] at (-3.5,-2.5)			{$\Omega_4$};
%
%
\node[above] at (0,1.25)				{$\twomat{1}{0}{0}{1}$};
\node[below] at (0,-1.25)				{$\twomat{1}{0}{0}{1}$};
%
%
\node[right] at (1.20,0.70)			{$#3$};
\node[left]   at (-1.20,0.70)			{$#4$};
\node[left]   at (-1.20,-0.70)			{$#5$};
\node[right] at (1.20,-0.70)			{$#6$};
\end{tikzpicture}
\label{#2}
\end{figure}
}

%
%
%
%

\newcommand{\JumpMatrixLeftCut}[6]
{
\begin{figure}
\centering
\caption{#1}
\vskip 15pt
\begin{tikzpicture}[scale=0.9]
%
%
\draw [fill] (4,4) circle [radius=0.075];												
\node at (4.0,3.65) {$\xi$};																
%
%
\draw 	[->, thick]  	(4,4) -- (5,5) ;				
\draw		[thick] 		(5,5) -- (6,6) ;
\draw		[->, thick] 	(2,6) -- (3,5) ;				
\draw		[thick]		(3,5) -- (4,4);	
\draw		[->, thick]	(2,2) -- (3,3);				
\draw		[thick]		(3,3) -- (4,4);
\draw		[->,thick]	(4,4) -- (5,3);				
\draw		[thick]  		(5,3) -- (6,2);
%
%
\draw [ thick, blue, decorate, decoration={snake,amplitude=0.5mm}] (0,4)  -- (4,4);				
\node at (6.5,4) {$-\pi < \arg (\zeta-\xi) < \pi$};
%
%
\node at (8.5,8.5)  	{$\Sigma_1$};
\node at (-0.5,8.5) 	{$\Sigma_2$};
\node at (-0.5,-0.5)	{$\Sigma_3$};
\node at (8.5,-0.5) 	{$\Sigma_4$};
%
%
\node at (7,7) {${#3}$};							
\node at (1,7) {${#4}$};							
\node at (1,1) {${#5}$};							
\node at (7,1) {${#6}$};							
\end{tikzpicture}
\label{#2}
\end{figure}
}

%
%
%
%

\newcommand{\JumpMatrixRightCut}[6]
{
\begin{figure}
\centering
\caption{#1}
\vskip 15pt
\begin{tikzpicture}[scale=0.9]
%
%
\draw [fill] (4,4) circle [radius=0.075];						
\node at (4.0,3.65) {$\xi$};										
%
%
\draw 	[->, thick]  	(4,4) -- (5,5) ;								
\draw		[thick] 		(5,5) -- (6,6) ;
\draw		[->, thick] 	(2,6) -- (3,5) ;								
\draw		[thick]		(3,5) -- (4,4);	
\draw		[->, thick]	(2,2) -- (3,3);								
\draw		[thick]		(3,3) -- (4,4);
\draw		[->,thick]	(4,4) -- (5,3);								
\draw		[thick]  		(5,3) -- (6,2);
%
%
\draw [  thick, blue, decorate, decoration={snake,amplitude=0.5mm}] (4,4)  -- (8,4);				
\node at (1.5,4) {$0 < \arg (\zeta-\xi) < 2\pi$};
%
%
\node at (8.5,8.5)  	{$\Sigma_1$};
\node at (-0.5,8.5) 	{$\Sigma_2$};
\node at (-0.5,-0.5)	{$\Sigma_3$};
\node at (8.5,-0.5) 	{$\Sigma_4$};
%
%
\node at (7,7) {${#3}$};						
\node at (1,7) {${#4}$};						
\node at (1,1) {${#5}$};						
\node at (7,1) {${#6}$};						
\end{tikzpicture}
\label{#2}
\end{figure}
}

%
%

\section{Introduction}

%
%

This paper is devoted to the 
large-time asymptotic behavior  of  solutions  to  the  Derivative 
Nonlinear Schr\"{o}dinger Equation (DNLS)
\begin{equation}
\label{DNLS1}
i u_t + u_{xx} = i \varepsilon (|u|^2 u)_x \qquad x\in \bbR
\end{equation}
where  $\varepsilon = \pm 1$. It follows our recent work \cite{LPS15} (referred to hereafter as Paper I)  where we established global existence of solutions for initial conditions in weighted Sobolev spaces satisfying some
additional spectral constraints. 
To make these assumptions more precise, let us first fix  
$\eps=1$ (since solutions of \eqref{DNLS1} with 
$\eps=1$ are mapped to solutions of \eqref{DNLS1} with 
$\eps = -1$ by $u \mapsto u(-x,t)$).  It is convenient to 
consider  a gauge-equivalent form of \eqref{DNLS1}. 
 Under the transformation
\begin{equation}
\label{q.gauge}
 q(x,t) = u(x,t) 
 			\exp
 				\left(
 					{-i\eps \int_{-\infty}^x |u(y,t)|^2 dy}
 				\right),
\end{equation}
solutions of \eqref{DNLS1} are mapped into solutions of 
\begin{equation}
\label{DNLS2}
i q_t +q_{xx} +i  q^2 \bar q_x +\frac{1}{2} |q|^4 q=0.
\end{equation}
This equation  is sometimes referred to as the Gerjikov-Ivanov equation \cite{F00}.

It is well-known since the seminal article of Kaup and Newell \cite{KN78} that the DNLS equation is 
solvable by the inverse scattering method. In his doctoral thesis, Lee \cite{Lee83} studied in detail the spectral problem posed by Kaup and Newell, and  the direct and inverse scattering maps for 
 generic   Schwartz class data. 

In Paper I,  we develop a rigorous analysis of the direct and inverse scattering transform 
for a class of initial conditions $q_0(x)=q(x,t=0)$ belonging to the space  $H^{2,2}(\bbR)$ 
and obeying additional spectral constraints that rule out ``bright'' and algebraic solitons that led us to a global existence result in this setting.
Here,  $H^{2,2}(\bbR)$  denotes  
the completion of $C_0^\infty(\bbR)$ in the norm
$$
\norm{u}{H^{2,2}(\bbR)}
 	= \left( \norm{(1+|x|^2)u}{2}^2 + \norm{u''}{2}^2 \right)^{1/2}. 
$$
A recent work by Pelinovsky-Shimabukuro \cite{PS16}
addresses these questions in slightly different spaces.
In the present paper, we give a full description of the large-time behavior of solutions.
Before stating our assumptions and results more precisely,  we recall known results concerning the long-time behavior of DNLS solutions. The first results go back to 
the work of Hayashi, Naumkin and Uchida \cite{HNU99} where the authors consider a class of  one-dimensional 
nonlinear Schr\"odinger  equations with 
general nonlinearities containing first-order derivatives.They prove a global existence result for  smooth initial conditions that are small in some weighted 
Sobolev spaces, as well as  a  time-decay rate. Their analysis gives   the existence of asymptotic states and a logarithmic correction to the phase.

In the context of inverse scattering,  the first  work to provide
explicit formulas (i.e.,  depending only on initial conditions) 
for
large-time asymptotics of solutions
is due to  Zakharov and Manakov  \cite{ZM76} in the context of  the NLS equation.
In this setting, the  inverse scattering map  and the reconstruction of the solution (potential)  
is formulated through an oscillatory Riemann Hilbert problem (RHP).  The  latter  
(in our case,  Problem \ref{prob:DNLS.RH0})
consists of  an oriented contour specifying the discontinuities of a piecewise analytic function, and jump matrices relating their limits from above and below.  The solution to the original PDE is recovered from the asymptotics of solutions to the RHP (for our case, see the reconstruction formula \eqref{DNLS.q}). 

The now well-known steepest descent method of Deift and Zhou \cite{DZ93} provides a systematic  method to 
reduce the original RHP 
to a canonical model RHP whose solution is calculated in terms of parabolic cylinder  functions.  This reduction is done through a sequence of transformations whose effects do not change  the  large--time behavior 
of the recovered solution at leading order. 
In this way, one  obtains the asymptotic 
behavior of the solution in terms of the spectral data (thus in terms of the initial conditions) 
with a degree of precision that is not currently obtainable
through direct PDE methods.  
This approach  has  been applied to a number of
integrable systems including mKdV \cite{DIZ93,DZ93} and defocusing NLS \cite{DZ03}.

A formal analysis of general oscillatory  RHP with Schwartz class scattering data is presented in Varzugin
\cite{V96}. More recently, Do \cite{Do11} developed  a version of { the} Deift-Zhou steepest descent 
method that  emphasizes  real-variable methods and extends to a much larger class of RHPs. 
A key step in the nonlinear steepest descent method consists in deforming the contour 
associated to  the RHP in a way adapted to the structure of the phase function that  defines the oscillatory 
dependence on parameters (for our case, see \eqref{DNLS.V} for the jump matrix, 
\eqref{DNLS.phase} for the phase function, and
Figure \ref{fig:contour} for the deformation). 
When the entries of the jump matrix are not analytic, 
they must be approximated by rational functions so that the deformation can be carried out, and the error in the recovered solution due to the approximation must be estimated.

Dieng and McLaughlin \cite{DM08} proposed a variant of Deift-Zhou method combining  
steepest descent and $\bar{\partial}$-problem asymptotics. This approach 
allows a certain amount 
of non-analyticity in the RHP reductions,  leading to a $\bar{\partial}$-problem to 
be solved in some sectors of the complex plane  where analyticity of 
 the jump matrix (and hence the solution to the RHP) fails.
 The
new  $\bar{\partial}$-problem can be recast into an integral equation and solved by 
Neumann series.   
These ideas were implemented by Miller and McLaughlin \cite{MM08} to the 
study
of asymptotic stability of orthogonal polynomials.
In the context of NLS with soliton solutions, 
they were
successfully applied to prove  asymptotic stability of 
$N$-soliton solutions to defocusing NLS \cite{CJ14} 
and address the soliton resolution  problem for focusing NLS \cite{BJM16}.  

In this paper, we adapt this analysis  to
the DNLS equation for initial conditions excluding solitons, building on our Paper I 
where we proved the
Lipschitz continuity of  the direct and inverse scattering map from $H^{2,2}(\bbR)$ to itself. 
The presence of solitons will be addressed in a forthcoming  article.

To describe our approach, we recall that \eqref{DNLS2} generates an isospectral flow for the problem
\begin{equation}
\label{L}
\frac{d}{dx} \Psi = -i\zeta^2 \sigma_3 \Psi + \zeta Q(x) \Psi + P(x) \Psi
\end{equation}
where
$$ \sigma_3 = \diagmat{1}{-1}, \,\,\, Q(x) = \offdiagmat{q(x)}{\overline{q(x)}}, \,\, \, P(x) =\frac{i}{2} \diagmat{-|q(x)|^2}{|q(x)|^2}.$$
If $q \in L^1(\bbR) \cap L^2(\bbR)$, equation \eqref{L} admits bounded 
solutions 
for $\zeta \in \Sigma$ where
$$ \Sigma = \left\{ \zeta \in \bbC: \Imag(\zeta^2) = 0 \right\}. $$
For $\zeta \in \Sigma$ and 
$q \in L^1(\bbR)\cap L^2(\bbR)$,  there exist unique solutions $\Psi^\pm$ of \eqref{L} obeying the respective asymptotic conditions
$$\lim_{x \rarr \pm \infty} \Psi^\pm(x,\zeta) e^{ix\zeta^2 \sigma_3} = \diagmat{1}{1},$$
and there is a matrix $T(\zeta)$, the transition matrix, with 
 $\Psi^+(x,\zeta)=\Psi^-(x,\zeta) T(\zeta)$.
The matrix $T(\zeta)$ takes the form
\begin{equation} \label{matrixT}
 T(\zeta) = \twomat{a(\zeta)}{\bb(\zeta)}{b(\zeta)}{\ba(\zeta)} 
 \end{equation}
where $a$, $b$, $\ba$, $\bb$ obey the determinant relation
$$ a(\zeta)\ba(\zeta) - b(\zeta)\bb(\zeta) = 1 $$
and the symmetry relations (see Paper I, eq. (1.20))
\begin{align} \label{symmetry}
a(-\zeta) = a(\zeta), \quad b(-\zeta) = -b(\zeta), \quad \ba(\zeta)=\overline{a(\zetabar)}, \quad \bb(\zeta) = \overline{ b(\zetabar)}. 
\end{align}
In order to rule out algebraic and bright solitons, we assume that $q_0$ is so chosen that $a(\zeta)$ is nonvanishing on $\Sigma$ {(which rules out algebraic solitons)} and admits a zero-free analytic continuation to $\Imag(\zeta^2)<0$ 
(which rules out bright solitons). 

%
%

\begin{figure}[h!]
\caption{The Contours $\Sigma$ and $\mathbb{R}$}
\vskip 15pt
\begin{tikzpicture}[scale=0.7]
\path[fill=cyan,opacity=0.5]	(0,0) rectangle(4,4);
\path[fill=cyan,opacity=0.5]    (-4,-4) rectangle(0,0);
\path[fill=pink,opacity=0.5]		(-4,0) rectangle (0,4);
\path[fill=pink,opacity=0.5]		(0,-4) rectangle (4,0);
\draw[fill] (0,0) circle[radius=0.075];
\draw[thick,->,>=stealth] 	(0,0) -- (2,0);
\draw	[thick]    	(2,0) -- (4,0);
\draw[thick,->,>=stealth]	(0,0) -- (-2,0);
\draw[thick]		(-2,0) -- (-4,0);
\draw[thick,->,>=stealth]	(0,4) -- (0,2);
\draw[thick]		(0,2) -- (0,0);
\draw[thick,->,>=stealth]	(0,-4) -- (0,-2);
\draw[thick]		(0,-2) -- (0,0);
\node at (2,2) 	{$\Omega^+$};
\node at (-2,-2)	{$\Omega^+$};
\node at (-2,2)	{$\Omega^-$};
\node at (2,-2)	{$\Omega^-$};
\node[above] at (2,0.05)  	{$+$};
\node[below] at (2,-0.05)		{$-$};
\node[above] at (-2,0.05)		{$-$};
\node[below] at (-2,-0.05)	{$+$};
\node[left] at (-0.05,2)			{$-$};
\node[right] at (0.05,2)			{$+$};
\node[left] at (-0.05,-2)		{$+$};
\node[right] at (0.05,-2)		{$-$};
\end{tikzpicture}
\qquad
\begin{tikzpicture}[scale=0.7]
\path[fill=cyan,opacity=0.5] 	(-4,0)  rectangle (4,4);
\path[fill=pink,opacity=0.5] 	(-4,-4) rectangle (4,0);
\draw[fill] (0,0) circle[radius=0.075];
\draw[->,thick,>=stealth] 	(-4,0) -- (-2,0);
\draw[thick]		(-2,0) -- (0,0);
\draw[->,thick,>=stealth]	(0,0) -- (2,0);
\draw[thick]		(2,0) -- (4,0);
\node[above] at (0,1) {$\mathbb{C}^+$};
\node[below] at (0,-1) {$\mathbb{C}^-$};
\node[above] at (1.25,0.05) {$+$};
\node[below] at (1.25,-0.05) {$-$};
\end{tikzpicture}
\label{fig:contours}
\end{figure}
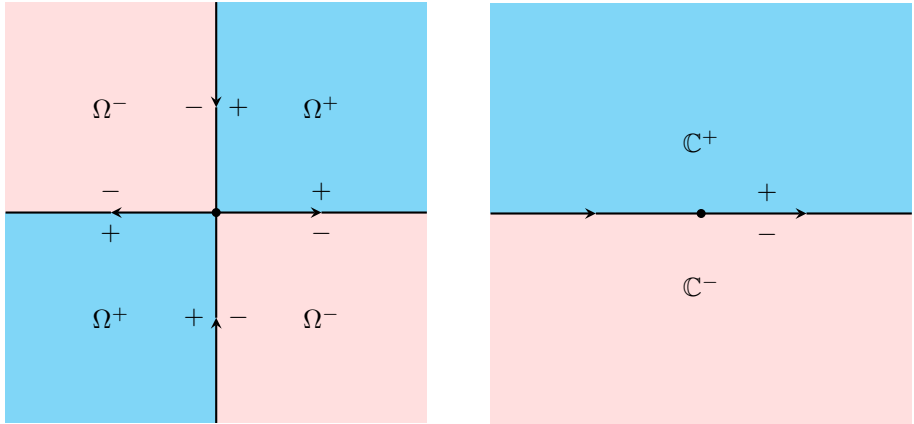

%
%

As shown in Section 1.2 of Paper I, the scattering data and Jost solutions, which are naturally functions of $\zeta \in \Sigma$, may be transformed to functions on $\bbR$, with consequent simplifications of the direct and inverse scattering problems.
Even functions $f$ on $\Sigma$ define functions $g$ on the real line $\bbR$ via $g(\zeta^2) = f(\zeta)$ and the map $\zeta \rarr \zeta^2$ maps the contour $\Sigma$ onto the contour $\bbR$. 
This fact, together with the 
symmetry relations \eqref{symmetry}, implies that the functions 
$\zeta^{-1} \bb(\zeta)/a(\zeta)$ and $\zeta^{-1} b(\zeta)/\ba(\zeta)$ induce functions $\rho(z)$ and $\brho(z)$ on the real line, and, under an appropriate change of variable (see Section 1.2 of Paper I), the Jost solutions may be regarded as functions of $z=\zeta^2$. The functions $\rho$ and $\brho$ are called the \emph{scattering data} 
for $q_0$. 

Figure \ref{fig:contours} displays the contours $\Sigma$ and $\bbR$ with their orientation as well as the sectors
$\Omega^\pm = \{ \zeta  \in \bbC\ : \ \pm ~{\rm Im } (\zeta^2) > 0\}$ and $\bbC^\pm = \{  z \in \bbC\ : \ \pm ~ {\rm Im } (z)   > 0\}$. 
The map $\zeta \mapsto \zeta^2$ preserves the orientations shown there.

We note the important identity
\begin{equation*}
a(\zeta) \ba(\zeta)   = (1-z|\rho(z)|^2)^{-1} = (1-z|\brho(z)|^2)^{-1}, \quad z=\zeta^2.
\end{equation*}
Hence, $1- z|\rho(z)|^2 >  c > 0$ if $|a(\zeta)| $ is bounded from above. The latter is true when in particular  $q \in H^{2,2}(\bbR)$ (see Propositions   3.1 and 3.2
 of Paper I).

In Paper I, we showed that the maps $q_0 \mapsto \rho$ and $q_0 \mapsto \brho$ are Lipschitz continuous from the soliton-free $H^{2,2}(\bbR)$ potentials $q_0$ 
into
 $H^{2,2}(\bbR)$.  We assume that the Cauchy data are soliton-free, thus  only the reflection coefficient $\rho$ is needed for the reconstruction of the
solution.

The scattering data $\rho$ and $\brho$ are not independent; as showed in Section 6
of Paper I (see the remarks at the beginning of Section 6 and Lemma 6.14),
$\brho$ can be recovered from $\rho$ by solving a scalar RHP. 
We proved in turn that, given $\rho$ {corresponding to the Cauchy data $q(x,0)$,} 
we may recover the solution $q(x,t)$ of \eqref{DNLS2}  through 
RHPs. There are two versions of  the RHP,
one to recover the solution for  $x \ge 0$ and one for  $x \le  0$. 
For example, the following RHP  provides the reconstruction formula when $x \ge 0$.
 
\begin{problem}
\label{prob:DNLS.RH0}
Given $\rho \in H^{2,2}(\bbR)$ with $1-z |\rho(z)|^2 > 0$ for all $z \in \bbR$, 
find a row vector-valued function $\bfN(z;x,t)$ on $\bbC \setminus \bbR$ with the following properties:
\begin{enumerate}
\item		$\bfN(z;x,t) \rarr (1,0)$ as $|z| \rarr \infty$,
			\medskip
			
\item		$\bfN(z;,x,t)$ is analytic for $z \in \bbC \setminus \bbR$ with continuous boundary values
			$$\bfN_\pm(z;x,t) = \lim_{\eps \darr 0} \bfN(z\pm i\eps;x,t),$$
			\medskip
			
\item		
The jump relation $\bfN_+(z;x,t) = \bfN_	-(z;x,t) V(z)$ holds, where
			\begin{equation}
			\label{DNLS.V}
			V(z)	=		\Twomat{1-z |\rho(z)|^2}{\rho(z)e^{2it \theta}}
											{-z\overline{\rho(z)}e^{-2it\theta}}{1}
			\end{equation}
and the real phase function $\theta$ is given by
\begin{equation}
\label{DNLS.phase}
\theta(z;x,t) = -\left( z\frac{x}{t} + 2z^2 \right).
\end{equation}
\end{enumerate}
\end{problem}
\bigskip\noindent
From the solution of Problem \ref{prob:DNLS.RH0}, we recover
\begin{equation}
\label{DNLS.q}
q(x,t) = \lim_{z \rarr \infty} 2iz \bfN_{12}(x,t,z)
\end{equation}
for $x \ge 0$,  where the limit is taken in $\bbC\setminus \bbR$ along any direction not tangent to $\bbR$.

\bigskip

\begin{remark}
\label{rem:DNLS.RHP0}
The jump matrix \eqref{DNLS.V} satisfies $V \in L^\infty(\bbR)$ and $\det V(z)=1$.
It follows from a standard result in RHP theory (see, for example, \cite[Theorem 2.10]{DZ03}) that Problem \eqref{prob:DNLS.RH0} may have at most one solution.

\end{remark}

\begin{remark}
The symmetry reduction from the contour $\Sigma$ to the contour $\bbR$ significantly simplifies the analysis of the 
RHP because in this setting, the phase factor $\theta$ has only one stationary point (instead of two as is the case 
in \cite{KV97}). The reason why we seek a row vector-valued solution rather than a matrix-valued 
solution is that the matrix-valued solution is not properly normalized; see
Paper I, Section 1.2 for further discussion.
\end{remark}

The central results of this paper are the following theorems
that give the long-time behavior of  the  solutions $q$
 of \eqref{DNLS2} and $u$ of \eqref{DNLS1} respectively.

\begin{theorem}
\label{thm:main1}
Suppose that $q_0 \in H^{2,2}(\bbR)$ is a soliton-free potential. In particular,   its reflection coefficient $\rho \in H^{2,2}(\bbR)$ and $c  =\inf_{z \in \bbR} \left(1-z |\rho(z)|^2 \right)>0$. Denote by  $\xi=-x/4t$ the stationary phase point of the phase function \eqref{DNLS.phase}.

\begin{enumerate}
\item[(i)]  As $t \rarr + \infty$, 
\begin{equation}
\label{result-Th1-1}
q(x,t) \sim 
	\begin{cases}
			\dfrac{1}{\sqrt{t}} \alpha_1(\xi) e^{-i\kappa(\xi) \log(8t) + ix^2/(4t)} + \bigO{t^{-3/4}}, 	&	x>0\\
			\\
			\dfrac{1}{\sqrt{t}} \alpha_2(\xi) e^{-i\kappa(\xi) \log(8t) + ix^2/(4t)} + \bigO{t^{-3/4}}, 	&	x<0
	\end{cases}
\end{equation}

\item[(ii)] As $t \rarr -\infty$, 
\begin{equation}
\label{result-Th1-2}
q(x,t) \sim
	\begin{cases}
			\dfrac{1}{\sqrt{-t}} \alpha_2(\xi) e^{i\kappa(\xi) \log(-8t) + ix^2/(4t)} + \bigO{(-t)^{3/4}}, & x >0\\
			\\
			\dfrac{1}{\sqrt{-t}} \alpha_1(\xi) e^{i\kappa(\xi) \log(-8t) + ix^2/(4t)} + \bigO{(-t)^{3/4}}, & x < 0.
	\end{cases}
\end{equation}
\end{enumerate}
Here
\begin{align}
\label{result-Th1-kappa}
\kappa(z)			&	=	-\frac{1}{2\pi}\log(1-z|\rho(z)|^2),\\
\label{result-Th1-alpha} 
|\alpha_1(\xi)|^2 &= |\alpha_2(\xi)|^2 =	\frac{\kappa(\xi)}{2\xi}.
\end{align}
%
%
For $t>0$,
\begin{align*}
\arg \alpha_1(\xi) 	&=		\frac{\pi}{4}+ \arg \Gamma(i\kappa(\xi)) + \arg \rho(\xi)\\
							&\quad	+ \frac{1}{\pi} \int_{-\infty}^\xi \log|s-\xi| \, d \log \left(1-s|\rho(s)|^2 \right),\\[5pt]
\arg \alpha_2(\xi)	&=		\arg \alpha_1(\xi) - \pi
\end{align*}
while for $t<0$,
\begin{align*}
\arg \alpha_1(\xi)	&=		-\frac{\pi}{4} - \arg \Gamma(i\kappa(\xi)) + \arg \rho(\xi)\\		
							&\quad	+ \frac{1}{\pi} \int_\xi^\infty \log|s-\xi| \, d \log \left(1-s|\rho(s)|^2 \right),\\[5pt]
\arg \alpha_2(\xi)	&=		\arg \alpha_1(\xi) + \pi.
\end{align*}
%
%
In \eqref{result-Th1-1} and \eqref{result-Th1-2}, the implied constants in the remainder terms depend only on
$\norm{\rho}{H^{2,2}(\bbR)}$ and $c>0$.
\end{theorem}

As a consequence, we get the long-time behavior of the solution $u$  to the original DNLS equation \eqref{DNLS1}.  
\begin{theorem}
\label{thm:main2}
Suppose that $u_0 \in H^{2,2}(\bbR)$ and let 
$$ q_0(x) = u_0(x) \exp\left(-i\int_{-\infty}^x |u_0(y)|^2 \, dy \right).$$
Let  $\rho$ be the reflection coefficient associated to $q_0$ by the direct scattering map 
and  $\kappa$ defined by \eqref{result-Th1-kappa}. Assume also that $c=\inf_{z \in \bbR} \left(1-z|\rho(z)|^2\right) > 0$. Denote  by $\xi = -x/4t$ the stationary 
phase point of the phase function \eqref{DNLS.phase} and fix $\xi \neq 0$. Then:
\begin{enumerate}
\item[(i)]		As $t \rarr +\infty$, 
\begin{equation}
\label{result-Th2-1}
u(x,t)	\sim
	\begin{cases}
		\dfrac{1}{\sqrt{t}} \alpha_3(\xi) e^{-i\kappa(\xi) \log(8t) + ix^2/(4t)} + \bigO[\xi]{ t^{-3/4}}, 	&	x > 0 \\
		\\
		\dfrac{1}{\sqrt{t}} \alpha_4(\xi) e^{-i\kappa(\xi) \log(8t) + ix^2/(4t)} + \bigO[\xi]{ t^{-3/4}},		&	x < 0
	\end{cases}	
\end{equation}
\item[(ii)] As $t \rarr -\infty$,
\begin{equation}
\label{result-Th2-2}
u(x,t) \sim
	\begin{cases}
		\dfrac{1}{\sqrt{-t}}\alpha_4(\xi) e^{i\kappa(\xi) \log(-8t) + ix^2/(4t)} + \bigO[\xi]{(-t)^{-3/4}}		&	x > 0 \\
		\\
		\dfrac{1}{\sqrt{-t}}\alpha_3(\xi) e^{i\kappa(\xi) \log(-8t) + ix^2/(4t)}  + \bigO[\xi]{(-t)^{-3/4}}		&	x < 0
	\end{cases}
\end{equation}
\end{enumerate}
Here,
\begin{equation}
\label{result-Th2-alpha}
|\alpha_3(\xi)|^2 	= 	|\alpha_4(\xi)|^2 = \frac{\kappa(\xi)}{2\xi} 
\end{equation}
For $t>0$,
\begin{align}
\arg \alpha_3(\xi) 	&=	\arg \alpha_1(\xi)  - \frac{1}{\pi} \int_{\xi}^{ \infty} \frac{\log(1-s|\rho(s)|^2)}{s} \, ds\\
\label{result-Th2-arg-alpha+}
\arg \alpha_4(\xi)		&=	\arg \alpha_2(\xi) - \frac{1}{\pi} \int_{\xi}^{ \infty} \frac{\log(1-s|\rho(s)|^2)}{s} \, ds.
\end{align}
while for $t<0$,
\begin{align}
\arg \alpha_3(\xi) 	&=	\arg \alpha_1(\xi)  - \frac{1}{\pi} \int_{-\infty}^{ \xi} \frac{\log(1-s|\rho(s)|^2)}{s} \, ds\\
\label{result-Th2-arg-alpha-}
\arg \alpha_4(\xi)		&=	\arg \alpha_2(\xi) - \frac{1}{\pi} \int_{-\infty}^{ \xi} \frac{\log(1-s|\rho(s)|^2)}{s} \, ds.
\end{align}
\end{theorem}

Theorem \ref{thm:main2}
is a direct consequence of  Theorem \ref{thm:main1}  and
Proposition \ref{prop:gauge+}.

%
%

\begin{remark}
Here we examine the continuity of our asymptotic formulas for $q(x,t)$ at $x=0$ by computing left- and right-hand
limits as $x \rarr 0$  for the two cases in \eqref{result-Th1-1}. A similar analysis can be made for the two cases 
in \eqref{result-Th1-2}. First, notice that the Gamma function has the property that
$$ 
\lim_{x \rarr 0^+} \arg \Gamma(ix) = -\frac{\pi}{2},
\quad
\lim_{x \rarr 0^-} \arg \Gamma(ix) =  \frac{\pi}{2}.
$$
Recalling that 
$$
\kappa(\xi) = -\frac{1}{2\pi} \log \left( 1- \xi |\rho(\xi)|^2 \right),
$$
we see that $\kappa(\xi) <0$ for $\xi < 0$ while $\kappa(\xi)>0$ for $\xi>0$. 
Since $\xi = -x/4t$,  for $x>0$ and $t>0$, $\xi < 0$,  and therefore 
$$ \lim_{x \rarr 0^+} \arg \left( \Gamma(i\kappa(\xi))\right) = \frac{\pi}{2}, $$
while for $x<0$ and $t>0$, $\xi<0$ and therefore
$$ \lim_{x\rarr 0^-} \arg \left( \Gamma(i\kappa(\xi))  \right) = -\frac{\pi}{2}. $$
This observation, and the fact that $\arg \alpha_1(\xi)$ and $\arg \alpha_2(\xi)$ differ by $\pi$, shows that the asymptotic formulas
for $q(x,t)$ in \eqref{result-Th1-1} agree in the 
respective limits $x \rarr 0^-$ and $x \rarr 0^+$. A similar argument shows that the asymptotic formulas
for $q(x,t)$ when $t<0$ and $x \rarr 0^+$ and $x \rarr 0^-$ also agree.
\end{remark}

\begin{remark}
In contrast to Theorem \ref{thm:main1}, the remainder estimates 
depend on $\xi$  as well as on $\norm{\rho}{H^{2,2}}$ and $c>0$. 
This dependence arises from Proposition 
\ref{prop:gauge+}. The error estimate is  well-behaved for $|\xi|>1$ but poorly behaved as $|\xi| \rarr 0$.
\end{remark}

\begin{remark}
Although we do not make any explicit ``small data'' assumption, we are have so far been unable  to construct  large initial data satisfying our hypotheses. 
\end{remark}

Kitaev and Vartanian
\cite{KV97}  as well as more recently Xu and Fan \cite{XuFan12}, considered the same problem for 
 Schwartz class initial data in the soliton-free sector
and obtain in the asymptotic formula \eqref{result-Th1-1}  an error term of order $(\log t)/t$ .
Our results apply to a larger class of initial data and, thanks to 
the $\dbar$-approach, arguably entail a simpler proof than earlier studies of the problem. 

 The proof of  Theorem \ref{thm:main1}  addresses separately the four cases $x  \lessgtr 0$, $t \to \pm\infty$. 
Indeed,
to reconstruct the solution $q(x,t)$, we need
to solve two different RHPs, one for $x>0$ and one for $x<0$.
The  sign of $t$ is important in the phase factors of the entries of the 
jump matrix $V$ of \eqref{DNLS.V}. Depending on the sign of $t$, 
one  performs different factorizations of 
the jump matrix $V$ in order to have the correct exponential decay on the deformed contour. 
 Finally,  a  large-time estimate of the phase factor 
$ \exp{(-i\int_{-\infty}^x |q(y,t)|^2  dy)}$
of \eqref{q.gauge} in terms of the scattering data,
 obtained in Section \ref{sec:gauge},
is  needed to obtain Theorem \ref{thm:main2}. 

As discussed earlier, the proof of Theorem \ref{thm:main1}, following \cite{BJM16,DM08},
consists of  several steps corresponding to  transformations of the initial RHP
\ref{prob:DNLS.RH0} implemented successively. For sake of clarity, we present in Section  \ref{sec:summary} 
a summary of the analysis of the various steps in each of the four cases, $x  \lessgtr 0$, $t \to \pm\infty$ and we show how the RHPs and the
respective factorizations are modified to take into account the signs of $x$ and $t$.
In the next Sections
(Sections \ref{sec:prep} to \ref{sec:large-time}), we provide the details of each step in one  case $x>0$, $t \to \infty$ as follows.

The first step,  carried out 
in  Section \ref{sec:prep}, is the conjugation of 
the row vector $\bfN$ with a scalar function $\delta(z)$   
that solves the  scalar model RHP Problem  \ref{prob:RH.delta} (see equation \eqref{N1}). This operation is standard
 and leads to a new RHP, Problem \ref{prob:DNLS.RHP1}. It is performed in order to ensure that the phase factors
in the factorization of the jump matrix \eqref{DNLS.V1}
have the correct exponential decay when the contour deformation described in Section \ref{sec:mixed} is carried out.

The second  step
(Section \ref{sec:mixed}) is a  deformation of contour from $\bbR$ to a new
contour  $\Sigma^{(2)}$ defined in \eqref{new-contour} (see Figure \ref{fig:contour}),
in such a way that   the exponential factors $e^{\pm it\theta}$ have strong decay (in time) along the rays of the contour. 
The solution has no jump along the real axis (this is important because there is no decay of the phase for large $z \in \bbR$).
This transformation induces some  `small' deviation from analyticity in the sectors $\Omega_1\cup \Omega_3 \cup \Omega_4 \cup \Omega_6$,  and leads to
a mixed $\dbar$--RHP-problem,  Problem \ref{prob:DNLS.RHP.dbar},  for a new row-vector valued function denoted $\bfN^{(2)}$.
This is where the approach of Dieng-McLaughlin \cite{DM08} differs from the steepest descent of \cite{DZ03} 
which in contrast only deals with piecewise analytic solutions. 
In the approach of \cite{DZ03}, the contour deformation is carried out by approximating the  entries of the jump matrix by rational functions which admit a direct, analytic continuation.

The  third step (Section \ref{sec:model}) is a `factorization' of  $\bfN^{(2)}$ 
in the form $\bfN^{(2)} =  \bfN^{(3)} \bfN^{\rm{PC}}$ where $\bfN^{\rm{PC}}$ is solution of a 
model RHP problem,
Problem \ref{prob:DNLS.V.RHP}, and $\bfN^{(3)}$ 
a solution of $\bar\partial$ problem, Problem \ref{prob:DNLS.dbar}.

The fourth step is the  derivation of the  explicit solution of the RHP by parabolic cylinder functions
(Section \ref{sec:model}); 
this procedure is standard but we give the key steps for the reader's convenience.

The fifth step is the solution of the $\dbar$-problem using integral equation methods. 
The $\dbar$ problem may be written as an integral equation
(equation \eqref{DNLS.dbar.int}) whose integral operator has small norm  at  large times (see  equation \eqref{dbar.int.est1})
allowing the use of Neumann series (Section \ref{sec:dbar}).

At each step  of   the analysis, one needs to estimate how the reductions modify the long-time asymptotics of the solution and carefully keep track of the 
dependency of the constants (as functions of the stationary phase point $\xi$).

The sixth step, carried out in Section \ref{sec:large-time}, consists in  regrouping the transformations to  find the behavior  of the solution of DNLS for  $x>0$ as $t\to \infty$,  using the large-$z$ behavior of the
RHP solutions.  

Finally,  the long-time behavior of the phase factor   
appearing in \eqref{q.gauge} necessary
 to obtain Theorem \ref{thm:main2}, is given in Section \ref{sec:gauge}.

The paper ends with some technical appendices.  
Appendix \ref{app:delta-asy} gives the  asymptotics of the functions $\delta_\ell$ and $\delta_r$  
which solve scalar model RHPs and are used in the first step of the reduction. 
Appendix \ref{app:4-RHPs} outlines the solution of the 
appropriate RHP's for all four cases $\pm t>0$, $\pm x>0$.
Appendix \ref{app:Phi-sol} records solution formulae important for the four model RHP's.
Appendix \ref{app:NRHP.bd} proves $L^\infty$-bounds on the solution to the model RHP.
Appendix \ref{app:figures} contains figures  illustrating
how  the different jump matrices in the sequence of 
transformations of RHPs are modified according to the four cases $\pm t>0$, $\pm x>0$.

\section{Summary of the Proof} \label{sec:summary}

As discussed above, the large-time behavior of the solution to DNLS is obtained through a sequence of transformations of RHP's.
Special attention has to be given to the signs of $x$ and $t$ as slightly different RHP's are involved depending on the signs under consideration. In Sections 
\ref{sec:prep} to \ref{sec:large-time},  we present the full calculations   of the derivation in one case $x>0, t>0$. In this Section, 
we summarize the computations without details  in the four cases $\pm t>0$, $\pm x>0$  as they are needed to get the final expressions of Theorems \ref{thm:main1} and \ref{thm:main2}.
  
The initial normalized RHPs  
that provide the reconstruction formula for the potential have contour $\bbR$ and phase function
$$ \theta(z;x,t) = -\left(z \frac{x}{t}+2z^2\right).$$ 
If $x>0$,
the initial RHP is 
\begin{subequations}
\label{RHP.right}
\begin{align}
\label{RHP.right.jump}
\bfN_+(z;x,t)	&=	\bfN_-(z;x,t) e^{it\theta \ad \sigma_3} V_0(z)\\[5pt]
\label{RHP.right.matrix}
V_0(z)			&=	\Twomat{1-z|\rho(z)|^2}{\rho(z)}{-z\overline{\rho(z)}}{1}\\[5pt]
\label{RHP.right.asy}
\bfN(z;x,t)		&=	(1,0) + \bigO{\frac{1}{z}}
\end{align}
\end{subequations}
while if $x<0$, the initial RHP is 
\begin{subequations}
\label{RHP.left}
\begin{align}
\label{RHP.left.jump}
\bfN_+(z;x,t)	&=	\bfN_-(z;x,t) e^{it\theta \ad \sigma_3} \bV_0(z)\\[5pt]
\label{RHP.left.matrix}
\bV_0(z)		&=	\Twomat{1}{\brho(z)}{-z\overline{\brho(z)}}{1-z|\brho(z)|^2}\\[5pt]
\label{RHP.left.asy}
\bfN(z;x,t)		&=	(1,0) + \bigO{\frac{1}{z}}
\end{align}
\end{subequations}
where $\brho(z) = \rho(z)/\Delta(z)$ and
\begin{equation*}
\Delta(\lam)	=	\exp\left( \frac{1}{\pi i} \,\, \mathrm{p.v.} \int_{-\infty}^\infty \frac{\kappa(s)}{ \lam-s} \, ds \right).
\end{equation*}
In both of these cases, the solution  $q(x,t)$ of  \eqref{q.gauge} is recovered from the reconstruction formula
\begin{equation}
\label{q.recon.bis.bis}
q(x,t)	=	\lim_{z \rarr \infty} \left[ 2iz \left( \bfN(z;x,t) \right)_{12} \right].
\end{equation}
The derivation of the large-time behavior is obtained through several steps.
The first steps 
\begin{enumerate}
\item[(1)]	Preparation for steepest descent
\item[(2)]	Contour deformation from $\bbR$ to $\Sigma^{(2)}$ (see Figure \ref{fig:contour})
\item[(3)]	Reduction to a model RHP
\item[(4)]	Solution to the model RHP
\end{enumerate}
have to be performed successively for each case  $\pm t>0$, $\pm x>0$ as the calculations, although similar, are
 specific to each situation.
They are followed by 
\begin{enumerate}
\item[(5)]	Analysis of $\bar \partial$ problem
\item[(6)]   Regrouping   of the transformations.
\end{enumerate}
The latter are common to all cases and detailed  in Sections \ref{sec:dbar} and \ref{sec:large-time}  for $x>0, t>0$.

 We now summarize steps 1--4.

\bigskip

\textbf{Step 1}:  We change variables in the initial RHP using the analytic functions (with branch cut either on the left or right half-line with endpoint $\xi$)
\begin{equation}
\label{delta.app}	
\delta_\ell(z;\xi)	\coloneqq	\exp\left( i \int_{-\infty}^\xi \frac{\kappa(s)}{s-z} \, ds \right), \quad z \in \bbC \setminus (-\infty,\xi]
\end{equation}
and
\begin{equation}
\label{bdelta.app}
\delta_r(z;\xi) \coloneqq	\exp\left( - i \int_\xi^\infty \frac{\kappa(s)}{s-z} \, ds \right), \quad z \in \bbC \setminus [\xi,\infty).
\end{equation}
Here
\begin{equation*}
\kappa(s) = -\frac{1}{2\pi} \log \left( 1 - s|\rho(s)|^2 \right) = -\frac{1}{2\pi} \log \left( 1 - s|\brho(s)|^2 \right) .
\end{equation*}
The functions $\delta_\ell$ and $\delta_r$ are solutions of scalar model RHPs:  $\delta_\ell$ satisfies Problem  \ref{prob:RH.delta}
and $\delta_r$ satisfies a similar one with its branch cut at the right of the endpoint $\xi$. Their properties are recalled in  Appendix \ref{app:delta-asy}.
In particular,
they  obey the bounds
$$e^{-\norm{\kappa}{\infty}/2} \leq \left| \delta^*(z) \right| \leq e^{\norm{\kappa}{\infty}/2}$$
where $\delta^*$ is $\delta_\ell^{\pm 1}$ or $\delta_r^{\pm 1}$, as easily follows from 
$$ \left| \Imag\left(\int_{\pm \infty}^\xi \frac{\kappa(s)}{s-z} \, ds \right) \right| \leq \frac{\norm{\kappa}{\infty}}{2}. $$
By defining
\begin{equation}
\label{N1.def}
\bfN^{(1)}(z;x,t) = 
	\bfN(z;x,t) \times
	\begin{cases}
		\hspace{3pt} 	\delta_\ell^{-\sigma_3} 	&	t>0, \, x>0 \\
		\hspace{3pt}  \delta_r^{-\sigma_3}		&	t>0, \, x<0 \\
		\hspace{3pt} 	\delta_r^{\sigma_3}			&	t<0, \, x>0 \\
		\hspace{3pt} 	\delta_\ell^{\sigma_3}		&	t<0, \, x<0
	\end{cases}
\end{equation}
we obtain a RHP for $\bfN^{(1)}$ with a new
jump matrix $e^{2it\theta \ad \sigma_3} V^{(1)}$. We give expressions for $V^{(1)}$ for each of the four cases $\pm t>0$, $\pm x>0$ 
 in \eqref{V1++}, \eqref{V1+-}, \eqref{V1-+}, and \eqref{V1--} respectively.
The new RHP's are `prepared' for the steepest descent method in the sense that contours can be 
deformed so that the exponential functions $e^{\pm it\theta}$ have maximum decay in $|z-\xi|$.

\bigskip

\textbf{Step 2}: 
We introduce a new unknown 
\begin{equation} \label{N2.def} 
\bfN^{(2)} = \bfN^{(1)} \calR
\end{equation}
where $\calR$ is a piecewise continuous
matrix-valued function taking the form shown in Figure \ref{fig:R.++.+-} if $t>0$, and 
 in Figure \ref{fig:R.-+.--}
if $t<0$. The purpose of the deformation is to remove the jumps along the real axis and introduce jumps on the contours $\Sigma_1$, $\Sigma_2$, $\Sigma_3$, and $\Sigma_4$ corresponding to the model problem. Thus the values of the $R_i$ 
along $(-\infty,\xi)$ and $(\xi,\infty)$ are determined by the jump matrix $V^{(1)}$, while their values along the $\Sigma_i$ are determined as follows:
\begin{itemize}
\item[(1)]	Scattering data are replaced by their values at $z=\xi$ (`freezing coefficients')
\item[(2)]	Powers of $\delta$ are replaced by their asymptotic forms near $z=\xi$ (see Appendix \ref{app:delta-asy}, equations \eqref{delta.xi.asy.+}, \eqref{bdelta.xi.asy+}, \eqref{delta.xi.asy.-}, \eqref{bdelta.xi.asy.-}).
\end{itemize}
The expressions of the  matrix $\calR$ in each of the four cases are given respectively in  \eqref{R.++}, 
\eqref{R.+-}, \eqref{R.-+}, and \eqref{R.--}, noting that the symbols $\delta$, $\delta_0$, and $\delta_\pm$ are 
defined at the beginning of each subsection and 
\emph{have different meanings in each of them} as indicated in  \eqref{delta.++.def}, \eqref{delta.+-.def}, 
\eqref{delta.-+.def}, and \eqref{delta.--.def}.

The new unknown $\bfN^{(2)}$ has a jump matrix which is most easily described by introducing the scaled variable 
\begin{equation}
\label{z.to.zeta}
\zeta(z)=\sqrt{8|t|}(z-\xi).
\end{equation}
We then have
\begin{equation}
\label{V2}
V^{(2)} = \begin{cases}
					\zeta^{i \kappa \ad\sigma_3}e^{-\frac{i}{4}\zeta^2 \ad \sigma_3} 	V_{0}^{(2)}(\zeta;\xi) 	&	\pm x > 0,  \, t > 0\\
					\\
					\zeta^{-i\kappa \ad\sigma_3}e^{\frac{i}{4}\zeta^2 \ad \sigma_3} 	V_0^{(2)}	(\zeta;\xi)	&	\pm x > 0, \, t <0
				\end{cases}
\end{equation}
In the above expression, the complex powers are defined by choosing the branch of the logarithm with $-\pi< \arg \zeta < \pi$ in the cases $t>0, \, x>0$ and $t<0$, $x<0$, and the branch of the logarithm with $0 < \arg \zeta < 2\pi$ in the cases
$t>0$, $x<0$ and $t<0$, $x>0$.
The matrices $V^{(2)}_0(\zeta;\xi)$ for each of the four cases are shown in Figures \ref{fig:V20++}, \ref{fig:V20+-}, \ref{fig:V20-+},
and \ref{fig:V20--}. The branch cut for the logarithm is also indicated.
Because $\calR$ is not a holomorphic function, the new unknown $\bfN^{(2)}$ obeys
a mixed $\dbar$-RHP. 

\bigskip

\textbf{Step 3}:  Suppose that $\bfN^{\PC}$ solves the pure  
 RHP with jump matrix $V^{(2)}$. By factoring
 \begin{equation} \label{N3.def}
  \bfN^{(2)} = \bfN^{(3)} \bfN^{\PC}, 
  \end{equation}
  we see that 
 $\bfN^{(3)}$ solves the $\dbar$ problem
(in the $z$-variable)
\begin{align*}
\dbar \bfN^{(3)}(z;x,t)	&= 	\bfN^{(3)}(z;x,t) W(z;x,t)	\\
W(z;x,t)				&=	\bfN^{\PC}(\zeta;\xi) (\dbar \calR)(z;x,t) \bfN^{\PC}(\zeta;\xi)^{-1} \\
\bfN^{(3)}			&=	(1,0) + \bigO{\frac{1}{z}}
\end{align*}
which is equivalent to the integral equation
$$ \bfN^{(3)}(z;x,t) = (1,0) + \frac{1}{\pi} \int_{\bbC} \frac{1}{z-z'} \bfN^{(3)}(z';x,t) W(z',x,t) \, dz'. $$
It can be shown 
(see Proposition \ref{prop:N3.est})
that
$$ \bfN^{(3)}(z;x,t) = (1,0) + \frac{1}{z}\bfN^{(3)}_1(x,t) + o_{\xi,t}\left(\frac{1}{z}\right) $$
where
$$ \left| \bfN^{(3)}_1(x,t)\right| \lesssim t^{-3/4}. $$
This estimate shows that the leading asymptotics of $q(x,t)$, as computed from \eqref{q.recon.bis.bis}, 
will be determined by the solution $\bfN^{\PC}$ of the model Riemann-Hilbert problem.

\bigskip

\textbf{Step 4}: It remains to solve the model RHP for $\bfN^{\PC}$. 
It has contour $\Sigma^{(2)}_0$ (centered at $\zeta=0$ in the new variables) 
and the solution has the form
\begin{align*}
\bfN^{\PC}_+(\zeta;\xi) 	&=	\bfN^{\PC}_-(\zeta;\xi) V^{(2)}(\zeta;\xi)\\
\bfN^\PC(z;\xi) 				&\sim	I + \frac{m^{(0)}}{\zeta} + \littleO{\frac{1}{\zeta}} \text{ in } \bbC \setminus \Sigma^{(2)}_0
\end{align*}
where $V^{(2)}$ is given by \eqref{V2}.
This problem can be solved in a standard way using parabolic cylinder functions 
(see, for example, \cite{DIZ93,DZ93,DZ94,Its81}). 
We factor
\begin{equation}
\label{RHP.Model.factor}
\bfN^{\PC}(\zeta;\xi)	=	\begin{cases}
											\Phi(\zeta;\xi) P(\xi) 
											e^{\frac{i}{4}\zeta^2 \sigma_3} \zeta^{-i\kappa \sigma_3}
											&	t >0
											\\
											\\
											\Phi(\zeta;\xi) P(\xi) 
											e^{-\frac{i}{4}\zeta^2 \sigma_3} \zeta^{i\kappa \sigma_3}
											&	t < 0.
									\end{cases}
\end{equation}
The constant matrix $P(\xi)$ is derived from $V^{(2)}_0$ as shown in Figure \ref{fig:V.to.P}; for $i=1,2,3,4$, $V_i$ denotes the restriction of $V^{(2)}$ to $\Sigma_i$. This factorization introduces a new unknown, 
$\Phi(\zeta;\xi)$, which obeys an RHP with contour $\bbR$ and constant jump matrix. In case $t>0$, we have
\begin{equation}
\label{RHP.Phi.+}
\begin{aligned}
\Phi_+(\zeta;\xi)	&=		\Phi_-(\zeta;\xi) V^{(0)} \\[5pt]
V^{(0)}				&=		\Twomat{1-\xi|r_\xi|^2}{r_\xi}{-\xi \overline{r_\xi}}{1} \\[5pt]
\Phi(\zeta;\xi)		&\sim	e^{-\frac{i}{4}\zeta^2 \sigma_3} \zeta^{i\kappa \sigma_3} 
											\left(I + \frac{m^{(1)}}{\zeta} + \littleO{\zeta^{-1}} \right),
\end{aligned}
\end{equation}
while for $t<0$, we have
\begin{equation}
\label{RHP.Phi.-}
\begin{aligned}
\Phi_+(\zeta;\xi)	&=		\Phi_-(\zeta;\xi) \bV^{(0)} \\[5pt]
\bV^{(0)}			&=		\Twomat{1}{\br_\xi}{-\xi \overline{\br_\xi}}{1-\xi|\br_\xi|^2}\\[5pt]
\Phi(\zeta;\xi)		&\sim	e^{\frac{i}{4} \zeta^2 \sigma_3} \zeta^{-i\kappa \sigma_3}
													\left(I + \frac{ m^{(0)}}{\zeta} + \littleO{\zeta^{-1}} \right).
\end{aligned}
\end{equation}
Note that the meaning of $r_\xi$ or $\br_\xi$ is \emph{different} depending on which of the four cases
is under consideration (see equations \eqref{r.++}, \eqref{r.+-}, \eqref{r.-+}, \eqref{r.--}).

The matrix function $\Phi$ is obtained as a solution of an ODE. Differentiating the jump relation in 
\eqref{RHP.Phi.+} or \eqref{RHP.Phi.-} 
with respect to $\zeta$, one can show that
\begin{equation}
\label{Phi.DE}
\frac{d \Phi}{d\zeta} \pm i\frac{i\zeta}{2} \sigma_3 \Phi	= \beta \Phi, 	\quad \pm t >0 
\end{equation}
where 
\begin{equation}
\label{m.to.beta}
\beta = \frac{i}{2} \left[ \sigma_3, m^{(0)} \right]
\end{equation}
 or equivalently
$$ \beta_{12} = i \left( m^{(0)} \right)_{12}, \quad \beta_{21} = -i \left( m^{(0)} \right)_{21} $$ 
is unknown at this stage of the calculation.
The difference in sign between the $t>0$ and $t<0$ cases comes from the difference in the  prescribed 
factorization \eqref{RHP.Model.factor}.  
The goal is to compute $m^{(0)}$ which will determine leading asymptotics of $q(x,t)$. 

The solution of \eqref{Phi.DE} is expressed explicitly in terms parabolic cylinder functions, treating $\beta_{12}$ and $\beta_{21}$ as 
(unknown) constants. The solution formulas are given in Appendix \ref{app:Phi-sol}. One then substitutes these solutions into the appropriate jump relation \eqref{RHP.Phi.+} or \eqref{RHP.Phi.-}
in order to compute $\beta_{12}$ and hence, by \eqref{m.to.beta},  $m^{(0)}_{12}$. Indeed, one may easily deduce from the jump relation \eqref{RHP.Phi.+} that
\begin{equation}
\label{Phi.Wronski.+}
V^{(0)}_{21} 		= 	-\xi \overline{r_\xi} = \Phi_{11}^- \Phi_{21}^+ - \Phi_{21}^- \Phi_{11}^+ \
\end{equation}
for $t>0$, and similarly from the jump relation \eqref{RHP.Phi.-}, that
\begin{equation}
\label{Phi.Wronski.-}
\bV^{(0)}_{21}	= 	-\xi \overline{\br_\xi} = \Phi_{11}^- \Phi_{21}^+ - \Phi_{21}^- \Phi_{11}^+ 
\end{equation}						
for $t<0$. These Wronskians are evaluated for each of the four cases $\pm t>0$, $\pm x>0$ in Appendix \ref{app:Phi-sol},
equations \eqref{Phi.Wronski.+.comp} and \eqref{Phi.Wronski.-.comp}.
Using these results in \eqref{Phi.Wronski.+} and \eqref{Phi.Wronski.-}, we find
\begin{equation}
\label{beta12.+}
\beta_{12}	=	
\begin{cases}
		\dfrac{\sqrt{2\pi}e^{-\pi \kappa/2} e^{i\pi /4}}{-\xi \overline{r_\xi} \, \Gamma(-i\kappa)}
			&	t>0, \, x >0\\
		\\
		\dfrac{\sqrt{2\pi} e^{-\pi \kappa/2} e^{i\pi /4}}{-\xi \overline{r_\xi} \,  \Gamma(-i\kappa)} e^{-2\pi \kappa}
			& t>0, \, x<0
\end{cases} 
\end{equation}
and
\begin{equation}
\label{beta12.-}
\beta_{12} = 
\begin{cases}
	\dfrac{\sqrt{2\pi}e^{-\pi \kappa/2} e^{3\pi i /4}}{-\xi \overline{r_\xi} \, \Gamma(i\kappa)} e^{2\pi \kappa},
		&	t<0, \, x<0 \\
	\\
	\dfrac{\sqrt{2\pi}e^{-\pi \kappa/2}e^{3\pi i/4}}{-\xi \overline{\br_\xi} \,  \Gamma(i\kappa)},
		& t<0,  \, x<0
\end{cases}
\end{equation}
We recall that the values of $r_\xi$ and $\br_\xi$ differ from case to case.

We can now deduce the leading asymptotic behavior of $q(x,t)$ from the reconstruction formula
$$
q_{\mathrm{as}}(x,t)	
	=	\lim_{z \rarr \infty} 2iz \frac{\left(m^{(0)}\right)_{12}}{\zeta} 
	=	2\frac{\beta_{12}}{\sqrt{8|t|}}
$$
where we used \eqref{z.to.zeta} and \eqref{m.to.beta}. For $\pm t >0$ we find
\begin{equation}
\label{q.as.template}
q_{\mathrm{as}}(x,t) = \frac{1}{\sqrt{|t|}} \alpha(\xi) e^{\pm i\kappa(\xi) \log(8|t|)} e^{-i\frac{x^2}{4t}}
\end{equation}
with 
\begin{align}
\label{alpha.sq.bis}
|\alpha(\xi)|	^2		&=	\frac{1}{2}|\beta_{12}|^2\\
\label{alpha.arg.bis}
\arg \alpha(\xi)	&=	\arg \beta_{12}	\mp \kappa(\xi)  \log(8|t|) + x^2/4t
\end{align}
From \eqref{alpha.sq.bis}--\eqref{alpha.arg.bis}, \eqref{beta12.+}, \eqref{beta12.-}, and \eqref{q.as.template}, we can compute $q_{\mathrm{as}}(x,t)$ in each of the four cases. 
In Appendix \ref{app:4-RHPs}
we summarize the key formulae leading to $q_{\mathrm{as}}(x,t)$. 

In the next  five sections, we present the details of the proof  of Theorem \ref{thm:main1} in the case $x>0$,  $t>0$.

\section{Preparation for Steepest Descent}
\label{sec:prep}
In this section, we  provide the detailed analysis of Step 1 (as described in Section \ref{sec:summary}),  for the case $x>0$,  $t>0$.
In order to apply the method of steepest descent, we introduce a new unknown 
\begin{equation}
\label{N1}
\bfN^{(1)}(z;x,t) = \bfN(z;x,t) \delta(z)^{-\sigma_3} 
\end{equation}
where $\delta(z)=\delta_\ell(z)$ as defined in \eqref{delta.app}
and  solves 
the scalar RHP 
Problem \ref{prob:RH.delta} below. 
To state the scalar RHP, recall that the phase function \eqref{DNLS.phase} satisfies 
$$ \theta_z(x,t,z) = -\left(\frac{x}{t} +4 z\right)$$ and has a single critical point at $$\xi=-\frac{x}{4t}.$$ 

\begin{problem}
\label{prob:RH.delta}
Given $\xi \in \bbR$ and $\rho \in H^{2,2}(\bbR)$ with $1-s|\rho(s)|^2 >0$ for all $s \in \bbR$, find a scalar function 
$\delta(z) = \delta(z;\xi)$, analytic for
$z \in \bbC \setminus (-\infty,\xi]$ with the following properties:
\begin{enumerate}
\item		$\delta(z) \rarr 1$ as $z \rarr \infty$,
\item		$\delta(z)$ has continuous boundary values $\delta_\pm(z) =\lim_{\eps \darr 0} \delta(z \pm i\eps)$ for $z \in (-\infty, \xi)$,
\item		$\delta_\pm$ obey the jump relation
			$$ \delta_+(z) = \begin{cases}
											\delta_-(z)  \left(1 - z \left| \rho(z) \right|^2 \right),	&	 z\in (-\infty,\xi)\\
											\delta_-(z), &	z \in (\xi,\infty)
										\end{cases}
			$$
\end{enumerate}
\end{problem}
The following lemma is ``standard'' (see, for example, \cite[Proposition 2.12]{DZ03} or \cite[Proposition 6.1 and Lemma 6.2]{Do11}). Recall the definition \eqref{result-Th1-kappa} of $\kappa$.
\begin{lemma}
\label{lemma:delta}
Suppose $\rho \in H^{2,2}(\bbR)$ and that $\kappa(s)$ is real for all $s \in \bbR$. 
\begin{itemize}
\item[(i)]		(Existence, Uniqueness) Problem \ref{prob:RH.delta} has the unique solution
\begin{equation}
\label{RH.delta.sol}
\delta(z) = \exp	\left( 
								i \int_{-\infty}^\xi \frac{1}{s-z}\kappa(s) \, ds 
						\right).
\end{equation}
Moreover,
\begin{equation*}
\delta(z) \overline{\delta(\zbar)}=1
\end{equation*}
holds. 
The function  $\delta(z)$ satisfies the estimate
$$ e^{-\norm{\kappa}{\infty}/2} \leq |\delta(z)| \leq e^{\norm{\kappa}{\infty}/2}.  $$

\item[(iii)] (Large-$z$ asymptotics) 
 It admits a large-$|z|$ asymptotic expansion
$$ \delta(z)  = 1 + \frac{i}{z} 
				\int_{-\infty}^\xi 
						\kappa(s)  \, ds + \bigO{\frac{1}{z^2}}.$$
												
\item[(iv)] (Asymptotics as $z \rarr \xi$ along a ray in $\bbC \setminus \bbR$)
				Along any ray of the form $\xi + e^{i\phi}\bbR^+$ with $0<\phi<\pi$ or $\pi < \phi < 2\pi$, 
				
				$$ 
						 \left| \delta(z) - \delta_0(\xi) (z-\xi)^{i\kappa(\xi)} \right| 
						 		\lesssim_{\, \rho, \phi}
						-\left| z-\xi \right| \log |z -\xi|.$$
				The implied constant depends on $\rho$ through its $H^{2,2}(\bbR)$-norm  
				and is independent of $\xi \in \bbR$.
				Here
				$ \delta_0(\xi) = e^{i\beta(\xi,\xi)}$ and 
				$$ \beta(z,\xi) = - \kappa(\xi) \log(z- \xi+1) + \int_{-\infty}^\xi \frac{\kappa(s) - \chi(s)\kappa(\xi)}{s-z} \, ds, $$
				where $\chi$ is the characteristic function of the interval 
				 $(\xi-1,\xi)$.
				We choose the branch of the logarithm with $-\pi  < \arg(z) < \pi$.  
\end{itemize}
\end{lemma}

\begin{proof}
The proofs of these properties are   similar, for example, to  proofs 
given in  \cite[Section 2]{DZ03}.  We provide some details  for the reader's convenience.

(i) Existence follows from the explicit formula  \eqref{RH.delta.sol}. Since $\rho$ is $C^1$, uniqueness follows from Liouville's theorem.

(ii) These estimates  are obtained from the observation that 
$$ 
\left| \Real \left( i \int_{-\infty}^\xi \frac{\kappa(s)}{s-z} \, ds \right) \right| \leq \frac{\norm{\kappa}{\infty}}{2}. 
$$
(iii) and (iv) are proved 
 in  Appendix \ref{app:delta-asy}.
\end{proof}

If $\bfN(z;x,t)$ solves Problem \ref{prob:DNLS.RH0} and $\delta(z)$ solves Problem \ref{prob:RH.delta}, then the row vector-valued function $\bfN^{(1)}(z;x,t)$ defined in 
\eqref{N1} solves the following RHP.  

\begin{problem}
\label{prob:DNLS.RHP1}
Given $\rho \in H^{2,2}(\bbR)$ with $1-z |\rho(z)|^2 > 0$ for all $z \in \bbR$, find a row vector-valued function $\bfN^{(1)}(z;x,t)$ on $\bbC \setminus \bbR$ with the following properties:
\begin{enumerate}
\item		$\bfN^{(1)}(z;x,t) \rarr (1,0)$ as $|z| \rarr \infty$,
\item		$\bfN^{(1)}(z;x,t)$ is analytic for $z \in  \bbC \setminus \bbR$
			with continuous boundary values
			$$\bfN^{(1)}_\pm(z;x,t) 
				= \lim_{\eps \darr 0} \bfN^{(1)}(z+i\eps;x,t)$$
\item		The jump relation $$\bfN^{(1)}_+(z;x,t)=\bfN^{(1)}_-(z;x,t)	
			V^{(1)}(z)$$
			 holds,
			 where $$V^{(1)}(z) = \delta_-(z)^{\sigma_3} V(z) \delta_+(z)^{-\sigma_3}.$$
			 \noindent
			The jump matrix $V^{(1)} $ is factorized as 
			\begin{align}
			\label{DNLS.V1}
			V^{(1)}(z)	=
			\begin{cases}
					\Twomat{1}{0}{-\dfrac{\delta_-^{-2} z \rhobar}{1-z |\rho|^2}  e^{-2it\theta}}{1}
					\Twomat{1}{\dfrac{\delta_+^2 \rho}{1-z |\rho|^2} e^{2it\theta}}{0}{1},
						& z \in (-\infty,\xi),\\
						\\
						\Twomat{1}{\rho \delta^2 e^{2it\theta}}{0}{1}
					\Twomat{1}{0}{-z\rhobar \delta^{-2} e^{-2it\theta}}{1},
						& z \in (\xi,\infty) .
			\end{cases}
			\end{align}
		
\end{enumerate}
\end{problem}

\begin{remark}
The jump matrix $V^{(1)}$ for Problem \ref{prob:DNLS.RHP1} has determinant $1$.
A standard argument (see  Remark \ref{rem:DNLS.RHP0}) 
 shows that Problem \ref{prob:DNLS.RHP1}
has at most one solution.
\end{remark}

\section{Deformation to a Mixed $\dbar$-Riemann-Hilbert Problem}
\label{sec:mixed}

We now seek to deform Problem \ref{prob:DNLS.RHP1} using the  
method of  Dieng and McLaughlin \cite{DM08} and Borghese, Jenkins and McLaughlin \cite{BJM16}. 
The phase function \eqref{DNLS.phase} has a single critical point
at $\xi = -x/4t$. The new contour 
\begin{equation}
\label{new-contour}
\Sigma^{(2)} = \Sigma_1 \cup \Sigma_2 \cup \Sigma_3 \cup \Sigma_4
\end{equation}
 is shown in Figure \ref{fig:contour} and consists of oriented half-lines $\xi + e^{i\phi}\bbR^+$
where $\phi = \pi/4, 3\pi/4,5\pi/4, 7\pi/4$. 

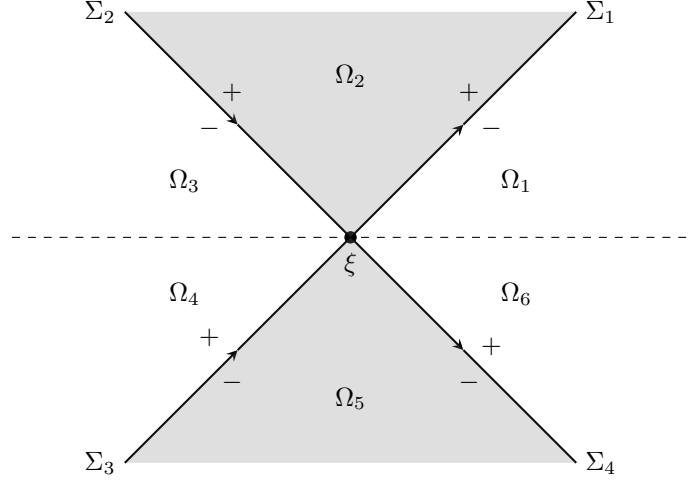
\begin{figure}[h!]
\caption{Deformation from $\mathbb{R}$ to $\Sigma^{(2)}$}
\vskip 15pt
\begin{tikzpicture}[scale=0.75]
\draw[dashed] 					(-6,0) -- (6,0);								
\draw[->,thick,>=stealth] 		(0,0) -- (2,2);								
\draw[thick]						(2,2) -- (4,4);
\draw[->,thick,>=stealth] 		(-4,4) -- (-2,2);							
\draw[thick] 						(-2,2) -- (0,0);
\draw[->,thick,>=stealth]		(-4,-4) -- (-2,-2);							
\draw[thick]						(-2,-2) -- (0,0);
\draw[thick,->,>=stealth]		(0,0) -- (2,-2);								
\draw[thick]						(2,-2) -- (4,-4);
\draw	[fill]							(0,0)		circle[radius=0.1];		
\path[fill=gray,opacity=0.25]	(0,0) -- (4,4) -- (-4,4) -- (0,0);
\path[fill=gray,opacity=0.25] (0,0) -- (-4,-4) -- (4,-4) -- (0,0);
\node[below] at (0,-0.1)			{$\xi$};
\node[right] at (4,4)					{$\Sigma_1$};
\node[left] at (-4,4)					{$\Sigma_2$};
\node[left] at (-4,-4)					{$\Sigma_3$};
\node[right] at (4,-4)				{$\Sigma_4$};
\node[right] at (2.5,1)				{$\Omega_1$};
\node[above] at (0,2.5)			{$\Omega_2$};
\node[left] at (-2.5,1)				{$\Omega_3$};
\node[left] at (-2.5,-1)				{$\Omega_4$};
\node[below] at (0,-2.5)			{$\Omega_5$};
\node[right] at (2.5,-1)				{$\Omega_6$};
\node[above] at (2.1,2.25)		{$+$};
\node[below] at (2.5,2.25)		{$-$};
\node[above] at (2.5,-2.25)		{$+$};
\node[below] at (2.1,-2.25)		{$-$};
\node[above] at (-2.1,2.25)		{$+$};
\node[below] at (-2.5,2.25)		{$-$};
\node[above] at (-2.5,-2.1)		{$+$};
\node[below] at (-2.1,-2.25)		{$-$};
\end{tikzpicture}
\label{fig:contour}
\end{figure}

In order to deform the contour $\bbR$ to the contour $\Sigma^{(2)}$, we 
 introduce a new unknown $\bfN^{(2)}$ obtained from $\bfN^{(1)}$ as
$$ \bfN^{(2)}(z) = \bfN^{(1)}(z)  \calR^{(2)}(z). $$
We  choose $\calR^{(2)}$ to remove the jump on the real axis and provide analytic jump matrices with the correct decay properties 
on the contour $\Sigma^{(2)}$. We have
$$
\bfN^{(2)}_+	=\bfN^{(1)}_+ \calR^{(2)}_+ 
				= \bfN^{(1)}_- V^{(1)} \calR^{(2)}_+ 
				= \bfN^{(2)}_- \left(\calR^{(2)}_-\right)^{-1}
						 V^{(1)} \calR^{(2)}_+
$$
so the jump matrix will be the identity matrix on $\bbR$ provided 
$$ 
(\calR^{(2)}_-)^{-1} V^{(1)} \calR^{(2)}_+ = I
$$
where $\calR_\pm^{(2)}$ are the boundary values of $\calR^{(2)}(z)$ as $\pm \Imag(z) \darr 0$. 
On the other hand, the function $e^{2it\theta}$ is exponentially increasing on $\Sigma_1$ and $\Sigma_3$, and decreasing on $\Sigma_2$ and $\Sigma_4$, while the reverse is true of $e^{-2it\theta}$. 
Hence, we choose $\calR^{(2)}$ as shown in Figure \ref{fig:R.++.+-}, where,
letting
\begin{equation} \label{eta}
\eta(z;\xi) = (z-\xi)^{i\kappa(\xi)}, 
\end{equation}
 the functions $R_1$, $R_3$, $R_4$, 
and $R_6$ satisfy 
\begin{align}
\label{R1}
R_1(z)	&=	\begin{cases}
						z\overline{\rho(z)} \delta^{-2},			
								&	z \in (\xi,\infty)\\[10pt]
						\xi \overline{\rho(\xi)} \delta_0(\xi)^{-2}\eta(z;\xi)^{-2},
								&	z	\in \Sigma_1
					\end{cases}\\[10pt]
\label{R3}
R_3(z)	&=	\begin{cases}
						-\dfrac{\delta_+^2(z)\rho(z)}{1-z|\rho(z)|^2},		
								& z \in (-\infty,\xi)\\[10pt]
					\mathrlap{-\dfrac{\delta_0^{2} \eta(z;\xi)^2 \rho(\xi)}{1-\xi|\rho(\xi)|^2},}
						\hphantom{\xi \overline{\rho(\xi)} \delta_0(\xi)^{-2}\eta(z;\xi)^{-2},}	
								& z \in \Sigma_2
					\end{cases}
					\\[10pt]
\label{R4}
R_4(z)	&=	\begin{cases}
						\mathrlap{- \dfrac{ z \overline{\rho(z)}\delta_-^{-2}}{1-z|\rho(z)|^2},}
						\hphantom{\xi \overline{\rho(\xi)} \delta_0(\xi)^{-2}\eta(z;\xi)^{-2},}	
								&	z \in (-\infty,\xi)\\[10pt]
					- \dfrac{\delta_0^{-2} \eta(z;\xi)^{-2}  \xi \overline{\rho(\xi)}}{1-\xi|\rho(\xi)|^2},
								&	z \in \Sigma_3							
					\end{cases}
					\\[10pt]
\label{R6}
R_6(z)	&=	\begin{cases}
						\mathrlap{\rho(z) \delta(z)^2 }
						\hphantom{\xi \overline{\rho(\xi)} \delta_0(\xi)^{-2}\eta(z;\xi)^{-2},}	
								&	 z  \in (\xi,\infty)\\[10pt]
					\rho(\xi)  \delta_0(\xi)^{2}\eta(z;\xi)^{2},
								&	 z \in \Sigma_4
					\end{cases}
\end{align}
The idea is to construct $R_i(z)$ in $\Omega_i$ to have the prescribed boundary values and $\dbar R_i(z)$ small in the sector. 
This will allow us to reformulate Problem \ref{prob:DNLS.RHP1} as a mixed RHP-$\dbar$ problem. We will  show how to remove the RHP component through an explicit model problem and then formulate a $\dbar$ problem for which the large-time contribution to the asymptotics of $q(x,t)$ is negligible.
Note that the values of $R_i(z)$ on the contours $\Sigma_i$ localize the scattering data to 
the stationary phase point $\xi$. This localization
corresponds to the localization of the weights  in the steepest descent method  \cite{DZ03}. The latter  
requires a delicate  analysis of modified Beals-Coifman resolvents that is  greatly simplified in the current approach.

  The following lemma and its proof are almost identical to  \cite[Lemma 4.1]{BJM16} or 
\cite[Proposition 2.1]{DM08}.   It is useful in the estimates of  the contribution of the solution of the 
$\bar \partial$-problem for large time (Section \ref{sec:dbar}).
To state it,
 we introduce the factors 
\begin{align*}
p_1(z)	&=	z \overline{\rho(z)},&
p_3(z)	&=	-\frac{\rho(z)}{1-z |\rho(z)|^2}, \\
p_4(z)	&= -\dfrac{z \overline{	\rho(z)}}{1 -z|\rho(z)|^2},	&
p_6(z)	&=	\rho(z).
\end{align*}
that appear in \eqref{R1}--\eqref{R6}.
\begin{lemma}
\label{lemma:dbar.Ri}
Suppose $\rho \in H^{2,2}(\bbR)$. There exist functions $R_i$ on $\Omega_i$, $i=1,3,4,6$ satisfying \eqref{R1}--\eqref{R6}, so that
$$ 
|\dbar R_i(z)| \lesssim
	\begin{cases}
			\left( |p_i'(\Real(z))| - \log|z-\xi| \right), 	
			&	z \in \Omega_i, \quad  |z-\xi| \leq 1\\
			\\
			\left( |p_i'(\Real(z))+ |z-\xi|^{-1}	\right),	
			&	z \in \Omega_i, \quad |z-\xi| >1,
	\end{cases}
$$ 
where the implied constants are uniform in  $\xi \in \bbR$ and  $\rho $ in a fixed bounded subset of $H^{2,2}(\bbR)$ with $1-z|\rho(z)|^2 \geq c>0$ for a fixed constant $c$.
\end{lemma}

\begin{remark}
\label{rem:dbar.Ri}
By adjusting numerical constants, we can rewrite the estimate on $\dbar R_i$ for $|z-\xi| > 1$ as 
$$ \left| \dbar R_i \right| \lesssim |p_i'(\Real(z))| + (1+|z-\xi|^2)^{-1/2}. $$
\end{remark}

\begin{proof}
We  give the construction for $R_1$. Define $f_1(z)$ on $\Omega_1$ by
$$ f_1(z) = p_1(\xi) \delta_0^{-2} (\xi) \eta(z;\xi)^{-2} \delta(z)^{2} $$
and let
 $$ R_1(z) = \left( f_1(z) + \left[ p_1(\Real(z)) - f_1(z) \right] \cos 2\phi \right) \delta(z)^{-2} $$
where $\phi = \arg (z-\xi)$. It is easy to see that $R_1$ as constructed has the boundary values \eqref{R1}.
Writing $z-\xi = r e^{i\phi}$ we have
$$ \dbar = \frac{1}{2}\left( \frac{\dee}{\dee x} + i \frac{\dee}{\dee y} \right)
			=	\frac{1}{2} e^{i\phi} \left( \frac{\dee}{\dee r} + \frac{i}{r} \frac{\dee}{\dee \phi} \right).
$$
We therefore have 
$$ 
\dbar R_1 (z) =  \frac{1}{2}  p_1'(\Real z) \cos 2\phi  ~ \delta(z)^{-2} -
		\left[ p_1(\Real z) - f_1(z) \right]\delta(z)^{-2}  \frac{ie^{i\phi}}{|z-\xi|}  \sin 2\phi. 
$$
It follows from Lemma \ref{lemma:delta}(iv) that
$$ 
 \left|\left( \dbar R_1 \right)(z)  \right| \lesssim_{\, \rho}
\begin{cases}
	|p_1'(\Real z)| - \log|z-\xi|	,			&	|z-\xi| \leq 1,\\
	\\
	|p_1'(\Real z)| + \dfrac{1}{|z-\xi|},		&	|z-\xi|  > 1,
\end{cases}
$$
where the implied constants depend on $\inf_{z \in \bbR} (1-z |\rho(z)|^2)$ and $\norm{\rho}{H^{2,2}}$. 
The remaining constructions are similar.
\end{proof}

The unknown $\bfN^{(2)}$ satisfies a mixed $\dbar$-RHP. We first compute the jumps of $\bfN^{(2)}$ along the contour $\Sigma^{(2)}$ with the given orientation, remembering that $\bfN^{(1)}$ is analytic there so that the jumps are determined entirely by the change of variables.
Diagrammatically, the jump matrices are as in Figure \ref{fig:jumps}. Away from $\Sigma^{(2)}$ we   have 
\begin{equation}
\label{N2.dbar}
 \dbar \bfN^{(2)} = \bfN^{(2)} \left( \calR^{(2)} \right)^{-1} \dbar \calR^{(2)} = \bfN^{(2)} \dbar \calR^{(2)} 
 \end{equation}
where the last step follows by triangularity.

\begin{problem}
\label{prob:DNLS.RHP.dbar}
Given $\rho \in H^{2,2}(\bbR)$ with $1-z|\rho(z)|^2 > 0$ for all $z \in \bbR$, find a row vector-valued function $\bfN^{(2)}(z;x,t)$ on $\bbC \setminus \bbR$ with the following properties:
\begin{enumerate}
\item		$\bfN^{(2)}(z;x,t) \rarr (1,0)$ as $|z| \rarr \infty$ in $\bbC \setminus \Sigma^{(2)}$,
\item		$\bfN^{(2)}(z;x,t)$ is continuous for $z \in  \bbC \setminus \Sigma^{(2)}$
			with continuous boundary values 
			$\bfN^{(2)}_\pm(z;x,t) $
			(where $\pm$ is defined by the orientation in Figure \ref{fig:contour})
\item		The jump relation $\bfN^{(2)}_+(z;x,t)=\bfN^{(2)}_-(z;x,t)	
			V^{(2)}(z)$ holds, where
			$V^{(2)}(z)	$ is given in Figure \ref{fig:jumps},
\item		The equation 
			$$
			\dbar \bfN^{(2)} = \bfN^{(2)} \, \dbar \calR^{(2)}
			$$ 
			holds in $\bbC \setminus \Sigma^{(2)}$, where
			$$
			\dbar \calR^{(2)}=
				\begin{doublecases}
					\Twomat{0}{0}{(\dbar R_1) e^{-2it\theta}}{0}, 	& z \in \Omega_1	&&
					\Twomat{0}{(\dbar R_3)e^{2it\theta}}{0}{0}	,	& z \in \Omega_3	\\
					\\
					\Twomat{0}{0}{(\dbar R_4)e^{-2it\theta}}{0},	&	z \in \Omega_4	&&
					\Twomat{0}{(\dbar R_6)e^{2it\theta}}{0}{0}	,	&	z	\in \Omega_6	 \\
					\\
					0	&\hspace{-5pt} \text{ otherwise.}	
				\end{doublecases}
			$$
\end{enumerate}
\end{problem}

%
%


\begin{figure}[h!]
\caption{Jump Matrices  $V^{(2)}$  for $\textbf{N}^{(2)}$}
\vskip 15pt
\begin{tikzpicture}[scale=0.8]
\draw[dashed] 				(-6,0) -- (6,0);							
\draw[->,thick,>=stealth] 	(0,0) -- (1.5,1.5);						
\draw[thick]					(1.5,1.5) -- (3,3);
\draw[->,thick,>=stealth]	(-3,3) -- (-1.5,1.5);					
\draw[thick]					(-1.5,1.5) -- (0,0);
\draw[->,thick,>=stealth]	(-3,-3) -- (-1.5,-1.5);					
\draw[thick]					(-1.5,-1.5) -- (0,0);
\draw[->,thick,>=stealth]	(0,0) -- (1.5,-1.5);					
\draw[thick]					(1.5,-1.5) -- (3,-3);
\draw[fill]						(0,0)	circle[radius=0.075];		
\node [below] at  			(0,-0.15)		{$\xi$};
\node[right] at					(3.2,3)		{$\unitlower{-R_1 e^{-it\theta}}$};
\node[left] at					(-3.2,3)		{$\unitupper{-R_3 e^{it\theta}}$};
\node[left] at					(-3.2,-3)		{$\unitlower{R_4 e^{-2it\theta}}$};
\node[right] at					(3.2,-3)		{$\unitupper{R_6 e^{-2it\theta}}$};
\node[above] at 				(1.4,1.6)	{$+$};
\node[below] at				(1.7,1.4)	{$-$};
\node[above] at				(-1.4,1.6)	{$+$};
\node[below] at				(-1.7,1.4)	{$-$};
\node[above] at 				(-1.7,-1.4)	{$+$};
\node[below] at				(-1.4,-1.6)	{$-$};
\node[above] at				(1.7,-1.4)	{$+$};
\node[below] at				(1.4,-1.6)	{$-$};
\node[left] at					(2.5,3)		{$\Sigma_1$};
\node[right] at					(-2.5,3)		{$\Sigma_2$};
\node[right] at					(-2.5,-3)		{$\Sigma_3$};
\node[left] at					(2.5,-3)		{$\Sigma_4$};
\end{tikzpicture}
\label{fig:jumps}
\end{figure}

%
%

\section{The Model Riemann-Hilbert Problem}
\label{sec:model}
 The next step is to extract from $\bfN^{(2)}$  a contribution that is a pure RHP. We write 
\begin{equation*}
\bfN^{(2)} = \bfN^{(3)} \bfN^{\RHP}
\end{equation*}
and we request that $\bfN^{(3)} $ has no jump. Thus we look for $ \bfN^{\RHP}$ solution of the model  RHP \ref{DNLS.RHP.Model} below
with the jump matrix $V^\RHP=V^{(2)}$.   Unlike the previous RHP's, we seek a matrix-valued solution.

In the following RHPs (Problems \ref{DNLS.RHP.Model}, \ref{prob:DNLS.V.RHP}, \ref{prob-model}),
 $\xi$ is fixed, and we assume  that $1-\xi |\rho(\xi)|^2 > 0$. This is a spectral  condition,  automatically satisfied if $\xi>0$  
 (i.e.  if $x$ and $t$ have the same sign),
 but  imposed on  the spectral data $\rho$, to address the cases  where $x$ and $t$ have opposite signs.
\begin{problem}
\label{DNLS.RHP.Model}
Find a $2\times 2$ matrix-valued function $\bfN^\RHP(z;\xi)$, analytic on $\bbC \setminus \Sigma^{(2)}$,
with the following properties:
\begin{enumerate}
\item	$\bfN^\RHP(z;\xi) \rarr I$ as $|z| \rarr \infty$ in $\bbC \setminus \Sigma^{(2)}$, where $I$ is the $2\times 2$ identity matrix,
\item	$\bfN^\RHP(z;\xi)$ is analytic for $z \in \bbC \setminus \Sigma^{(2)}$ with continuous boundary values $\bfN^\RHP_\pm$
		on $\Sigma^{(2)}$,
\item	The jump relation $\bfN^\RHP_+(z;\xi) = \bfN^\RHP_-(z;\xi) V^\RHP(z)$ holds on $\Sigma^{(2)}$, where
		\begin{equation*}	
		V^\RHP(z) =	V^{(2)}(z).
		\end{equation*}
\end{enumerate}
\end{problem}
Now set 
\begin{equation} \label{zeta-z}
\zeta(z)=\sqrt{8t}(z-\xi)
\end{equation}
and 
\begin{equation}\label{r-xi}
r_{\xi}=\rho(\xi)\delta_0^2 e^{-2i\kappa(\xi)\log\sqrt{8t}}e^{4it\xi^2}.
\end{equation}
Under the change of variables \eqref{zeta-z}, the phase $e^{2it\theta}$ identifies to $e^{-i\zeta^2/2} e^{ix^2/4t}$. The factor $e^{-i\zeta^2/2}$ will be later important in the identification of parabolic  cylinder functions.

By abuse of notation, set $\bfN^\RHP( \zeta(z) ;\xi) =\bfN^\RHP(z;\xi)$ where $\zeta$ is given by \eqref{zeta-z}. We can then recast Problem \ref{DNLS.RHP.Model} as follows.

\begin{problem}
\label{prob:DNLS.V.RHP}
Find a $2\times 2$ matrix-valued function $\bfN^\RHP( \zeta(z);\xi)$, analytic on $\bbC \setminus \Sigma^{(2)}$,
with the following properties:
\begin{enumerate}
\item	$\bfN^\RHP(\zeta(z);\xi) \rarr I$ as $|z| \rarr \infty$ in $\bbC \setminus \Sigma^{(2)}$, where $I$ is the $2\times 2$ identity matrix,
\item	$\bfN^\RHP( \zeta(z);\xi)$ is analytic for $z \in \bbC \setminus \Sigma^{(2)}$ with continuous boundary values $\bfN^\RHP_\pm$
		on $\Sigma^{(2)}$,
\item	The jump relation 
		 $\bfN^\RHP_+( \zeta(z) ;\xi) = \bfN^\RHP_-(\zeta(z) ;\xi) V^\RHP(\zeta(z) ;\xi)$ holds on $\Sigma^{(2)}$, where
\begin{equation*}	
V^\RHP( \zeta(z);\xi) =
			\begin{cases} 
				\Twomat{1}{0}
							{-\xi \overline{r_\xi} \,\zeta^{-2i\kappa(\xi)}e^{i\zeta^2/2}}
							{1},					
				&	z 	\in \Sigma_1,\\
				\\
				\Twomat{1}
							{\dfrac{r_\xi}{1-\xi|r_\xi|^2}\, \zeta^{2i\kappa(\xi)}e^{-i\zeta^2/2}}
							{0}{1},	
				&	z	\in \Sigma_2\\
				\\
				\Twomat{1}{0}
							{ \dfrac{ -\xi \overline{r_\xi}}{1-\xi|r_\xi|^2}\, \zeta^{-2i\kappa(\xi)}e^{i\zeta^2/2}}
							{1},					
				&	z 	\in \Sigma_3,\\
				\\
				\Twomat{1}
				{r_\xi \,\zeta^{2i\kappa(\xi)}e^{-i\zeta^2/2}}
						{0}{1},		
				&	z 	\in \Sigma_4.
			\end{cases}
\end{equation*}

\end{enumerate}

\end{problem}

It is possible to further reduce the RHP for $\bfN^\RHP(\zeta;\xi) $ to a model RHP whose $2 \times 2 $ 
matrix solution is piecewise analytic in the upper and 
lower complex plane.  In each half-plane, the entries of the matrix satisfy ODEs that are obtained from analyticity properties as
well as the large-$\zeta$ behavior. The  solutions  of the ODEs are explicitly calculated in terms of 
parabolic cylinder functions.
This transformation is standard
and has been performed for NLS and mKdV (see, for example, \cite{DIZ93,DZ93,DZ94,Its81}).
Let 
\begin{equation}\label{tildeN}
 \bfN ^\RHP (\zeta;\xi) =  \Phi(\zeta;\xi)  \mathcal{P}(\xi) e^{\frac{i}{4} \zeta^2 \sigma_3} \zeta^{-i\kappa \sigma_3}. 
\end{equation}
where 
\begin{equation}
\label{P-matrix}
\mathcal{P}(\xi) =
\begin{doublecases}
\Twomat{1}{0}
						{\xi \, \overline{r_\xi}}	{1},					
				&	z 	\in \Omega_1 &&
				\Twomat{1}
							{\dfrac{-r_\xi}{1-\xi|r_\xi|^2}\,  }{0}
							{1},		
				&	z	\in \Omega_3,\\
				\\
				\Twomat{1}{0}
							{\dfrac{- \xi \overline{r_\xi}}{1-\xi|r_\xi|^2}\, }
							{1},					
				&	z 	\in \Omega_4, &&
				\Twomat{1}{r_\xi}
						{0}
							{1},		
				&	z 	\in \Omega_6, \\ \\
                                     \Twomat{1}
							{0}
							{0}
							{1},	
				\qquad	z	\in \Omega_2 \cup \Omega_5.  \hspace{-2cm} \\	
\end{doublecases}
\end{equation}
By construction, the matrix $\Phi$ is continuous along the rays of $\Sigma^{(2)}$. Let us set up the RHP it satisfies and 
 compute
its jumps along the real axis. We have along the real axis 
\begin{equation}
\Phi_+  =      ~\Phi_-\Big( \mathcal{P}  \,e^{i\sigma_3 \zeta^2/4} \zeta^{-i\kappa(\xi)\sigma_3} \Big)_-   \Big(\,e^{-i\sigma_3\zeta^2/4}\zeta^{i\kappa(\xi)\sigma_3} \mathcal{P}^{-1}  \Big)_+
\end{equation}
 Due to the branch cut of the logarithmic function along 
$\bbR^-$, we have
 along the negative real axis,
 \begin{align*}
 ( \zeta ^{-i \kappa(\xi) \sigma_3})_-   ( \zeta ^{ i \kappa(\xi) \sigma_3})_+ =  e^{-2\pi \kappa(\xi) \sigma_3}= e^{ \log (1-\xi|r_\xi|^2 )\sigma_3} 
 \end{align*}
 while along the positive real axis,
 \begin{align*}
 ( \zeta ^{-i \kappa(\xi) \sigma_3})_-   ( \zeta ^{i \kappa(\xi) \sigma_3})_+ =  \ {\rm I}. 
 \end{align*}
This implies that the matrix $\Phi$ has  the same (constant)  jump matrix along the negative and positive real axis:
 \begin{equation}
 V^{(0)} =  \Twomat{1-\xi\,|r_\xi|^2}{r_\xi}{-\xi\,\overline{r}_\xi}{1} .
\end{equation}
Note that the matrix $V^{0}$ is similar to the jump matrix $V^{(1)}$ of the original RHP \ref{prob:DNLS.RH0} (see \eqref{DNLS.V}).
The effect of our sequence of transformations is that, in the large $t$  limit, the entries have been replaced by their localized version at the stationary phase
 point $\xi$.

The $2\times 2$ matrix $\Phi$ satisfies the following model RHP.
\begin{problem}
\label{prob-model}
Find a $2\times 2$ matrix-valued function $\Phi(z;\xi)$, analytic on $\bbC \setminus \bbR$,
with the following properties:
\begin{enumerate}

\item	$\Phi(\zeta;\xi) \sim e^{-\frac{i}{4} \zeta^2 \sigma_3} \zeta^{i\kappa \sigma_3}$ as $|\zeta| \rarr \infty$ in $\bbC \setminus \bbR$.
\item	$\Phi(\zeta;\xi)$ is analytic for $z \in \bbC \setminus \bbR$ with continuous boundary values $\Phi_\pm$
		on $\bbR$,
		
\item	The jump relation along the real axis is 
\begin{equation}\label{jump-Phi}
\Phi_+(\zeta;\xi) = \Phi_-(\zeta;\xi) V^{(0)}
\end{equation}
\end{enumerate}
\end{problem}

To solve this problem, we need to be more precise about the behavior of $\Phi(z) $ as $\zeta  \to \infty$.
We write the 
 large-$\zeta$
behavior of $\Phi$ in the form
\begin{equation}\label{asym}
\Phi (\zeta) \sim    
\left(
	1 +\frac{m^0}{\zeta}
\right) ~ \zeta^{i\kappa \sigma_3} e^{-i\sigma_3 \zeta^2/4}, \quad \zeta \rarr \infty. 
\end{equation}
At this step of the calculation, 
 $m^0$
is unknown. It will be determined later 
when enforcing the jump conditions of the matrix $\Phi$ along the real axis.

We now compute the solution $\Phi$ in terms of parabolic cylinder functions by deriving differential equations for the entries of $\Phi$ and exploiting the required asymptotics.

\begin{lemma}
The entries of $\Phi$ obey the differential equations
\begin{align}
&{\Phi_{11}}'' + \left(\frac{\zeta^2}{4} -\beta_{12}\beta_{21} +\frac{i}{2}\right) \Phi_{11} =0  \label{m11} \\[3pt]
&{\Phi_{21}}'' + \left(\frac{\zeta^2}{4} -\beta_{12}\beta_{21} -\frac{i}{2}\right) \Phi_{21} =0  \label{m21} \\[3pt]
&{\Phi_{12}}'' + \left(\frac{\zeta^2}{4} -\beta_{12}\beta_{21} +\frac{i}{2}\right) \Phi_{12} =0  \label{m12} \\[3pt]
&{\Phi_{22}}'' + \left(\frac{\zeta^2}{4}  -\beta_{12}\beta_{21} -\frac{i}{2}\right) \Phi_{22} =0  \label{m22} 
\end{align}

\end{lemma}

\begin{proof}
Differentiating \eqref{jump-Phi} with respect to $\zeta$, we obtain
\begin{equation*}
\Big( \frac{d\Phi}{d\zeta} + \frac{1}{2} i \zeta \sigma_3 \Phi \Big)_+ =  \Big( \frac{d\Phi}{d\zeta} + \frac{1}{2}i  \sigma_3\zeta \Phi\Big)_- V^{(0)}.
\end{equation*}
We  know that $ {\rm det } ~V^{(0)} =1$, thus ${\rm det} ~ \Phi_+= {\rm det}~ \Phi_-$ and ${\rm det} ~\Phi$ is analytic in the whole complex plane.
It is equal to one at infinity, thus by Liouville theorem,  ${\rm det} ~ \Phi =1$. 
 It follows that $(\Phi)^{-1}$ exists and is bounded.
The matrix 
$\Big( \displaystyle{\frac{d\Phi}{d\zeta}} + \frac{ i}{2}  \sigma_3 \zeta\Phi \Big) \Phi^{-1} $  has no jump along the 
real line and is therefore an entire function of $\zeta$.
Let us compute its behavior at infinity. Returning to \eqref{tildeN}, we have that 
\begin{align} \label{inf1}
\Big(\frac{d\Phi}{d\zeta} + \frac{ i \zeta}{2}  \sigma_3 \Phi \Big) \Phi^{-1} 
	&= 	\Big( \frac{d  \bfN ^\RHP }{d\zeta} 
				+ \bfN ^\RHP \frac{i\kappa\sigma_3}{\zeta} \Big) 
					(\bfN^ \RHP)^{-1}  \\
			\nonumber
	&\quad
				+\frac{i\zeta}{2} \left[\sigma_3, \bfN ^\RHP\right] 
				 (\bfN^\RHP)^{-1}.
\end{align}
The first term in the right-hand side of \eqref{inf1} tends to 0
as $\zeta \rarr \infty$, while the second term behaves like
 $O(1/\zeta)$. 
For the last term in the  right-hand side of \eqref{inf1}, we use that 
\begin{equation*}
\label{asym-N}
\bfN ^\RHP(\zeta) \sim    \left(1 +\frac{m^{(0)}}{\zeta} \right).
\end{equation*}
Defining
\begin{equation*}
\beta \equiv \frac{i }{2} [\sigma_3,\bfN ^\RHP_{(1)}]  = \ \Twomat{0}
							{i m_{12}^{(0)}}
							{- im_{21}^{(0)}}   
							{0}	
\end{equation*}
Equivalently,  { $\beta_{12} = i m_{12}^{(0)} $ and 
$\beta_{21} =  -i m_{21}^{(0)}$}. 
Again applying Liouville's theorem, the $2\times 2$ matrix $\Phi$ 
satisfies the ODE:
\begin{equation} 
\label{ode}
\frac{d\Phi}{d\zeta} + \frac{ i \zeta}{2}  \sigma_3 \Phi  =\beta \Phi
\end{equation}
where $\beta$ is an off-diagonal matrix.  

The system  \eqref{ode} decouples into two first-order systems for $(\Phi_{11}, \Phi_{21})$ and $(\Phi_{12}, \Phi_{22})$,
\begin{align}
\label{psi1}
\begin{cases}
\dfrac{d\Phi_{11}}{d\zeta} + \frac{1}{2}i \zeta \Phi_{11} &= \beta_{12} \Phi_{21} \\
\\
\dfrac{d\Phi_{21}}{d\zeta} -\frac{1}{2} i \zeta \Phi_{21} &= \beta_{21} \Phi_{11}
\end{cases}
\end{align}
and
\begin{align}
\label{psi2}
\begin{cases}
\dfrac{d\Phi_{12}}{d\zeta} +\frac{1}{2} i \zeta \Phi_{12} &= \beta_{12} \Phi_{22}    \\
\\
\dfrac{d\Phi_{22}}{d\zeta} - \frac{1}{2}i \zeta \Phi_{22} &= \beta_{21} \Phi_{12}.
\end{cases}
\end{align}
Combining the above equations, one obtains that  
the  entries of $\Phi$ satisfy  \eqref{m11}-\eqref{m22}.
\end{proof}

The next step is to complement the  ODEs  with additional  conditions taking into  account the conditions at infinity  as well as the jump conditions of $\Phi$. This will determine $\Phi$ uniquely and will identify the coefficients $\beta_{12}, \beta_{21}$.

The  parabolic cylinder equation is  
\begin{equation} \label{para-cyl}
 y'' + \left(-\frac{z^2}{4} + a + \frac{1}{2} \right) y =0
\end{equation}
The  parabolic cylinder functions $D_a(z)$, $D_a(-z)$, $D_{-a-1}(iz)$, $D_{-a-1} (-iz)$  all satisfy \eqref{para-cyl} and are entire for any value $a$.

The  large-$z$ behavior of $D_a(z)$ is given by the following formulas.
\footnote{Writing $D_a(z)= U_{-a-1/2}(z)$ (see  \url{http://dlmf.nist.gov/12.1}), these formulae follow from
\url{http://dlmf.nist.gov/12.9.E1} and \url{http://dlmf.nist.gov/12.9.E3}.}
\begin{equation}
\label{para-123}
D_a (z) \sim 
	\begin{cases}
		z^{a} e^{-z^2/4}\ ,  
			& |\arg (z) | <\dfrac{3\pi}{4} \\
		z^{a} e^{-z^2/4}- 
			\dfrac{\sqrt{2\pi}}{\Gamma(-a)} e^{ia \pi} z^{-a-1} e^{z^2/4}\ ,
			& \dfrac{\pi}{4}  < \arg (z)  < \dfrac{5\pi}{4} \\
 		 z^{a} e^{-z^2/4}  - 
 		 	\dfrac{\sqrt{2\pi}}{\Gamma(-a)}
 		 		 e^{-ia \pi} z^{-a-1} e^{z^2/4} \ ,   
			&	-\dfrac{5\pi}{4}  < \arg (z)  < -\dfrac{\pi}{4} . 
	\end{cases}
\end{equation}

\begin{proposition}\label{explicit}
The unique solution to Problem \ref{prob-model} is given by
\begin{equation}
\label{SimpliPhi+}
 \Phi (\zeta;\xi)= \Twomat{e^{-\frac{3\pi}{4}\kappa}
				D_{i\kappa}(\zeta e^{-3i\pi/4})}
			{\dfrac{e^{\frac{\pi}{4}(\kappa-i)}}{\beta_{21} }
				 (-i\kappa) D_{-i\kappa-1}(\zeta e^{-\pi i/4})}
			{\dfrac{e^{-\frac{3\pi}{4}(\kappa+i)}}{\beta_{12}}
				i\kappa D_{i\kappa -1}(\zeta e^{-3i\pi/4})}
			{e^{\pi\kappa/4}D_{-i\kappa}(\zeta e^{-i\pi/4})}	
\end{equation}
for $\Imag(\zeta)>0$ and
\begin{equation}
\label{SimpliPhi-}
 \Phi (\zeta;\xi)= \Twomat{e^{\pi\kappa/4}D_{i\kappa}(\zeta e^{\pi i/4})}
			{-\dfrac{i\kappa}{\beta_{21}} 
				e^{-\frac{3\pi}{4}(\kappa-i)}
				D_{-i\kappa-1}(\zeta e^{3i\pi/4})}
			{\dfrac{(i\kappa)}{\beta_{12}}
				e^{\frac{\pi}{4}(\kappa+i)}
				D_{i\kappa-1}(\zeta e^{\pi i/4})}
			{e^{-3\pi \kappa/4}D_{-i\kappa}(\zeta e^{3i\pi/4})}
\end{equation}
if $\Imag(\zeta) <0$. 
\end{proposition}

\begin{proof}
We set $\nu = \beta_{12}\beta_{21}$.
 For $\Phi_{11}$, we introduce  the new variable $\zeta_1=  \zeta\,e^{-3i\pi/4}$, and equation \eqref{m11}
becomes
\begin{equation*} 
{\Phi_{11}}'' + \left(- \frac{\zeta_1^2}{4} + i\nu +\frac{1}{2}\right)  \Phi_{11} =0.
\end{equation*}
In the upper half plane, $0<{\rm Arg}  ~\zeta <\pi$, thus    $-3\pi/4 < {\rm Arg} ~\zeta_1 < \pi/4$. 
Choosing $ \nu= \kappa$ 
{ (by comparing \eqref{asym} and \eqref{para-123})}  and identifying the 
large-$\zeta$
behavior gives
\begin{equation}
\label{phi11+}
\Phi_{11} (\zeta) =    e^{ -\frac{3\pi}{4} \kappa} D_{i\kappa} (\zeta e^{-3i\pi/4}), \qquad  \zeta\in \bbC^+.
\end{equation}
Using equation \eqref{psi1}, we calculate 
\begin{equation}
\label{phi21+}
\Phi_{21}=\frac{1}{\beta_{12}} e^{{- \frac{3\pi}{4} \kappa}}
\left(\partial_{\zeta} (D_{i\kappa} (\zeta e^{-3i\pi/4}))+\frac{i\zeta}{2} D_{i\kappa} (\zeta e^{-3i\pi/4}) \right).
\end{equation}
We proceed in the same way for $\Phi_{12}$ and $\Phi_{22}$. 
 In term of ${  \zeta_1=e^{-\pi i/4}\zeta}$, equation \eqref{m22}
is
\begin{equation*} 
\Phi_{22}'' + \left(- \frac{\zeta_1^2}{4} -i\nu +\frac{1}{2}\right)  \Phi_{22} =0.
\end{equation*}
To correctly match the large-$\zeta$ behavior  $\Phi_{22} (\zeta) \sim \zeta^{-i\kappa} e^{i\zeta^2/4}$,
we choose the solution 
\begin{equation}
\label{phi22+}
\Phi_{22}(\zeta) =     e^{\frac{\pi \kappa}{4}}  D_{-i\kappa}( e^{-i\pi/4} \zeta) \qquad \zeta\in \bbC^+.
\end{equation}
Finally, using equation \eqref{psi2}
\begin{equation}
\label{phi12+}
\Phi_{12}(\zeta)=\frac{1}{\beta_{21}} e^{{ \frac{\pi}{4} }\kappa}
\left(\partial_{\zeta} (D_{-i\kappa} (\zeta e^{-i\pi/4}))-\frac{i\zeta}{2} D_{-i\kappa} (\zeta e^{-i\pi/4}) \right).
\end{equation}

\bigskip

We repeat this calculation to compute 
$\Phi(\zeta)$ in the lower complex plane. 

Let ${ \zeta_2 = \zeta e^{i\pi/4}}$.  In terms of $\zeta_2$ , $\Phi_{11}$ satisfies
\begin{equation}
\label{phi11-}
\Phi_{11}'' + \left(-\frac{\zeta_2^2}{4} + i\nu +\frac{1}{2}\right)  \Phi_{11} =0.
\end{equation}
For $  -\pi < {\rm Arg} ~ \zeta < 0$,   $  -3\pi/4 < {\rm Arg} ~(\zeta_2) <  \pi/4 $,
thus we choose to identify $\Phi_{11}$ to  a multiple of $D_{i\nu} ( \zeta_2)$.
We find that for $\zeta\in \bbC^-$
$$ \Phi_{11}(\zeta) = e^{\frac{\pi}{4}\kappa} D_{i\kappa} (e^{i\pi/4}\zeta).$$
Similarly,
\begin{equation}
\label{phi21-}
\Phi_{21}(\zeta)=\frac{1}{\beta_{12}} e^{{\frac{\pi}{4} \kappa}}
\left(\partial_{\zeta} (D_{i\kappa} (\zeta e^{i\pi/4}))+\frac{i\zeta}{2} D_{i\kappa} (\zeta e^{i\pi/4}) \right)
\end{equation}
We now turn to $\Phi_{22}$ and $\Phi_{12}$. To match the large-$\zeta$ behavior $\Phi_{22}(\zeta) \sim \zeta^{-i\kappa} e^{i\zeta^2/4}$
we choose to identify $\Phi_{22}$  as
\begin{equation}
\label{phi22-}
 \Phi_{22} (\zeta)=  e^{-3\pi\kappa/4}       D_{-i\kappa} ( e^{3i\pi/4} \zeta)
\end{equation}
 and
 \begin{equation}
 \label{phi12-}
 \Phi_{12} (\zeta)=\frac{1}{\beta_{21}} e^{{ -\frac{3\pi}{4} \kappa}}
\left(\partial_{\zeta} (D_{i\kappa} (\zeta e^{3i\pi/4}))-\frac{i\zeta}{2} D_{i\kappa} (\zeta e^{3i\pi/4}) \right).  
 \end{equation}
 Using \eqref{phi11+}, \eqref{phi21+}, \eqref{phi22+}, and \eqref{phi12+} together with the identity
\begin{equation} \label{D.recur}
D_a'(z) + \frac{z}{2} D_a(z) = a D_{a-1}(z),
\end{equation}
we can now write $\Phi(\zeta;\xi)$ for $\Imag(\zeta)>0$ in the form \eqref{SimpliPhi+}. Similarly, it follows from \eqref{phi11-}, \eqref{phi21-}, \eqref{phi22-}, and \eqref{phi12-} that $\Phi(\zeta;\xi)$ is given by \eqref{SimpliPhi-} for $\Imag(\zeta)<0$. 
\end{proof}

We  now impose the jump conditions to find  the coefficients $\beta_{12}$ and $\beta_{21}$. We will later use this computation of $\beta_{12}$ to compute the asymptotic behavior of $q(x,t)$.

\begin{lemma}
\label{lemma:beta12}
Suppose that $\rho \in H^{2,2}(\bbR)$ with $\inf_{z \in \bbR} (1-z|\rho(z)|^2) >0$. Then:
\begin{equation}
\label{beta12-ampl} 
|\beta_{12}|^2 =  \frac{\kappa}{\xi} = -\frac{1}{2 \pi\xi}\log \left(1-\xi |\rho(\xi)|^2 \right)
\end{equation}
and
\begin{multline} 
\label{beta12-arg}
\arg \beta_{12} =  \frac{\pi}{4}  -   \kappa \log (8t) \\
+4 t\xi^2 + \arg(\Gamma(i\kappa) )  +  \arg \rho(\xi)
+\frac{1}{\pi}\int_{-\infty}^{\xi}\log|s-\xi| \, d\log(1-s|\rho(s)|^2).
\end{multline}
\end{lemma}

\begin{remark}
Note that the amplitude \eqref{beta12-ampl} has a removable discontinuity at $\xi=0$ as
\begin{align*}
\lim_{\xi \rarr 0}
	 \frac{\log\left(1-\xi|\rho(\xi)|^2 \right)}{\xi}
	 & = 
	 	-\lim_{\xi \rarr 0}
	 		-\frac{|\rho(\xi)|^2+\xi 
	 			\left[
	 				\rho'(\xi) \overline{\rho(\xi)}+
	 				\overline{\rho(\xi)} \rho'(\xi)
	 			\right]}{1-\xi |\rho(\xi)|^2}\\[5pt]
	 &=	-|\rho(0)|^2.
\end{align*}
\end{remark}

\begin{proof}
We  know that $\beta_{12}\beta_{21} = \kappa(\xi)$ and
$$
(\Phi_-)^{-1}\Phi_+=V^{(0)}=\Twomat{1-\xi\,|r_\xi|^2}{r_\xi}{ -\xi\,\overline{r}_\xi}{1}.
$$
Combining 
\eqref{phi11+}, \eqref{phi21+}, \eqref{phi11-}, 
and \eqref{phi21-}, we obtain
\begin{align*}
-\xi\,\overline{r}_\xi &=\Phi_{11}^-\Phi_{21}^+-\Phi_{21}^-\Phi_{11}^+\\
                      &=e^{\frac{\pi}{4}\kappa} D_{i\kappa} (e^{i\pi/4}\zeta)\frac{1}{\beta_{12}} e^{{ -\frac{3\pi}{4} \kappa}}
\left(\partial_{\zeta} (D_{i\kappa} (\zeta e^{-3i\pi/4}))+\frac{i\zeta}{2} D_{i\kappa} (\zeta e^{-3i\pi/4}) \right)\\
                     &\quad -\frac{1}{\beta_{12}} e^{{\frac{\pi}{4} \kappa}}
\left(\partial_{\zeta} (D_{i\kappa} (\zeta e^{i\pi/4}))+\frac{i\zeta}{2} D_{i\kappa} (\zeta e^{i\pi/4}) \right) e^{ -\frac{3\pi}{4} \kappa} D_{i\kappa} (\zeta e^{-3i\pi/4})\\
                     &=\frac{e^{-\frac{\pi\kappa}{2}}}{\beta_{12}}
                     W\left(D_{i\kappa} (e^{i\pi/4}\zeta), \,D_{i\kappa} (\zeta e^{-3i\pi/4})\right)\\
                     & = 
                   \frac{\sqrt{2\pi}e^{-\frac{\pi\kappa}{2}}e^{i\pi/4}}{\beta_{12}\Gamma(-i\kappa)}
\end{align*}
where  we have used the  Wronskian  \eqref{D.Wronski} in the last equality. 
It follows from the above computations that
$$ 
\beta_{12}  = 
\frac{\sqrt{2\pi} e^{-\pi \kappa/2} e^{\pi i/4}}
		{ -\xi\, \overline{r_\xi} \,\Gamma(-i\kappa)}.
$$
From the identities 
$$  \Gamma(\overline{z}) = \overline{\Gamma(z) }, \quad 
\left| \Gamma(i\kappa) \right|^2
 = \frac{\pi }{\kappa \sinh \pi \kappa}
$$
we see that 
$$
|\beta_{12}|^2 =\left| \sqrt{2\pi}\frac{e^{-\pi \kappa/2}}{\xi \, \overline{r_\xi}\, \Gamma(i\kappa)} \right|^2 = 2
\, \frac{\kappa e^{-\pi \kappa} \sinh \pi \kappa}{\xi^2 |r_\xi|^2} = \kappa \, \frac{1}{\xi^2 |r_\xi|^2}\left(1- e^{-2\pi \kappa} \right).
$$
Recalling that 
$$ \kappa(\xi) = -\frac{1}{2\pi} \log\left(1-\xi |\rho(\xi)|^2 \right)
$$
we compute
$$ 
1- e^{-2\pi \kappa} = \xi |\rho(\xi)|^2
$$
so that  \eqref{beta12-ampl} holds.

On the other hand, since $\xi<0$
\begin{equation*}
 \arg \beta_{12} = \frac{\pi}{4} + \arg r_\xi +
\arg(\Gamma(i\kappa) ).
\end{equation*}
Substituting the definition  of $r_\xi$ given in \eqref{r-xi}
$$\arg r_\xi = \arg \rho(\xi) + \arg \delta_0^2  - \kappa(\xi) \log(8t) + 4 t \xi^2.$$
 We also have, by integration by parts
\begin{align*}
\delta_0^{2}(\xi)&=\exp\left(2i\int_{-\infty}^{\xi}\dfrac{\kappa(s)-\chi(s)\kappa(\xi)}{s-\xi}ds     \right)\\
                 &=\exp\left(-2i\int_{-\infty}^{\xi} \log|s-\xi|  \, d\kappa(s)\right)\\
                \nonumber
                 &=\exp\left(\frac{i}{\pi}\int_{-\infty}^{\xi} \log|s-\xi|  \,d\log(1-s|\rho(s)|^2)\right)
\end{align*}
thus \eqref{beta12-arg} holds.
\end{proof}

\section{The $\dbar$-Problem}
\label{sec:dbar}
 
We now define  the row vector-valued matrix
\begin{equation}
\label{N3}
\bfN^{(3)}(z;x,t) = \bfN^{(2)}(z;x,t) \bfN^\RHP(z;\xi)^{-1}. 
\end{equation}
It is clear that $\bfN^\RHP$ needs to be an invertible matrix-valued function in order to carry out this reduction.  An argument similar to that 
given in \cite{BJM16}  shows that $\bfN^{(3)}$ satisfies a pure $\dbar$-problem; we will use this fact to prove that $\bfN^{(3)}$ is close to $(1,0)$ as $t \rarr \infty$ with an explicit rate of decay. 

Since $\bfN^\RHP(z;\xi)$ is holomorphic in $\bbC \setminus \Sigma^{(2)}$, we may compute
\begin{align*}
\dbar \bfN^{(3)}(z;x,t) 	&=	\dbar \bfN^{(2)}(z;x,t) \bfN^\RHP(z;\xi)^{-1}\\	
								&=	\bfN^{(2)}(z;x,t) \, \dbar \calR^{(2)}(z) \bfN^\RHP(z;\xi)^{-1}	&\text{(by \eqref{N2.dbar})}\\
								&=	\bfN^{(3)}(z;x,t) \bfN^{\RHP}(z;\xi) \, \dbar \calR^{(2)}(z) \bfN^\RHP(z;\xi)^{-1}	& \text{(by \eqref{N3})}\\
								&=	\bfN^{(3)}(z;x,t)  W(z;x,t)
\end{align*}
where
$$ W(z;x,t) = \bfN^{\RHP}(z;\xi) \, \dbar \calR^{(2)}(z) \bfN^\RHP(z;\xi)^{-1}. $$
We thus arrive at the following pure $\dbar$-problem.

\begin{problem}
\label{prob:DNLS.dbar}
Given $x,t \in \bbR$ and $\rho \in H^{2,2}(\bbR)$ with $1-z |\rho(z)|^2 >0$ for all $z \in \bbR$, find a continuous, row vector-valued function
$\bfN^{(3)}(z;x,t)$ on $\bbC$ with the following properties:
\begin{enumerate}
\item		$\bfN^{(3)}(z;x,t) \rarr (1,0)$ as $|z| \rarr \infty$, 
\medskip
\item		$\dbar \bfN^{(3)}(z;x,t) = \bfN^{(3)}(z;x,t) W(z;x,t)$.
\end{enumerate}
\end{problem}

We can recast this problem as a Fredholm-type integral equation using the solid Cauchy transform
$$ (Pf)(z) = \frac{1}{\pi} \int_\bbC \frac{1}{z-\zeta} f(\zeta) \, dm(\zeta) $$
where $dm$ denotes Lebesgue measure on $\bbC$.  The following lemma is standard.

\begin{lemma}
A continuous, bounded row vector-valued function $\bfN^{(3)}(z;x,t)$ solves Problem \eqref{prob:DNLS.dbar} if and only if
\begin{equation}
\label{DNLS.dbar.int}
\bfN^{(3)}(z;x,t) = (1,0) + \frac{1}{\pi} \int_\bbC \frac{1}{z-\zeta} \bfN^{(3)}(\zeta;x,t) W(\zeta;x,t) \, dm(\zeta).
\end{equation}
\end{lemma}

 Using the formulation \eqref{DNLS.dbar.int}, we will prove:

\begin{proposition}
\label{prop:N3.est}
Suppose that $\rho \in H^{2,2}(\bbR)$ and $c \coloneqq \inf_{z \in \bbR} \left(1-z |\rho(z)|^2 \right) >0$ strictly.
Then, for sufficiently large times $t>0$, there exists a unique solution $\bfN^{(3)}(z;x,t)$ for Problem \ref{prob:DNLS.dbar} with the property that 
\begin{equation}
\label{N3.exp}
\bfN^{(3)}(z;x,t) = I + \frac{1}{z}  \bfN^{(3)}_1(x,t) + o_{\xi,t} \left( \frac{1}{z} \right) 
\end{equation}
for $z=i\sigma$ with $\sigma \rarr +\infty$. Here

\begin{equation}
\label{N31.est}
\left| \bfN^{(3)}_1(x,t) \right| \lesssim t^{-3/4}
\end{equation}
 where the implied constant in \eqref{N31.est} is independent of $\xi$ and $t$ and uniform for 
$\rho$ in a bounded subset of $H^{2,2}(\bbR)$ with $\inf_{z \in \bbR} (1-z|\rho(z)|^2) \geq c>0$ for a fixed $c>0$.
\end{proposition}

\begin{remark}
The remainder estimate in \eqref{N3.exp} need not be (and is not) uniform in $\xi$ and $t$; what matters for the proof of Theorem \ref{thm:main1} is that the implied constant in the estimate \eqref{N31.est} for $\bfN^{(3)}_1(x,t)$ is independent of $\xi$ and $t$.
\end{remark}

\begin{proof}[Proof of Proposition \ref{prop:N3.est}, given Lemmas \ref{lemma:dbar.R.bd}--\ref{lemma:N31.est}]
 As in \cite{BJM16}  and \cite{DM08}, we   first show that, for large times, the integral operator $K_W$ 
defined by
\begin{equation*}
\left( K_W f \right)(z) = \frac{1}{\pi} \int_\bbC \frac{1}{z-\zeta} f(\zeta) W(\zeta) \, dm(\zeta)
\end{equation*}
(suppressing the parameters $x$ and $t$) 
obeys the estimate
\begin{equation}
\label{dbar.int.est1}
\norm{K_W}{L^\infty \rarr L^\infty} \lesssim t^{-1/4} 
\end{equation}
where the implied constants depend only on $\norm{\rho}{H^{2.2}}$ and 
$c : \inf_{z \in \bbR} \left( 1-z|\rho(z)|^2 \right)$ 
 and, in particular, are independent of $\xi$ and $t$.   This is the object of Lemma \ref{lemma:KW}.
It shows in particular that the solution formula
\begin{equation}
\label{N3.sol}
\bfN^{(3)} = (I-K_W)^{-1} (1,0) 
\end{equation}
makes sense and defines an $L^\infty$ solution of \eqref{DNLS.dbar.int} bounded uniformly in $\xi\in \bbR$ and $\rho$ in a bounded subset of $H^{2,2}(\bbR)$ with $c>0$.  

We  then prove that the solution $\bfN^{(3)}(z;x,t)$ has a large-$z$ asymptotic expansion of the form \eqref{N3.exp}
where $z \rarr \infty$ along the \emph{positive imaginary axis} (Lemma \ref{lemma:N3.exp}). Note that, for such $z$, we can bound $|z-\zeta|$ below by a constant times $|z|+|\zeta|$.
The remainder need not be bounded uniformly in $\xi$. Finally, we  prove  estimate \eqref{N31.est}
where the constants are uniform in $\xi$ and in $\rho$ belonging to a bounded subset of $H^{2,2}(\bbR)$ with 
$\inf \left(1 - z|\rho(z)|^2 \right)$ bounded below by a strictly positive fixed constant ( Lemma \ref{lemma:N31.est}).

\end{proof}

Estimates \eqref{N3.exp}, \eqref{N31.est}, and \eqref{dbar.int.est1} rest on the  bounds stated in the next four lemmas.

\begin{lemma}
\label{lemma:dbar.R.bd}
Let 
 $z=(u-\xi)+iv$. We have 
\begin{equation}
\label{dbar.R2.bd}
\left| \dbar \calR^{(2)}(z;\xi) \right| 	\lesssim		
		\begin{cases}
				\left( |p_i'(\Real (z))| -\log|z-\xi|\right) e^{-8t|u||v|}, 			&	|z-\xi| \leq 1, \\
				\\
				\left( |p_i'(\Real (z))| +\dfrac{1}{\left(1+|z-\xi|^2 \right)} \right) e^{-8t|v||u|}, 	&	|z-\xi| > 1,
		\end{cases}
\end{equation}
where all implied constants are uniform in $\xi \in \bbR$ and $t>1$.
\end{lemma}

\begin{proof}
Estimate \eqref{dbar.R2.bd} follows from Lemma \ref{lemma:dbar.Ri} and Remark \ref{rem:dbar.Ri}.
The quantities $p_i'(\Real z)$ are all bounded uniformly for  $\rho$ in a bounded subset of $H^{2,2}(\bbR)$ and  
$\inf_{z \in \bbR} \left(1- z |\rho(z)|^2\right) \geq c >0$ for a fixed $c$.  

\end{proof}

\begin{lemma}
\label{lemma:RHP.bd}
\begin{align}
\label{RHP.bd1}
\norm{\bfN^\RHP(\dotarg;\xi)}{\infty}	&	\lesssim		1\\[5pt]
\label{RHP.bd2}
\norm{\bfN^\RHP(\dotarg;\xi)^{-1}}{\infty}	&	\lesssim	1
\end{align}
Again, all implied constants are uniform in $\xi \in \bbR$ and $t>1$.
\end{lemma}

The proof of this Lemma is given in Appendix \ref{app:NRHP.bd}.

\begin{lemma}
\label{lemma:KW}
Suppose that $\rho \in H^{2,2}(\bbR)$ and $c : \inf_{z \in \bbR} \left(1-z |\rho(z)|^2 \right) >0$ strictly.
Then, the estimate \eqref{dbar.int.est1}
holds, where the implied constants depend on $\norm{\rho}{H^{2,2}}$ and $c$.
\end{lemma}

\begin{proof}
To prove \eqref{dbar.int.est1}, first note that
\begin{align}
 \norm{K_W f}{\infty} &\leq \norm{f}{\infty} \int_\bbC \frac{1}{|z-\zeta|}|W(\zeta)| \, dm(\zeta) 
                                \end{align}
so that we need only estimate the right-hand integral. We will prove the estimate in the region $ z\in\Omega_1$ since estimates for $\Omega_3$, $\Omega_4$, and $\Omega_6$ are similar.  In the region $\Omega_1$, we may estimate
$$ |W(\zeta)| \leq \norm{\bfN^{\RHP}}{\infty} \norm{(\bfN^{\RHP})^{-1}}{\infty} \left| \dbar R_1\right| |e^{2it\theta}|.$$
Setting $z=\alpha+i\beta$ and $\zeta=(u-\xi)+iv$, the region $\Omega_1$ corresponds to $v \geq 0$, $u \geq v$. We then have from \eqref{dbar.R2.bd}l \eqref{RHP.bd1}, and \eqref{RHP.bd2} that
$$
 \int_{\Omega_1}  \frac{1}{|z-\zeta|} |W(\zeta)| \, dm(\zeta)  \lesssim  I_1 + I_2 + I_3
$$
where
\begin{align*}
I_1 	&=	\int_0^\infty \int_v^\infty \frac{1}{|z-\zeta|} |p_1'(u)| e^{-8tuv} \, du \, dv \\[5pt]
I_2	&=	\int_0^1 \int_v^1 \frac{1}{|z-\zeta|} \left| \log(u^2 + v^2) \right| e^{-8tuv} \, du \, dv\\[5pt]
I_3	&=	\int_0^\infty \int_v^\infty \frac{1}{|z-\zeta|} \frac{1}{1+ |\zeta-\xi|} e^{-8tuv} \, du \, dv.
\end{align*}
We recall from \cite[proof of Proposition C.1]{BJM16} the bound
\begin{equation*}
\norm{\frac{1}{|z-\zeta|}}{L^2(v,\infty)} \leq \frac{\pi^{1/2}}{|v-\beta|^{1/2}}
\end{equation*}
where $\zeta=u-\xi+iv$ and $z=\alpha + i \beta$ (our parameterization of $\zeta$ differs slightly from theirs).  Using this bound and Schwarz's inequality on the $u$-integration we may bound $I_1$  by constants times
$$
(1+\norm{p_1'}{2})  \int_0^\infty \frac{1}{|v-\beta|^{1/2}} e^{-tv^2} \, dv \lesssim t^{-1/4}
$$
(see for example \cite[proof of Proposition C.1]{BJM16} for the estimate)
For $I_2$, we remark that $ |\log(u^2+v^2)| \lesssim 1+ |\log(u^2)|$ and that $1+ |\log(u^2)|$ is square-integrable on $[0,1]$. We can then argue as before to conclude that
$ I_2 \lesssim t^{-1/4}$. Finally, the inequality 
$$ \frac{1}{1+|\zeta-\xi|} \leq \frac{1}{1+u}$$ shows that we can bound $I_3$ in a similar way. 
It now follows that
$$ \int_{\Omega_3} \frac{1}{|z-\zeta|} |W(\zeta)| \, dm(\zeta) \lesssim t^{-1/4} $$
which, together with similar estimates for the integrations over $\Omega_3$, $\Omega_4$, and $\Omega_6$,  proves \eqref{dbar.int.est1}.
\end{proof}

\begin{lemma}
\label{lemma:N3.exp}
For $z=i\sigma$ with $\sigma \rarr +\infty$, the expansion \eqref{N3.exp} holds with 
\begin{equation}
\label{N3.1}
N^{(3)}_1(x,t) = \frac{1}{\pi} \int_{\bbC} N^{(3)}(\zeta;x,t) W(\zeta;x,t) \, dm(\zeta) . 
\end{equation}
\end{lemma}

\begin{proof}
We write   \eqref{DNLS.dbar.int} as 
$$
\bfN^{(3)}(z;x,t) = (1,0) + \frac{1}{z} \bfN^{(3)}_1(x,t) + \frac{1}{\pi z} \int_{\bbC} \frac{\zeta}{z-\zeta} \bfN^{(3)}(\zeta;x,t) W(\zeta;x,t) \, dm(\zeta)
$$
where $\bfN^{(3)}_1$is given by \eqref{N3.1}. If $z=i\sigma$ and $\zeta \in \Omega_1 \cup \Omega_3 \cup \Omega_4 \cup \Omega_6$, it is easy to see that $|\zeta|/|z-\zeta|$ is bounded above by a fixed constant independent of $z$, while $|\bfN^{(3)}(\zeta;x,t)| \lesssim 1$ by the remarks following \eqref{N3.sol}. If we can show that $\int_\bbC |W(\zeta;x,t)| \, dm(\zeta)$ is finite, it will follow from the Dominated Convergence Theorem that 
$$
\lim_{\sigma \rarr \infty} \int_\bbC \frac{\zeta}{i\sigma-\zeta} \bfN^{(3)}(\zeta;x,t) W(\zeta;x,t) \, dm(\zeta) = 0 
$$ 
which implies the required asymptotic estimate. We will estimate $\dint_{\hspace{-1.25mm} \Omega_1} |W(\zeta)| \, dm(\zeta)$ since the other estimates are similar. We have 
$$\Omega_1= \left\{ (u-\xi,v): v \geq 0, \, v \leq u < \infty\right\}.$$ 
Using \eqref{dbar.R2.bd}, \eqref{RHP.bd1}, and \eqref{RHP.bd2}, we may then estimate
$$
\int_{\Omega_1} |W(\zeta;x,t)| \, dm(\zeta)	\lesssim  I_1+I_2+I_3
$$
where
\begin{align*}
I_1	&=	\int_0^\infty \, \int_v^\infty \left| p_1'(\xi-u) \right| e^{-8tuv} \, du \, dv\\
I_2	&=	\int_0^1 \int_v^1 \left| \log(u^2 + v^2) \right| e^{-8tuv} \, du \, dv \\
I_3	&=	\int_0^\infty \, \int_v^\infty \frac{1}{\sqrt{1+u^2+v^2}} e^{-8tuv} \, du \, dv.
\end{align*}
To estimate $I_1$, we use the Schwarz inequality on the $u$-integration to obtain
$$
I_1 	\leq 	\norm{p_1'}{2} \frac{1}{4\sqrt{t}} \int_0^\infty \frac{1}{\sqrt{v}}e^{-  8 tv^2} \, dv
		=		\norm{p_1'}{2} \frac{\Gamma(1/4)}{ 8^{5/4}t^{3/4}}.
$$
Similarly, since $\log(u^2+v^2) \leq \log(2u^2)$ for $v \leq u \leq 1$, we may similarly bound
$$
I_2  \leq  \norm{\log(2u^2)}{L^2(0,1)}\frac{\Gamma(1/4)}{8^{5/4}t^{3/4}}.
$$
Finally,  to estimate $I_3$, we note that $1+u^2+v^2 \geq  1+u^2$ and $(1+u^2)^{-1/2} \in L^2(\bbR^+)$, so we may 
similarly conclude that
$$
I_3 \leq \norm{(1+u^2)^{-1/2}}{2} \frac{\Gamma(1/4)}{ 8^{5/4}t^{3/4}}.
$$
These estimates together show that
\begin{equation}
\label{W.L1.est}
\int_{\Omega_1} |W(\zeta;x,t)| \, dm(\zeta)	\lesssim t^{-3/4}
\end{equation}
and that the implied constant depends only on $\norm{\rho}{H^{2,2}}$.  In particular, the integral \eqref{W.L1.est} is bounded uniformly as $t \rarr \infty$.
\end{proof}

The estimate \eqref{W.L1.est} is also strong enough to prove 
 \eqref{N31.est}.

\begin{lemma}
\label{lemma:N31.est}
The estimate   \eqref{N31.est} 
 holds with constants uniform in $\rho$ in a bounded subset of $H^{2,2}(\bbR)$ and $\inf_{z \in \bbR} \left(1-z|\rho(z)|^2 \right) > 0$ strictly.
\end{lemma}

\begin{proof}
From the representation formula \eqref{N3.1}, Lemma \ref{lemma:KW}, and the remarks following, we have
$$ \left|\bfN^{(3)}_1(x,t) \right| \lesssim \int_\bbC |W(\zeta;x,t)| \, dm(\zeta). $$
In the proof of Lemma \ref{lemma:N3.exp}, we bounded this integral by $t^{-3/4}$ modulo constants with the required uniformities.
\end{proof}



\section{Large-Time Asymptotics}
\label{sec:large-time}

We now use estimates on the RHPs to compute $q(x,t)$ via the reconstruction formula \eqref{DNLS.q}  in the case 
$x>0$, and $t \rarr +\infty$.  Working through the various changes of variables,  we have
\begin{equation}
\label{N3.to.N}
\bfN(z;x,t) = \bfN^{(3)}(z;x,t) \bfN^\RHP(z;\xi) \calR^{(2)}(z)^{-1} \delta(z)^{\sigma_3}
\end{equation}
Recalling \eqref{DNLS.q}, we need to compute the coefficient of $z^{-1}$ in the large-$z$ expansion for $\bfN(z;x,t)$. 


\begin{lemma}
\label{lemma:N.to.NRHP.asy}
For $z=i\sigma$ and $\sigma \rarr +\infty$, the asymptotic relations
\begin{align}
\label{N.asy}
\bfN(z;x,t) 				&=	(1,0) + \frac{1}{z} \bfN_1(x,t) + o\left(\frac{1}{z}\right)\\
\label{N.RHP.asy}
\bfN^\RHP(z;x,t)		&=	I + \frac{1}{z} \bfN^\RHP_1(x,t) + o\left(\frac{1}{z}\right).
\end{align}
hold. Moreover,
\begin{equation}
\label{N.to.NRHP.asy} 
\left(\bfN_1(x,t)\right)_{12} = \left(\bfN^\RHP_1(x,t)\right)_{12} + \bigO{t^{-3/4}}
\end{equation}
and the implied constants are uniform in $\xi$ and $t>0$.
\end{lemma}

\begin{proof}
By Lemma \ref{lemma:delta}(iii), 
the expansion
\begin{equation}
\label{delta.sigma.asy} 
\delta(z)^{\sigma_3} = \twomat{1}{0}{0}{1} + \frac{1}{z} \twomat{\delta_1}{0}{0}{\delta_1^{-1}} + \bigO{z^{-2}} 
\end{equation}
holds, with the remainder in \eqref{delta.sigma.asy}
 uniform in $\rho$ in a bounded subset of $H^{2,2}$. 
The form of the asymptotic expansion \eqref{N.RHP.asy} follows by construction, while \eqref{N.asy} follows from \eqref{N3.to.N},  \eqref{N.RHP.asy}, the fact that $\calR^{(2)} \equiv I$ in $\Omega_2$,
and \eqref{delta.sigma.asy}.

To prove  \eqref{N.to.NRHP.asy}, we notice that the 
diagonal matrix in \eqref{delta.sigma.asy} does not affect the $12$-component of $\bfN$. Hence, for 
$z=i\sigma$,
$$ 
\left(\bfN(z;x,t)\right)_{12} = 
		\frac{1}{z}\left(\bfN^{(3)}_1(x,t)\right)_{12} + 
		\frac{1}{z}\left(\bfN^\RHP_1(x,t)\right)_{12} + o\left(\frac{1}{z}\right)
$$
and result now follows from \eqref{N31.est}. 
\end{proof}

We now evaluate the leading asymptotic term using large-$z$ asymptotics of the model RHP.

\begin{proposition}
\label{lemma:N.RHP.asy}

The function 
\begin{equation}
\label{q.recon.bis}
q(x,t) = 2i \lim_{z \rarr \infty} zN_{12}(z;x,t)
\end{equation}
takes the form 
$$ q(x,t) = q_{as}(x,t) + \bigO{t^{-3/4}}$$
where $q_{as}(x,t)$ is given by  \eqref{result-Th1-1}
and the remainder is uniform in $\xi \in \bbR$. 
\end{proposition}

\begin{proof}
By Lemma \ref{lemma:N.to.NRHP.asy} and \eqref{q.recon.bis},
$$q_{as}(x,t)=\lim_{z\rightarrow\infty}\frac{2iz m^{(0)}_{12}}{\zeta}.$$
Recalling that  $m^{(0)}_{12}=  -i\beta_{12}$,  with $\beta_{12}$ given in \eqref{beta12-ampl}-\eqref{beta12-arg} of Lemma \ref{lemma:beta12}, and  that $z$ and $\zeta$ are related 
through \eqref{zeta-z},
we get 
\begin{align*}
q_{as}(x,t)         	&=\lim_{z\rightarrow\infty}\frac{2z\, \beta_{12}}{\sqrt{8t}(z-\xi)}\\
         				&=\frac{1}{\sqrt{t}} \alpha_1(\xi) e^{-i\kappa(\xi)\log 8t+ix^2/(4t)}
\end{align*}

where
$$\kappa(z)=-\frac{1}{2\pi}\log(1-z|\rho(z)|^2),\quad | \alpha_1(\xi)|^2=\frac{|\kappa(\xi)|}{2 |\xi|}$$
and
$$\arg \alpha_1(\xi)=
	\frac{\pi}{4}
	+\arg\Gamma(i\kappa)+\arg\rho(\xi)
	+\frac{1}{\pi}\int_{-\infty}^{\xi}\log|s-\xi| \, d\log(1-s|\rho(s)|^2).
$$

\end{proof}

Theorem \ref{thm:main1} in the case $x>0$, $t>0$ is an immediate consequence of Proposition \ref{lemma:N.RHP.asy}. We discuss 
the remaining three cases in Appendix \ref{app:4-RHPs}.




\section{Gauge Transformation}
\label{sec:gauge}

%
%

Given initial data
$u_0$ for \eqref{DNLS1}, we define gauge-transformed initial data  for \eqref{DNLS2}
\begin{equation*}
q_0(x) = u_0(x) \exp 
							\left( 
								-i \int_{-\infty}^x |u_0(y)|^2 \, dy 
							\right) 
\end{equation*}
and  the associated scattering data $\rho$ for $q_0$. From  these scattering data,  we compute the solution to \eqref{DNLS2}, and thus obtain 
the solution to the Cauchy problem for \eqref{DNLS1} with Cauchy data $u_0$ by the inverse gauge transformation
\begin{equation}
\label{q to u}
u(x,t) = q(x,t) \exp\left( i \int_{-\infty}^x |q(y,t)|^2 \, dy \right).
\end{equation}
To find  the large-time behavior for $u(x,t)$ purely in terms of 
spectral data, it suffices to evaluate large-time asymptotics for 
the expression
\begin{equation*}
\exp\left( i \int_{-\infty}^x |q(y,t)|^2 \, dy \right). 
\end{equation*}
We will prove:

\begin{proposition}
\label{prop:gauge+}
Suppose that $q_0 \in H^{2,2}(\bbR)$ and that $q(x,t)$ solves the 
Cauchy problem \eqref{DNLS2} with initial data $q_0$. 
 Let  $\rho$ be the right-hand scattering data associated to $q_0$ and 
 fix $\xi = -x/(4t)$  with $\xi \neq 0$. 
 We have the asymptotic formulae:
 \begin{enumerate}
 \item[(i)] For $t>0$, 
\begin{align*}
\exp  \Big(i  \int_{-\infty}^x |q(y,t)|^2 \, dy \Big) &= 
	 \exp \Big( -	\frac{i}{\pi} \int_{\xi}^\infty \frac{\log(1-s|\rho(s)|^2)}{s} \ ds \Big) +
	\mathcal{O}_\xi\left(\frac{1}{\sqrt{t}}\right).
\end{align*}

\item[(ii)] Similarly, for $t<0$,
\begin{align*}
\exp \Big( i \int_{-\infty}^x |q(y,t)|^2 \, dy \Big) &= \exp \Big(
	- \frac{i}{\pi} \int_{-\infty}^\xi \frac{\log(1-s|\rho(s)|^2)}{s} \ ds \Big) +
	\mathcal{O}_\xi\left(\frac{1}{\sqrt{t}}\right) .
\end{align*}
\end{enumerate}
\end{proposition}

\subsection{Beals-Coifman solutions}
Our analysis uses the Beals-Coifman solutions discussed in Paper I, Section 4. 
We recall a few key facts and refer the 
reader to Sections 1.2 and 4 of that paper for further details. Our Beals-Coifman solutions also depend on $t$ since 
the potential $q(x,t)$ and its scattering data evolve in time.

In the $\zeta$ variables, the Beals-Coifman solutions $M_\ell(x,\zeta,t)$ and $M_r(x,\zeta,t)$ are $2 \times 2$ matrix-valued functions defined 
for $\zeta \in \bbC \setminus \Sigma$, solve \eqref{L}, are analytic in $\zeta$, and have the respective 
spatial normalizations
\begin{equation}
\label{BC.norm}
\lim_{x \rarr +\infty} M_r(x,\zeta,t)		=	 \twomat{1}{0}{0}{1}, \qquad
\lim_{x \rarr -\infty}  M_\ell(x,\zeta,t)	=	\twomat{1}{0}{0}{1}.
\end{equation}
By exploiting the symmetry reduction described in Section 1.2 of Paper I, we can form Beals-Coifman solutions
$N_\ell(x,z,t)$ and $N_r(x,z,t)$ with the same respective spatial normalizations  but analytic for 
$z \in \bbC \setminus \bbR$. 
The function $\bfN(z;x,t)$ that solves Problem \ref{prob:DNLS.RH0} (the ``right'' Riemann-Hilbert problem) is the first row of $N_r(x,z,t)$.  Analogously, the first row of $N_\ell(x,z,t)$ solves the corresponding ``left'' Riemann-Hilbert problem.

If $\zeta=0$, \eqref{L} becomes
$d\Psi / dx = P(x)\Psi$ and we can use the normalizations \eqref{BC.norm} to compute 
\begin{equation}
\label{m11R.sol}
M_{11}^\pm(x,0,t)_r
	=	\exp\left( -\frac{i}{2} \int_{+\infty}^x |q(y)|^2 \, dy \right)
\end{equation}
and
\begin{equation}
\label{m11L.sol}
M_{11}^\pm(x,0,t)	_\ell
	=	\exp\left( -\frac{i}{2} \int_{-\infty}^x |q(y)|^2 \, dy \right).
\end{equation}
According to Proposition 2.9, Proposition 5.7  and equation (2.13) of Paper I, if $N_r$ is the solution  to the RHP Problem 5.2 of Paper I, then  $M^\pm_{11}(x,0,t)_r=N^\pm_{11}(x,0,t)_r$. Following a similar argument, we  have $M^\pm_{11}(x,0,t)_\ell=N^\pm_{11}(x,0,t)_\ell$.   One can also directly read off from (6.1) and (6.2) of Paper I that 
$$N^+_{11}(x,0, t)_r=N^-_{11}(x,0, t)_r\,\,\, ,  \,\,\, N^+_{11}(x,0, t)_\ell=N^-_{11}(x,0, t)_\ell.$$
We conclude that
\begin{align}
\label{N11r.sol}
N_{11}^\pm(x,0,t)_r		&=	\exp\left( -\frac{i}{2} \int_{+\infty}^x |q(y,t)|^2 \, dy \right)\\
\label{N11l.sol}
N_{11}^\pm(x,0,t)_\ell	&=	\exp\left( -\frac{i}{2} \int_{-\infty}^x  |q(y,t)|^2 \, dy \right).
\end{align}
As we will see, we can also compute the large-$\xi$ asymptotics of $N_{11}^\pm(x,t;0)_\ell$ and 
$N_{11}^\pm(x,t;0)_r$ since these functions are the first entry in the respective solutions of the ``left'' and
``right'' Riemann-Hilbert problems for $\bfN(z,x,t)$ evaluated at $z=0$. We will obtain asymptotic formulas 
in terms of scattering data alone which prove Proposition \ref{prop:gauge+}.

\subsection{A weak Plancherel identity}
The following  lemma that can be seen as a weak version of a nonlinear  Plancherel identity.
\begin{lemma}
Suppose that $q_0 \in H^{2,2}(\bbR)$ and let $\rho$ be the scattering data. Then, the identity
\begin{equation*}
\exp \Big(i \int_{-\infty}^{+\infty} |q_0(y)|^2 \, dy\Big) = \exp \Big(-\frac{i}{\pi} \int_{-\infty}^\infty \frac{\log\left(1-s|\rho(s)|^2 \right)}{s} \, ds\Big)
\end{equation*}
holds.
\end{lemma}

\begin{proof}
The proof consists in  computing  the scattering coefficient $ a(0)$ (defined in  \eqref{matrixT}) in two ways using the construction of   left and right Beals-Coifman solutions $M_\ell, M_r$.

First, it follows from Lemma 
 5.6  of Paper I and the identity $\alpha(\zeta^2) = a(\zeta)$ that 
$$ 
\alpha(z) = 
	\exp \left( 
					\int_\bbR \frac{\log(1-\lam |\rho(\lam)|^2 )}{\lam-z} \, \frac{d\lam}{2\pi i}
			\right).
$$
Since $\rho \in H^{2,2}(\bbR)$, the function $\log(1-\lam|\rho(\lam)|^2)$ has a first-order 
zero at $\lam=0$, so that
$ \alpha(0) = \lim_{z \rarr 0, z \in  \bbC^-} \alpha(z)$ is given by
\begin{equation*}
\alpha(0) = \exp \left( 
								\int_\bbR \frac{\log(1-\lam |\rho(\lam)|^2 )}{\lam} \, \frac{d\lam}{2\pi i}
						\right)
\end{equation*} 
(although these identities are proved in Section 5 of Paper I for $\rho \in \calS(\bbR)$, their proof readily extends to $\rho \in H^{2,2}(\bbR)$). 
On the other hand, from eq.  (4.20) of Paper I, we have 
\begin{equation*}
\alpha(0)  = \lim_{x \rarr -\infty} (M_{11}^-(x,0))_{r}.
\end{equation*}
It follows from \eqref{m11R.sol} that
\begin{equation*}
 \lim_{x \rarr -\infty}(M_{11}^-(x,0))_{r} = \exp \left( \frac{i}{2} \int_{-\infty}^{+\infty} |q(y)|^2 \, dy \right).
\end{equation*}
This concludes the proof of the lemma. 
\end{proof}

\begin{remark}
When $x<0$ we reconstruct $q(x,t)$ using the left RHP, which, as shown in Proposition 6.2 of Paper I, gives a Lipschitz continuous map from soliton-free $H^{2,2}$ scattering data to $H^{2,2}(-\infty,a)$ for any fixed $a \in \bbR$.   When we use the right RHP to recover $q$ for $x>0$,  the reconstruction map is only continuous into $H^{2,2}(a,\infty)$ (see Proposition 6.1 of Paper I) but need not be stable as $x \rarr -\infty$. In this case, the gauge transformation (\ref{q to u}) is still valid for the following reason:
\begin{align}
\label{x to infty}
&\exp \Big (i \int_{-\infty}^x |q(y,t)|^2 \, dy \Big)
	=\exp \Big( 
						i \int_{-\infty}^{+\infty} |q(y,t)|^2 \, dy- i \int_{x}^{+\infty} |q(y,t)|^2 \, dy
			\Big)\\
\nonumber
& \qquad \quad
	= \exp \Big(
						i\int_{-\infty}^{+\infty} |q_0(y)|^2 \, dy \Big)  \exp \Big( -i \int_{x}^{+\infty} |q(y,t)|^2 \, dy 
			\Big)\\
\nonumber
& \qquad \quad 
	=\exp \Big(
						-\frac{i}{\pi}\int_{-\infty}^{+\infty}\frac{\log(1-s|\rho(s)|^2)}{s}\,ds 
			\Big)   
			\exp \Big(  i  \int_{+\infty}^{x} |q(y,t)|^2 \, dy \Big).
\end{align}
The first term of (\ref{x to infty})  only depends on the initial data and
 the second term is stable. 
\end{remark}

\subsection{Proof of Proposition \ref{prop:gauge+}}

\begin{proof} 
The proof is a  consequence of  \eqref{t+,x+},  \eqref{t+ x-}, \eqref{t- x-} and \eqref{t- x+} below.
It suffices to evaluate $N_{11}^\pm(x,0,t)$ for large $t$ from  the spectral data 
via the RHP. 
We  compute an asymptotic expression for the first row of ${N}^\pm(x,0,t)$
using the solution formula
\begin{equation}
\label{N.bis}
 \bfN(z;x,t) = 
	\bfN^{(3)}(z;x,t) \bfN^\RHP(\zeta(z);\xi)
	\calR^{(2)}(z)^{-1}\delta^\sharp(z;\xi)^{\sigma_3}
\end{equation}
 using equations \eqref{N1.def}, \eqref{N2.def}, \eqref{N3.def} of Section \ref{sec:summary}, 
where 
(see \eqref{delta.app} and \eqref{bdelta.app} for the definitions of $\delta_\ell$ and $\delta_r$)
\begin{equation}
\label{delta.sharp}
\delta^\sharp(z;\xi) 	=	\begin{cases}	
									\delta_\ell(z;\xi)		&	t>0,	\,  	x>0 	\\
									\delta_r(z;\xi)			&	t>0, 	\, 	x<0	\\
									\delta_r(z;\xi)^{-1}	&	t<0,	\,	x>0	\\
									\delta_\ell(z;\xi)^{-1}&	t<0,	\, 	x<0
								\end{cases}
\end{equation}
and the respective formulas 
\begin{equation}
\label{Nzero.lim}
\bfN(0;x,t) = 	\begin{cases}
						\lim_{z \rarr 0, z \in \Omega_1} \bfN(z;x,t)		&	t>0, \, x >0 \\
						\lim_{z \rarr 0, z \in \Omega_4} \bfN(z;x,t) 		&  t>0, \, x<0 \\
						\lim_{z \rarr 0, z \in \Omega_3} \bfN(z;x,t)  	& t<0, \, x>0 \\
						\lim_{z \rarr 0, z \in \Omega_6} \bfN(z;x,t)  	& t<0, \, x< 0 .
					\end{cases}
\end{equation}
 Let us examine each right-hand factor of \eqref{N.bis} in turn. Since
$$ \bfN^{(3)}(z;x,t) = (1,0) + \bigO{t^{-3/4}},$$
we need to consider only the last three factors. 

Since $\bfN^\RHP(z;x,t)$ is continuous at $z=0$ (if $\xi \neq 0$), we may evaluate
\begin{align*}
\lim_{z \rarr 0} \bfN^\RHP(\zeta(z);\xi) 
	&=	\bfN^\RHP(\sqrt{8|t|}\xi;\xi) \\
	&=	\twomat{1}{0}{0}{1} + \bigO{\frac{1}{\sqrt{8|t|}\xi}}  
\end{align*}

We show that, in each  case of \eqref{Nzero.lim}, $\lim_{z \rarr 0} \calR^{(2)}(z;x,t)^{-1}$ is the 
identity matrix when the limit is taken in the prescribed sector. 

\begin{itemize}
\item
$t>0$, $x>0$: 	The function $R_1(z)$ is continuous near $z=0$ and $R_1(0)=0$ (Figure \ref{fig:R.++.+-} and equation \eqref{R.++}).
\item
$t>0$, $x<0$: 		The function $R_4(z)$ is continuous near $z=0$ and $R_4(0)=0$ (Figure \ref{fig:R.++.+-} and equation \eqref{R.+-}).
\item
$t<0$, $x>0$: 		The function $R_3(z)$ is continuous near $z=0$ and $R_3(0)=0$ (Figure \ref{fig:R.-+.--} and equation \eqref{R.-+}).
\item
$t<0$, $x<0$: 		The function $R_6(z)$ is continuous near $z=0$ and $R_6(0)=0$ (Figure \ref{fig:R.-+.--} and equation \eqref{R.--}).
\end{itemize}

Finally, we evaluate $\lim_{z \rarr 0} \delta(z,\xi)$ for the appropriate choice of $\delta$. 
\begin{itemize}
\item
$t>0$, $x>0$: 		$\xi<0$ and $z=0$ lies to the right of the branch cut  (Figure \ref{fig:V20++})
\item
$t>0$, $x<0$: 		$\xi>0$ and $z=0$ lies to the left of the branch cut (Figure \ref{fig:V20+-})
\item
$t<0$, $x>0$: 		$\xi>0$ and $z=0$ lies to the left of the branch cut (Figure \ref{fig:V20-+})
\item
$t<0$, $x<0$: 	$\xi<0$ and $z=0$ lies to the right of the branch cut (Figure \ref{fig:V20--})
\end{itemize}
In all cases, $\delta$ is continuous at $z=0$ and $\lim_{z \rarr 0} \delta(z;\xi)^{\sigma_3} = \delta^\sharp(0;\xi)^{\sigma_3}$. Finally we arrive at
\begin{equation}
\label{N.zero.asy}
\bfN(0;x,t) = (\delta^\sharp(0;\xi),0) + \bigO[\xi]{t^{-1/2}}
\end{equation}
where $\delta^\sharp$ is given by \eqref{delta.sharp}. We now use \eqref{N11l.sol} and \eqref{N11r.sol} together with 
\eqref{N.zero.asy} to prove Proposition \ref{prop:gauge+} in four cases.

\medskip
In the following, we assume $\xi$ is fixed, thus letting $x$ and $t$ to infinity.

\medskip

\textbf{The case $t>0$, $x>0$:} 
We solve the right RHP
(see \eqref{RHP.right} and the summary in Appendix \ref{app:4-RHPS.++}).
Using  \eqref{N.zero.asy}, we have 
$$
 N_{11}^+(x,t;0)_r= \delta_\ell(0)  
	+ \mathcal{O}_\xi\left(\frac{1}{\sqrt{t}}\right).
$$
On the other hand, 
$$ \delta_\ell(z) = 
	\exp\left( 	
		\int_{-\infty}^\xi
				\frac{\log(1-s|\rho(s)|^2)}{s-z} 
		\, \frac{ds}{2\pi i} 
	\right)
$$
Hence,
\begin{align*}
\delta(0)
	&=	\lim_{z \rarr 0, \, z \in \bbC^+}
				\exp\left(\int_{-\infty}^\xi
						\frac{\log(1-s|\rho(s)|^2)}{s-z}
						\, \frac{ds}{2\pi i} 
						\right)\\
&=	\exp\left(
					\int^\xi_{-\infty}
							\frac{\log(1-s|\rho(s)|^2)}{s}
							\, \frac{ds}{2\pi i}
					\right).
\end{align*}
and
\begin{equation}
\label{n11-1}
N_{11}^+(x,t;0)_r =  
					\exp\left(
					 \int^\xi_{-\infty}
							\frac{\log(1-s|\rho(s)|^2)}{s}
							\, \frac{ds}{2\pi i}
					\right) +
					 \mathcal{O}_\xi\left(\frac{1}{\sqrt{t}}\right).
\end{equation}
Using \eqref{m11R.sol} and \eqref{n11-1} we conclude that
$$ 
\exp \Big( -\frac{i}{2} \int_{+\infty}^x |q(y,t)|^2 \, dy \Big)
=  \exp \Big( \int^\xi_{-\infty}
 \frac{\log(1-s|\rho(s)|^2)}{s}\, \frac{ds}{2\pi i} \Big)
+ \mathcal{O}_\xi\left(\frac{1}{\sqrt{t}}\right).
$$
which leads to 
\begin{equation}
\label{t+,x+}
\exp \Big( i \int_{-\infty}^x |q(y,t)|^2 \, dy \Big)
=
 \exp \Big( - i \int_\xi^{+\infty}
 \frac{\log(1-s|\rho(s)|^2)}{s}\, \frac{ds}{\pi} \Big)
+ \mathcal{O}_\xi\left(\frac{1}{\sqrt{t}}\right).
\end{equation} 

\medskip

\textbf{The case $t>0$, $x<0$ :}
 We 
use the the left-hand RHP 
(see \eqref{RHP.left} and the summary in Appendix \ref{app:4-RHPS.+-}).
From \eqref{N.zero.asy} we conclude that 
$$N_{11}^-(x,t;0)_\ell = \delta_r(0) + \mathcal{O}_\xi\left( t^{-1/2} \right)
$$
Now
\begin{align*}
\delta_r(0)	&=	
	\lim_{z \rarr 0, \, z \in \Omega_4}
		\exp \left(  
					-\int^{\infty}_\xi
						\frac{\log(1-s|\rho(s)|^2)}{s-z} 
					\, \frac{ds}{2\pi i} 
				\right)\\
	&= \exp \left(  
				- \int^{\infty}_\xi
						\frac{\log(1-s|\rho(s)|^2)}{s} 
					\, \frac{ds}{2\pi i} 
				\right)
\end{align*}
 This gives
\begin{equation}
\label{n11-2}
N_{11}^-(x,t;0)_\ell = \exp \left(  
				 -\int^{\infty}_\xi
						\frac{\log(1-s|\rho(s)|^2)}{s} 
					\, \frac{ds}{2\pi i} 
				\right) + \mathcal{O}_\xi\left( \frac{1}{\sqrt{t}} \right).
\end{equation}
We deduce from \eqref{n11-2} and \eqref{m11L.sol} that
\begin{equation}
\label{t+ x-}
\exp \Big( i \int_{-\infty}^x |q(y,t)|^2 \, dy \Big) 
= \exp \Big(
- i \int^{\infty}_\xi
	\frac{\log(1-s|\rho(s)|^2)}{s} 
	\, \frac{ds}{\pi }  \Big)
 + \mathcal{O}_\xi\left( \frac{1}{\sqrt{t}} \right).
\end{equation}

\medskip

\textbf{The case $t<0$, $x>0$ :} 
We use the asymptotic formulas 
for  the right-hand RHP \eqref{RHP.right} of  Appendix \ref{app:4-RHPS.-+}.
From \eqref{N.zero.asy} we conclude that
$$
N_{11}^+(x,t;0)_r = \delta_r(0)^{-1}  + \mathcal{O}_\xi\left( t^{-1/2} \right).
$$
Now
\begin{align*}
\delta_r(0)^{-1}	&=	
	 \exp \left(  
				\int^{\infty}_\xi
						\frac{\log(1-s|\rho(s)|^2)}{s} 
					\, \frac{ds}{2\pi i} 
				\right).
\end{align*}
This gives
\begin{equation}
\label{n11-4}
N_{11}^+(x,t;0)_r = \exp \left(  
				\int^{\infty}_\xi
						\frac{\log(1-s|\rho(s)|^2)}{s} 
					\, \frac{ds}{2\pi i} 
				\right) + \mathcal{O}_\xi\left( \frac{1}{\sqrt{t}} \right).
\end{equation}
From \eqref{n11-4} and \eqref{m11R.sol},  we get
$$ \exp \Big(
-\frac{i}{2} \int_{+\infty}^x |q(y,t)|^2 \, dy \Big)
= \exp \Big(
\int^{\infty}_\xi
	\frac{\log(1-s|\rho(s)|^2)}{s} 
	\, \frac{ds}{2\pi i}  \Big)
 + \mathcal{O}_\xi\left( \frac{1}{\sqrt{t}} \right)
$$
which leads to 
\begin{equation}
\label{t- x+}
\exp \Big(i \int_{-\infty}^x |q(y,t)|^2 \, dy  \Big)
=\exp \Big( -  i \int_{-\infty}^\xi
	\frac{\log(1-s|\rho(s)|^2)}{s} 
	\, \frac{ds}{\pi }  \Big)
 + \mathcal{O}_\xi\left( \frac{1}{\sqrt{t}} \right).
\end{equation}

\medskip

\textbf{The case $t<0$, $x<0$:}
 Using the asymptotic formula
for the  left-hand RHP of Appendix \ref{app:4-RHPS.--}.
 and  \eqref{N.zero.asy} we have
$$
N_{11}^-(x,t;0)_\ell= \delta_r(0)^{-1}  + \mathcal{O}_\xi\left( t^{-1/2} \right)
$$
 From
\begin{align*}
\delta_r(0)^{-1}	&=	
	 \exp \left(  
				- \int_{-\infty}^\xi
						\frac{\log(1-s|\rho(s)|^2)}{s} 
					\, \frac{ds}{2\pi i} 
				\right),
\end{align*}
 we have
\begin{equation}
\label{n11-3}
N_{11}^-(x,t;0)_\ell = \exp \left(  
				-\int_{-\infty}^\xi
						\frac{\log(1-s|\rho(s)|^2)}{s} 
					\, \frac{ds}{2\pi i} 
				\right) + \mathcal{O}_\xi\left( \frac{1}{\sqrt{t}} \right).
\end{equation}
Finally from \eqref{n11-3} and \eqref{m11L.sol},
\begin{equation}
\exp \Big( \label{t- x-}
\int_{-\infty}^x |q(y,t)|^2 \, dy \Big)
=\exp \Big( - i \int_{-\infty}^\xi
	\frac{\log(1-s|\rho(s)|^2)}{s} 
	\, \frac{ds}{\pi }  \Big)
 + \mathcal{O}_\xi\left( \frac{1}{\sqrt{t}} \right).
\end{equation}
\end{proof}

%
%

\appendix  

%
%

\section{Solutions to model scalar RHPs}
\label{app:delta-asy}

\subsection{Large-$z$ Asymptotics}

Since $\kappa \in H^{2,2}(\bbR)$, it follows that $s \kappa(s) \in L^1(\bbR)$ and we may expand
\begin{align}
\label{log.delta.infty}
 \int_{-\infty}^\xi\frac{\kappa(s)}{s-z} \, ds
	&=	-\frac{1}{z}  \int_{-\infty}^\xi\ \kappa(s) \, ds 
			- \frac{1}{z} \int_{-\infty}^\xi\ \frac{s}{s-z} \kappa(s) \, ds\\[5pt]
	&= 	-\frac{1}{z}  \int_{-\infty}^\xi\ \kappa(s) \, ds + \bigO{\frac{1}{z^2}}.
	\nonumber
\end{align}
where the implied constant is uniform in $z$ with 
$-\pi+\eps < \arg(z-\xi) < \pi-\eps$ for a fixed $\eps>0$. Using \eqref{log.delta.infty} in \eqref{delta.app} we 
conclude that
\begin{equation*}
\delta_\ell(z)	\sim	1 - \frac{i}{z} \int_{-\infty}^\xi\ \kappa(s) \, ds + \bigO{\frac{1}{z^2}},
\end{equation*}
and, by a similar argument
\begin{equation*}
\delta_r(z) 		\sim	1	+	\frac{i}{z} \int_\xi^\infty \kappa(s) \, ds +  \bigO{\frac{1}{z^2}}.
\end{equation*}

\subsection{Asymptotics Near the Stationary Phase Point}

The following asymptotic relations for $\delta_\ell$, $\delta_r$, $\delta_r^{-1}$, and $\delta_\ell^{-1}$
are used to compute leading asymptotics near the critical point $\xi$ and determine the model RHPs.
Define complex powers of $(z-\xi)$ using the appropriate branch of the 
logarithm ($-\pi < \arg (\zeta-\xi) < \pi$ for $\delta_\ell^{\pm 1}$, and $0 < \arg (\zeta-\xi) < 2\pi$ for $\delta_r^{\pm 1}$).  
As $z \rarr \xi$ in the respective domains of $\delta_\ell$ and $\delta_r$,
\begin{align}
\label{delta.xi.asy.+}
\left|\delta_\ell(\zeta)-\delta_{0\ell}(z-\xi)^{i\kappa(\xi)}	\right| 			
	&\lesssim	-|z-\xi| \log |z-\xi|		\\[5pt]
\label{bdelta.xi.asy+}
\left|\delta_r(\zeta)-\delta_{0r}(z-\xi)^{i\kappa(\xi)} \right|					
	&\lesssim	-|z-\xi| \log |z-\xi|		\\[5pt]
\label{delta.xi.asy.-}
\left|\delta_r(\zeta)^{-1}-\delta_{0r}^{-1}(z-\xi)^{-i\kappa(\xi)}	\right| 	
	&\lesssim	-|z-\xi| \log |z-\xi|		\\[5pt]
\label{bdelta.xi.asy.-}
\left|\delta_\ell(\zeta)^{-1}-\delta_{0\ell}^{-1}(z-\xi)^{-i\kappa(\xi)} \right|		
	&\lesssim	-|z-\xi| \log |z-\xi|,
\end{align}
where the implied constants depend on $\norm{\kappa}{H^{2,2}}$ and a fixed $\eps>0$.  The constants
are uniform in $z$ with
$-\pi+\eps < \arg(z-\xi) < \pi-\eps$ (for $\delta_\ell^{\pm 1}$) or $\eps < \arg (z-\xi) < 2\pi-\eps$ (for $
\delta_r^{\pm 1}$).

The constants $\delta_{0\ell}$ and $\delta_{0r}$
are defined as follows.
Let $\chi_-$ be the characteristic function of $(\xi-1,\xi)$, and let $\chi_+$ be the characteristic function of 
$(\xi,\xi+1)$. Then:

\begin{align*}
\delta_{0\ell}	&=		\exp
									\left( 
										i \int_{-\infty}^\xi \frac{\kappa(s) - \chi_-(s)\kappa(\xi)}{s-\xi} \, ds
									\right),	
									\\
\delta_{0r}		&=	e^{\pi \kappa(\xi)} 		
								\exp
									\left(
										-i\int_\xi^\infty \frac{\kappa(s) - \chi_+(s)\kappa(\xi)}{s-\xi} \, ds
									\right),
\end{align*}
These asymptotics are easily deduced from the integral formulas \eqref{delta.app} and \eqref{bdelta.app}. We illustrate the ideas for $\delta_\ell$; these computations are standard but we include them for the reader's convenience.

Using \eqref{delta.app}, we compute,
for $z \in \bbC \setminus (-\infty,\xi]$, 
\begin{align*}
\delta_\ell(z)	&=	\exp\left( 
								i  \int_{\xi-1}^{\xi} 
									\frac{\kappa(\xi)}{s-z} \, ds 
								\right) \cdot
						\exp\left( 
								i \int_{-\infty}^\xi
										\frac{\kappa(s) - \chi_-(s)\kappa(\xi)}{s-z} \, ds 
								\right)\\
					&=	(z-\xi)^{i\kappa(\xi)}  e^{i\beta(z;\xi)}
\end{align*}
where
\begin{equation*}
\beta(z;\xi)	=	-\kappa(\xi) \log(z-\xi+1)
	+  \int_{-\infty}^\xi  \frac{\kappa(s) - \chi_-(s) \kappa(\xi)}{s-z} \, ds.
\end{equation*}
We will show that $\beta(z,\xi)$ is continuous at $z=\xi$ and we set $\delta_{0\ell}(\xi)=\exp(i\beta(\xi,\xi))$. We wish to prove that
\begin{equation}
\label{delta.xi.beta.diff.est}
\left| \delta(z) - \delta_0(\xi)(\xi-z)^{-i\kappa(\xi)} \right| \lesssim_{\, \rho, \phi} -|z-\xi| \log |z-\xi| 
\end{equation}
as $z \rarr \xi$ for $z-\xi = r e^{i\phi}$ with 
 $-\pi < \phi < \pi$
and implied constants independent of $\xi \in \bbR$. To do this, it suffices to show that
$$ \left| \beta(\xi+r e^{i\phi};\xi)-\beta(\xi;\xi) \right| \lesssim_{\, \rho, \phi}  -r \log r $$
where the implied constants have the same uniformity.
But
\begin{align}
\label{delta.xi.beta.diff}
\beta(\xi+re^{i\phi};\xi) - \beta(\xi;\xi)
	&=	\kappa(\xi) \log(1+re^{i\phi}) \\
\nonumber
	&\qquad	+
			  \int_{-\infty}^\xi 
			 		\left( \frac{1}{s-z}-\frac{1}{s-\xi} \right) \left( \kappa(s) -\chi(s)\kappa(\xi) \right) 
			 	\, ds \\
\nonumber
	&=		\int_{\xi-1}^{\xi} 
					\left( \frac{1}{s-z}-\frac{1}{s-\xi} \right) 
					\left(\kappa(s)-\kappa(\xi) \right) 
				\, ds + \bigO{r}
\end{align}
where the implied constants in $\bigO{r}$ depend on $\norm{\kappa}{\infty}$  and are independent of $\xi \in \bbR$. The first right-hand integral in the last line of \eqref{delta.xi.beta.diff} may be written
\begin{align*}
 I(r;\xi) 	&	= 	re^{i\phi} 
 						\int_{\xi-1}^{\xi} 
 							\frac{1}{s-\xi - re^{i\phi}} \frac{\kappa(s)-\kappa(\xi)}{s-\xi} \, ds \\
 			&	=	re^{i\phi} 
 						\int_{\xi-1}^{\xi} 
 							 \frac{1}{s-\xi-re^{i\phi}} \kappa'(\xi) \, ds \\
 			&\quad	+	 re^{i\phi} 
 							\int_{\xi-1}^{\xi}
 								\frac{1}{s-\xi-re^{i\phi}} 
 									\frac{\int_{\xi}^s 
 										(s-y)\kappa''(y) \, dy}{s-\xi} \, ds \\[5pt]
 			&=	I_1(r;\xi) + I_2(r;\xi)
\end{align*}
By explicit computation,
\begin{equation}
\label{delta.xi.beta.diff1}
I_1(r;\xi) = re^{i\phi} \kappa'(\xi) 
 \left( \log(-re^{i\phi}) - \log(-1-re^{i\phi}) \right)
\lesssim -r \log r 
\end{equation}
with constants depending on $\kappa$ through $\norm{\kappa'}{\infty}$  and  otherwise
independent of $\xi$. On the other hand, since
$|s-y|/|s-\xi| \leq 1$ we may estimate
\begin{align*}
\left| I_2(r;\xi) \right| 	&\leq 	r \norm{\kappa''}{2} 
 \int_{\xi-1}^{\xi} 
\frac{1}{|s-\xi-re^{i\phi}|} \, ds
\end{align*}
The right-hand integral is easily seen to equal
$$ 
\int_{-\frac{1}{r}-\cot \phi}^{-\cot \phi} 
\frac{1}{\sqrt{\mu^2+1}}\, d\mu $$
which is $\bigO{\log r}$ as $r \darr 0$ with constants depending on $\phi$. 
These constants are 
bounded if 
 $\eps < \phi < \pi - \eps$ or $-\pi+\eps < \phi < -\eps$  
for some  fixed $\eps>0$. For such $\phi$ we have 
\begin{equation}
\label{delta.xi.beta.diff2}
|I_2(r;\xi)| \lesssim_{\, \rho, \phi} -r \log r
\end{equation}
with constants  independent of $\xi \in \bbR$ and depending on $\rho$ through $\norm{\rho}{H^{2,2}}$ since 
$\norm{\rho}{H^{2,2}}$ controls $\norm{\kappa''}{2}$. 

Since $\norm{\rho}{H^{2,2}}$ also controls $\norm{\kappa}{\infty}$ and $\norm{\kappa'}{\infty}$, we conclude from
\eqref{delta.xi.beta.diff}, \eqref{delta.xi.beta.diff1}, and \eqref{delta.xi.beta.diff2}  that \eqref{delta.xi.beta.diff.est} holds.

\section{Four model RHPs}
\label{app:4-RHPs}
We summarize the key formulas leading to $q_{\mathrm{as}}(x,t)$. 
We will write $\kappa$ for $\kappa(\xi)$ when it appears in 
formulas. We denote by $\eta(z;\xi)$ or simply $\eta$ the function 
$$ \eta(\zeta;\xi) = (z-\xi). $$
Thus $\eta^{\pm i\kappa}$ is shorthand for $(z-\xi)^{\pm i\kappa(\xi)}$, etc. 
We will make use of the identities 
$$ 
\overline{\Gamma(z)}=\Gamma(\overline{z}), \quad 
|\Gamma(i\kappa)|^2 = \frac{ \pi }{ \kappa \sinh(\pi \kappa)}
$$
as well as 
$$ e^{-2\pi\kappa} = 1- \xi |\rho(\xi)|^2 = 1- \xi |\brho(\xi)|^2 $$
in the computations.
Recall that the symbols $\delta$, $\delta_0$, and $\delta_\pm$ are 
defined at the beginning of each subsection and 
\emph{have different meanings in each of them} as indicated in  \eqref{delta.++.def}, \eqref{delta.+-.def}, 
\eqref{delta.-+.def}, and \eqref{delta.--.def}.

\subsection{The Case $t>0$, $x>0$}
\label{app:4-RHPS.++}

To prepare the initial RHP for steepest descent we set
$ \bfN^{(1)} = \bfN \delta_\ell^{-\sigma_3}$. Throughout this subsection,
\begin{equation}
\label{delta.++.def}
\delta = \delta_\ell, 
\quad
\delta_\pm =(\delta_\ell)_\pm,
\quad
\delta_0=\delta_{0\ell}.
\end{equation}
From \eqref{RHP.right} we get a new RHP for $\bfN^{(1)}$ with jump matrix $V^{(1)}$ where
\JumpMatrixFactors{V^{(1)}}{V1++}
	{\Twomat{1}{0}{-\dfrac{\delta_-^{-2} z \rhobar}{1-z|\rho|^2}e^{-2it\theta}}{1}}
	{\Twomat{1}{\dfrac{\delta_+^2\rho}{1-z|\rho|^2}e^{2it\theta}}{0}{1}}
	{\Twomat{1}{\rho \delta^2e^{2it\theta}}{0}{1}}
	{\Twomat{1}{0}{-z\rhobar \delta^{-2}e^{-2it\theta}}{1}}
	
$\bfN^{(1)}$ is then ready for steepest descent. We reduce to 
a mixed $\dbar$-RHP in the new variable $\bfN^{(2)} = \bfN^{(1)} \calR$ where $\calR$ is a piecewise continuous matrix-valued  as shown in Figure \ref{fig:R.++.+-}. Here

\RMatrix{R.++}{z\overline{\rho(z)}\delta(z)^{-2}}
			{\xi \overline{\rho(\xi)}\delta_0^{-2}\eta^{-2i\kappa}}
			{-\dfrac{\rho(z)\delta_+^2(z)}{1-z|\rho(z)|^2}}
			{-\dfrac{\rho(\xi) \delta_0^2}{1-\xi|\rho(\xi)|^2}\eta^{2i\kappa}}
			{-\dfrac{z \overline{\rho(z)} \delta_-^{-2}}{1-z|\rho(z)|^2}}
			{-\dfrac{\xi \overline{\rho(\xi)} \delta_0^{-2}}{1-\xi|\rho(\xi)|^2}\eta^{-2i\kappa}}
			{\rho(z)\delta(z)^2}
			{\rho(\xi)\delta_0 \eta^{2i\kappa}}

The resulting unknown $\bfN^{(2)}$ satisfies a mixed $\dbar$-RHP with jump matrix
$V^{(2)}$ defined on the oriented contours of $\Sigma^{(2)}_\xi$.

As discussed above we reduce to a model RHP with contour $\Sigma$ and jump matrix \eqref{V2}
where $V^{(2)}_0$ is shown in Figure \ref{fig:V20++} and
\begin{equation}
\label{r.++}
r_\xi
	=	\rho(\xi)\delta_{0}^2 
			e^{-2i\kappa}e^{-2i\kappa\log(\sqrt{8t})}
			e^{4it\xi^2}
\end{equation}

Using \eqref{beta12.+},  \eqref{alpha.sq.bis},  \eqref{alpha.arg.bis}, and \eqref{r.++}, we conclude that
\begin{align*}
|\alpha(\xi)|	^2	&=	\frac{\kappa(\xi)}{2\xi},	\\
\arg \alpha(\xi)	
	&=	\frac{\pi}{4}
				+\arg\Gamma(i\kappa)+\arg\rho(\xi)\\
\nonumber
	&\quad
				+\frac{1}{\pi}\int_{-\infty}^{\xi} \log | s - \xi | \, d\log(1-s|\rho(s)|^2).
\end{align*}

\subsection{The Case $t>0$, $x<0$}
\label{app:4-RHPS.+-}

To prepare for steepest descent we set
$\bfN^{(1)} =\bfN \delta_r^{-\sigma_3}$.
Throughout this subsection,
\begin{equation}
\label{delta.+-.def}
\delta = \delta_r,
\quad
\delta_\pm = (\delta_r)_\pm,
\quad
\delta_0=\delta_{0r}.
\end{equation}
The new RHP for $\bfN^{(1)}$ has jump matrix $ \bV^{(1)}$ where
\JumpMatrixFactors{\bV^{(1)}}{V1+-}
	{\Twomat{1}{0}{-z\overline{\brho} \delta^{-2}e^{-2it\theta}}{1}}
	{\Twomat{1}{\brho \delta^2 e^{2it\theta}}{0}{1}}
	{\Twomat{1}{\dfrac{\brho \delta_-^2}{1-z|\brho|^2}e^{2it\theta}}{0}{1}}
	{\Twomat{1}{0}{\dfrac{-z\overline{\brho} \delta_+^{-2}}{1-z|\brho|^2}e^{-2it\theta}}{1}}
$\bfN^{(1)}$ is then ready for steepest descent. As before we reduce to a mixed $\dbar$-RHP problem n the new variable $\bfN^{(2)} = \bfN^{(1)} \calR$, where $\calR$ is the piecewise continuous matrix-valued function as shown in Figure \ref{fig:R.++.+-}. We have the following formulas for $R_1$, $R_3$, $R_4$, and $R_6$:
\RMatrix{R.+-}{\dfrac{z\overline{\brho(z)}\delta_+(z)^{-2}}{1-z|\rho(z)|^2}}
			{\dfrac{\xi \brho(\xi) \delta_0^{-2}}{1-\xi|\rho(\xi)|^2}\eta^{-2i\kappa}}
			{-\brho(z) \delta_+(z)^2}
			{-\brho(\xi) \delta_0^2 \eta^{2i\kappa}}
			{-z\overline{\brho(z)}\delta_-(z)^{-2}}
			{-\xi \overline{ \brho(\xi)}\delta_0^{-2}\eta^{-2i\kappa}}
			{\dfrac{\brho(z)}{1-z|\brho(z)|^2}\delta_-(z)^2}
			{\dfrac{\brho(\xi)}{1-\xi |\brho(\xi)|^2}\delta_0^2 \eta^{2i\kappa}}
The new unknown $\bfN^{(2)}$ satisfies a mixed $\dbar$-RHP in $\bfN^{(2)}$ with jump matrix $V^{(2)}$ on $\Sigma^{(2)}_\xi$. 

Following the procedure outlined at the beginning of this section we arrive at a model RHP with contour $\Sigma^{(2)}_0$ and jump matrix \eqref{V2} where $V_0^{(2)}$ is shown in Figure \ref{fig:V20+-} and
\begin{align}
\label{r.+-}
\br_{\xi}& =  \brho(\xi) \delta_{0}^2 e^{-2i\kappa(\xi)\log\sqrt{8t}}e^{4it\xi^2}.\\
\nonumber
       & =\brho(\xi)e^{2\kappa\pi} \exp\left(-2i\int_\xi^\infty \frac{\kappa(s) - \chi(s)\kappa(\xi)}{s-\xi} \, ds\right) e^{-2i\kappa(\xi)\log\sqrt{8t}}e^{4it\xi^2}
\end{align}

From \eqref{beta12.+}, \eqref{alpha.sq.bis},  \eqref{alpha.arg.bis}, and \eqref{r.+-}, we conclude that
\begin{align*}
|\alpha(\xi)|	^2	&=	\frac{\kappa(\xi)}{2\xi}	\\
\arg \alpha(\xi)	
	&=	-\frac{3\pi}{4}
				+\arg\Gamma(i\kappa)+\arg\rho(\xi)  \\
\nonumber
	&\quad
				+\frac{1}{\pi}\int_{-\infty}^{\xi}\log(\xi-s)d\log(1-s|\rho(s)|^2).
\end{align*}

\subsection{The Case $t<0$, $x>0$}
\label{app:4-RHPS.-+}

In what follows we will set $t'=-t$ so that $|t|=t'$ and 
$$ \theta(z;x,t) = -\left( -z \frac{x}{t'} + 2z^2 \right). $$

To prepare the initial RHP for steepest descent we take $\bfN^{(1)}=\bfN \delta_r^{\sigma_3}$.
Throughout this subsection
\begin{equation}
\label{delta.-+.def}
\delta = \delta_r^{-1},
\quad
\delta_\pm = \left(\delta_r^{-1}\right)_\pm,
\quad
\delta_0=\delta_{0r}^{-1}.
\end{equation}
The resulting RHP for $\bfN^{(1)}$ has jump matrix $V^{(1)}$ where
\JumpMatrixFactors{V^{(1)}(z)}{V1-+}
	{\Twomat{1}{\rho\delta^2 e^{-2it'\theta}}{0}{1}}
	{\Twomat{1}{0}{-z\overline{\rho}e^{2it'\theta}}{1}}
	{\Twomat{1}{0}{\dfrac{-z\rhobar \delta_-^{-2}}{1-z|\rho|^2}e^{2it'\theta}}{1}}
	{\Twomat{1}{\dfrac{\rho \delta_+^2}{1-z|\rho|^2}e^{-2it'\theta}}{0}{1}}
We write $\bfN^{(1)}=\bfN^{(2)} \calR$ where the piecewise continuous matrix-valued function $\calR$ is shown in Figure \ref{fig:R.-+.--}, and the functions $R_i$ are described as follows:
\RMatrix{R.-+}{-\dfrac{\rho(z) \delta_+(z)^2}{1-z|\rho(z)|^2}}
		{-\dfrac{\rho(\xi) \delta_0^2}{1-\xi|\rho(\xi)|^2} \eta^{2i\kappa}}
		{z\overline{\rho(z)}\delta_+(z)^{-2}}
		{\xi\overline{\rho(\xi)}\delta_0^{-2} \eta^{-2i\kappa}}
		{\rho(z)\delta_-(z)^2}
		{\rho(\xi)\delta_0^2 \eta^{2i\kappa}}
		{\dfrac{-z\overline{\rho(z)}\delta_-(z)^{-2}}{1-z|\rho(z)|^2}}
		{\dfrac{-\xi\overline{\rho(\xi)}\delta_0^{-2}}{1-\xi|\rho(\xi)|^2} \eta^{-2i\kappa}}

The function $\bfN^{(2)}$ obeys a mixed $\dbar$-RHP with jump matrix $V^{(2)}$ that we describe below.

Following the standard procedure we arrive at a model RHP with contour $\Sigma^{(2)}_0$ and jump matrix \eqref{V2} 
where $V_0^{(2)}$ is shown in Figure \ref{fig:V20-+} and
\begin{align}
\label{r.-+}
r_{\xi}& =\rho(\xi)\delta_{0r}^{-2} e^{2i\kappa(\xi)\log\sqrt{8t'}}e^{-4it'\xi^2}.\\
\nonumber
       & =\rho(\xi)e^{-2\kappa\pi} \exp\left(2i\int_\xi^\infty \frac{\kappa(s) - \chi(s)\kappa(\xi)}{s-\xi} \, ds\right) e^{2i\kappa(\xi)\log\sqrt{8t'}}e^{-4it'\xi^2}
\end{align}
From \eqref{beta12.-}, \eqref{alpha.sq.bis}, and \eqref{alpha.arg.bis}, and \eqref{r.-+}, we conclude that
\begin{align*}
|\alpha(\xi)|	^2	&=	\frac{\kappa(\xi)}{2\xi}	\\
\arg \alpha(\xi)	
	&=	\frac{3\pi}{4}
				-\arg\Gamma(i\kappa)+\arg\rho(\xi)\\
\nonumber
	&\quad 
				+\frac{1}{\pi}\int^{\infty}_{\xi}\log | s - \xi | \, d\log(1-s|\rho(s)|^2).
\end{align*}

\subsection{The Case $t<0$, $x<0$}
\label{app:4-RHPS.--}

To prepare for steepest descent we set $\bfN^{(1)} = \bfN \delta_\ell^{\sigma_3}$.
Throughout this subsection,
\begin{equation}
\label{delta.--.def}
\delta=\delta_\ell^{-1},
\quad
\delta_\pm = (\delta_\ell^{-1})_\pm,
\quad
\delta_0=\delta_{0\ell}^{-1}.
\end{equation}
The 
resulting RHP for $\bfN^{(1)}$ has jump matrix $ \bV^{(1)}$ where
\JumpMatrixFactors{\bV^{(1)}}{V1--}
	{\Twomat{1}{\dfrac{\brho(z)}{1-z|\brho(z)|^2}\bdelta_-^2 e^{-2it'\theta}}{0}{1}}
	{\Twomat{1}{0}{\dfrac{-z\overline{\brho(z)}\bdelta_+^{-2}}{1-z|\brho(z)|^2}e^{2it'\theta}}{1}}
	{\Twomat{1}{0}{-z\overline{\brho(z)}\bdelta^{-2}e^{2it'\theta}}{1}}
	{\Twomat{1}{\brho(z)\bdelta^2 e^{-2it'\theta}}{0}{1}}
We can now deform to a mixed $\dbar$-RHP by passing to $\bfN^{(2)} = \bfN^{(1)} \calR$ where $\calR$ is the piecewise continuous matrix-valued function shown in Figure \ref{fig:R.-+.--}, and the functions $R_i$ have the boundary values:
\RMatrix{R.--}{-\brho(z) \delta_+(z)^2}
			{-\brho(\xi) \delta_0^2 \eta^{2ik}}
			{\dfrac{z\brho(z)}{1-z|\rho(z)|^2}\delta_+^{-2}}
			{\dfrac{\xi \overline{\brho(\xi)}}{1-\xi|\rho(\xi)|^2}\delta_0^{-2} \eta^{-2i\kappa}}
			{\dfrac{\brho(z)}{1-z|\brho(z)|^2} \delta_-^2}
			{\dfrac{\brho(\xi)}{1-\xi |\brho(\xi)|^2} \delta_0^2 \eta^{2i\kappa}}
			{-z\overline{\brho(z)} \delta(z)^{-2}}
			{-\xi \overline{\brho(\xi)} \delta_0^{-2} \eta^{-2i\kappa}}
 The new unknown $\bfN^{(2)}$ satisfies a mixed $\dbar$-RHP with jump matrix $\bV^{(2)}$ on 
 $\Sigma^{(2)}_\xi$. 

Following the procedure outlined at the beginning of the section we arrive at a model RHP with contour $\Sigma^{(2)}_0$ 
and jump matrix \eqref{V2}, 
where $V^{(2)}_0$ is shown in Figure \ref{fig:V20--} and
\begin{equation}
\label{r.--}
\br_{\xi}=\brho(\xi)\delta_{0\ell}^{2} e^{2i\kappa(\xi)\log\sqrt{8t'}}e^{-4it'\xi^2}.
\end{equation}

From \eqref{beta12.-}, \eqref{alpha.sq.bis}, \eqref{alpha.arg.bis}, and \eqref{r.--}, we conclude that
\begin{align*}
|\alpha(\xi)|	^2	&=	\frac{\kappa(\xi)}{2\xi}\\
\arg \alpha(\xi)	
	&=	- \frac{\pi}{4}
					-\arg\Gamma(i\kappa)+\arg\rho(\xi)\\
\nonumber
	&\quad
					+\frac{1}{\pi}\int_{\xi}^{\infty}\log | s - \xi | \, d\log(1-s|\rho(s)|^2).
\end{align*}

\section{Formulae and Wronskian  for parabolic cylinder functions}
\label{app:Phi-sol}

We record the solution formulae for $\Phi(\zeta,\xi)$ arising in the factorization of the model 
RHP in each of the four cases $\pm t>0$,$\pm x>0$; see Step 4 of    Section  \ref{sec:summary} 
and especially 
\eqref{RHP.Model.factor} for the set-up; see also \eqref{Phi.DE} and the comments following for the solution 
method. These formulae together with the Wronskian  identity for parabolic cylinder functions, allow the  evaluations  of   \eqref{Phi.Wronski.+} and \eqref{Phi.Wronski.-} that in turn provide  $\beta_{12}$ in terms of frozen-coefficient scattering data.

We give explicit formulae for the solutions of the equations \eqref{Phi.DE} with asymptotic behavior
$$ \Phi(\zeta;\xi) \sim e^{\mp \frac{i}{4}\zeta^2 \sigma_3} \zeta^{\pm i\kappa \sigma_3} 
	\left( I + \frac{m^{(1)}}{\zeta} + \littleO{\zeta^{-1}} \right). $$
We denote by $D_a(z)$ the usual parabolic cylinder function, i.e., the solution to   \eqref{para-cyl} 
with asymptotics  prescribed in \eqref{para-123}.
The identity  \eqref{D.recur} 
is easily be derived from the relation 
\begin{equation}
\label{U.to.D}
U(a,z) = D_{-a-\frac{1}{2}}(z)
\end{equation}
(see \cite[\S 12.1]{DLMF}\footnote{\url{http://dlmf.nist.gov/12.1}}) and 
\cite[12.8.2]{DLMF}\footnote{\url{http://dlmf.nist.gov/12.8.E2}}. 
We also record the Wronskian identity
\begin{equation}
\label{D.Wronski}
W(D_{a}(z),D_{a}(-z)) = \frac{\sqrt{2\pi}}{\Gamma(-a)}
\end{equation}
 which is a consequence of \eqref{U.to.D} and 
\cite[eq. (12.2.11)]{DLMF}\footnote{\url{http://dlmf.nist.gov/12.2.E11}}. We use this identity 
to compute $\beta_{12}$ 
 (see proof of Lemma \ref{lemma:beta12}).

For the $+$ case of \eqref{Phi.DE}, taking $-\pi < \arg \zeta < \pi$ corresponding to $t>0$, $x>0$,
the solution $\Phi(\zeta;\xi)$  is given by expressions \eqref{SimpliPhi+} and \eqref{SimpliPhi-} of
Proposition \ref{explicit}.

For the $+$ case of \eqref{Phi.DE}, taking $0 < \arg \zeta < 2\pi$ corresponding to 
$t>0$, $x<0$, the solution $\Phi(\zeta;\xi)$ is given by
$$
\begin{cases}
\Twomat{ e^{-\frac{3\pi } {4} \kappa }																
			D_{i\kappa}(\zeta e^{-\frac{3i\pi}{4}})}				
		{-\dfrac{i\kappa}{\beta_{21}} e^{\frac{\pi}{4}(\kappa-i)	}							
			D_{-i\kappa-1}(\zeta e^{-\frac{\pi i}{4}})}
		{\dfrac{i\kappa}{\beta_{12}}																	
			e^{-\frac{3\pi}{4}(\kappa+i)}
			D_{i\kappa-1}(\zeta e^{-\frac{3i\pi}{4}})}		
		{e^{\frac{\pi \kappa}{4}}  D_{-i\kappa}( e^{-i\pi/4} \zeta) }							
	&	\zeta \in \bbC^+,\\
\\
\Twomat{e^{-\frac{7\pi \kappa}{4}} 																
		D_{i\kappa}(\zeta e^{-\frac{7\pi i}{4}})}   
		{\dfrac{-i\kappa}{\beta_{21}}																	
		 e^{\frac{5\pi}{4}(\kappa-i)}D_{-i\kappa-1}(\zeta e^{-\frac{5\pi i}{4}})}                                                                                                                       
		{\dfrac{i\kappa}{\beta_{12}}																	
			e^{-\frac{7\pi}{4}(\kappa+i)}
			D_{i\kappa-1}(\zeta e^{-\frac{7 \pi i}{4}})}                  					
		{e^{\frac{5\pi\kappa}{4}}D_{-i\kappa}(\zeta e^{-\frac{5\pi i}{4}})}				
	&	\zeta \in \bbC^-
\end{cases}
$$

For the $-$ case of \eqref{Phi.DE}, taking $0< \arg \zeta < 2\pi$ corresponding to  $t<0$, $x>0$,
the solution $\Phi(\zeta;\xi)$ is 
$$
\begin{cases}
\Twomat{e^{\frac{\pi}{4} \kappa} 																	
					D_{-i\kappa} (\zeta e^{-\frac{\pi i}{4}})}											
			{\dfrac{i\kappa}{\beta_{21}} e^{ -\frac{3\pi}{4} (\kappa + i)}					
					D_{i\kappa-1}(\zeta e^{-\frac{3\pi i}{4}})}
			{\dfrac{-i\kappa}{\beta_{12}} e^{{ \frac{\pi}{4} (\kappa}-i)}						
					D_{-i\kappa-1}(\zeta e^{-\frac{\pi i}{4}})}																	
			{e^{-\frac{3\pi \kappa}{4}}  D_{i\kappa}( \zeta e^{-\frac{3\pi i}{4}} )}		
			&	 \zeta \in \bbC^+,\\
			\\
\Twomat{e^{\frac{5\pi}{4}\kappa} D_{-i\kappa} (e^{-\frac{5\pi i}{4}}\zeta)}			
			{\dfrac{i \kappa}{\beta_{21}} e^{{ -\frac{7\pi}{4} (\kappa+i)}}					
				 D_{i\kappa-1}(\zeta e^{-\frac{7\pi i}{4}}) }																	
			{\dfrac{ -i\kappa }{\beta_{12}}e^{\frac{5\pi}{4}(\kappa-i)}						
				D_{-i\kappa-1}(\zeta e^{-5\pi i/4})}
			{e^{-\frac{7\pi}{4}\kappa} D_{i\kappa} ( e^{-\frac{7\pi i}{4}} \zeta)}			
			&  \zeta \in \bbC^-. 
\end{cases}
$$

Finally, for the $-$ case of \eqref{Phi.DE}, taking $-\pi < \arg \zeta < \pi$ corresponding to $t<0$, $x<0$,
the solution $\Phi(\zeta;\xi)$ is 
$$
\begin{cases}
\Twomat{e^{\frac{\pi \kappa}{4}} D_{-i\kappa}(\zeta e^{-\frac{\pi i}{4}})}				
	{\dfrac{i\kappa}{\beta_{21}} e^{-\frac{3\pi}{4}(\kappa+i)} 								
			D_{i\kappa-1} (\zeta e^{-\frac{3\pi i}{4}})}
	{\dfrac{-i\kappa}{\beta_{12}} e^{\frac{\pi}{4}(\kappa-i)} 								
			D_{-i\kappa-1} (\zeta e^{-\frac{\pi i}{4}})}
	{e^{\frac{-3\pi\kappa}{4}}																			
			D_{i\kappa}(\zeta e^{-\frac{3\pi i}{4}})   }
	&	\zeta \in \bbC^+,\\
\\
\Twomat{e^{-\frac{3\pi \kappa}{4}} D_{-i\kappa} (\zeta e^{\frac{3\pi i}{4}})}			
	{\dfrac{i\kappa}{\beta_{21}} e^{\frac{\pi}{4}(\kappa+i)} 								
			D_{i\kappa-1} (\zeta e^{\frac{\pi i}{4}})}
	{\dfrac{-i\kappa}{\beta_{12}} e^{-\frac{3\pi}{4}(\kappa-i)} 								
			D_{-i\kappa-1}(\zeta e^{\frac{3\pi i}{4}})}		
	{e^{\frac{\pi\kappa}{4}}D_{i\kappa}(\zeta e^{\frac{i\pi}{4}})   }						
	&	\zeta \in \bbC^-
\end{cases}
$$

From these formulae and the identities \eqref{D.recur} and \eqref{D.Wronski}, we can compute (cf. \eqref{Phi.Wronski.+}--\eqref{Phi.Wronski.-})
\begin{equation}
\label{Phi.Wronski.+.comp}
\Phi^-_{11}\Phi^+_{21} - \Phi_{21}^-\Phi_{11}^+
	=	\begin{cases}
			\dfrac{1}{\beta_{12}}e^{-\pi \kappa/2} e^{\pi i/4} \dfrac{\sqrt{2\pi}}{\Gamma(-i\kappa)}	&	t>0, \, x>0 \\
			\\
			\dfrac{1}{\beta_{12}}e^{-\pi \kappa/2} e^{i\pi /4} \dfrac{\sqrt{2\pi}}{\Gamma(-i\kappa)} e^{-2\pi \kappa}	& t>0, \, x<0
		\end{cases}
\end{equation}
and 
\begin{equation}
\label{Phi.Wronski.-.comp}
\Phi^-_{11}\Phi^+_{21} - \Phi_{21}^-\Phi_{11}^+
	=	\begin{cases}
			\dfrac{1}{\beta_{12}}e^{-\pi \kappa/2}e^{3i\pi/4} \dfrac{\sqrt{2\pi}}{\Gamma(i\kappa)} e^{2\pi \kappa},	
					&	t<0, \, x >0 \\
			\\
			\dfrac{1}{\beta_{12}}e^{-\pi \kappa/2} e^{3\pi i/4} \dfrac{\sqrt{2\pi}}{\Gamma(i\kappa)},
					&	t<0, \, x<0
		\end{cases}
\end{equation}


\section{$L^\infty$-Bounds for the Model RHP}
\label{app:NRHP.bd}

We prove the bounds \eqref{RHP.bd1} and \eqref{RHP.bd2}. 
it suffices to prove \eqref{RHP.bd1} since 
the bound \eqref{RHP.bd2} follows from \eqref{RHP.bd1} and the fact that $\bfN^\PC(\zeta;\xi)$ takes values in $SL(2,\bbC)$.
together with explicit estimates on the parabolic cylinder functions $D_a(\zeta)$, following a similar discussion in 
\cite[\S3.1.1, Lemma 3.5]{CP14}.

\begin{lemma}
Let $c_1$ and $c_2$ be strictly positive constants, and 
suppose that $\rho \in H^{2,2}(\bbR)$ with $\rho$ with $\norm{\rho}{H^{2,2}} \leq c_1$, $\inf_{z \in \bbR} (1-z |\rho(z)|^2) \geq c_2$. Then,
the estimate
$$
\left| \bfN^\PC(\zeta;\xi) \right|	\lesssim	1\\
$$
holds,
where the implied constant depend only on $c_1$ and $c_2$.
\end{lemma}

\begin{proof}
We give the bound for the region $\Omega_1$ since estimates for the other regions are similar. 
 Using  \eqref{tildeN}, \eqref{P-matrix},  \eqref{SimpliPhi+}, \eqref{SimpliPhi-}   and  writing 
$$p_1(\xi)=r_\xi/(1-\xi|r_\xi|^2),$$ 
we have that,   for $\zeta$ with $0 < \arg(\zeta) < \pi/4$,  the entries $N_{ij} $ of $\bfN^\PC$
are given by 
\begin{align*}
N_{11}(\zeta;\xi) 	
	&=	e^{\pi\kappa/4}e^{-\frac{i}{4}\zeta^2} \zeta^{i\kappa} 
				D_{-i\kappa}(\zeta e^{-i\pi/4}) \\[5pt]
N_{12}(\zeta;\xi) 
	&=	p_1(\xi)
				e^{\frac{i}{4}\zeta^2}\zeta^{-i\kappa}
				e^{\pi \kappa/4} 
				D_{-i\kappa}(\zeta e^{-i\pi/4})\\
	&\quad
		+ \frac{1}{\beta_{21}} 
			e^{-3\pi \kappa/4} 
			e^{-3\pi i/4} ( i\kappa)
			e^{\frac{i}{4}\zeta^2} \zeta^{-i\kappa}
			D_{i\kappa-1} (\zeta e^{-3\pi i/4}) \\[10pt]
N_{21}(\zeta;\xi)
	&=	e^{\pi \kappa/4} e^{-\pi i/4}
			(-i\kappa)
			e^{-\frac{i}{4}\zeta^2}\zeta^{i\kappa}
			D_{-i\kappa-1}(\zeta e^{-i\pi/4})	\\[5pt]
N_{22}(\zeta;\xi)
	&=	\frac{1}{\beta_{12}}
			p_1(\xi)
			e^{\frac{i}{4}\zeta^2} \zeta^{-i\kappa}
			e^{\frac{\pi }{4}\kappa} e^{-i \pi/4} 
			(-i\kappa) D_{-i\kappa-1}(\zeta e^{-i\pi/4})\\
	&\quad
		+	
		e^{ -3 \pi \kappa/4}
			e^{\frac{i}{4}\zeta^2} \zeta^{-i\kappa}
			D_{i\kappa}(\zeta e^{-3 i \pi /4}).
\end{align*}
Since
$$
D_{-i\kappa}(\zeta e^{-i\pi/4}) 
	\sim e^{-\pi \kappa/4} \zeta^{-i\kappa} e^{\frac{i}{4}\zeta^2},
\quad
D_{i\kappa}(\zeta e^{-3i\pi/4})
	\sim	e^{3 \pi \kappa/4} e^{\frac{i}{4} \zeta^2}
$$
it is clear that $\bfN^\PC(\zeta;\xi) \rarr I$ as $\zeta \rarr \infty$ in $\Omega_1$. 
To prove the uniform $L^\infty$-estimate, we need a quantitative version of the asymptotics. 
We claim that, uniformly in $a$, in compacts of $\bbC$  and $z$ with $|z| \geq 1$ and $|\arg(z)| < 3\pi /4$, the estimate
\begin{equation}
\label{para-cyl-est}
\left| e^{z^2/4} z^{-a} D_a(z) \right| \lesssim 1
\end{equation}
 holds.
The uniform $L^\infty$- estimate will follow from the boundedness of $\kappa$, the fact that $\left| e^{\frac{i}{4}\zeta^2} \right| \leq 1$ for $\zeta \in \Omega_1$,  and the estimates
\begin{align*}
\left| 	e^{-\frac{i}{4} \zeta^2} \zeta^{i\kappa} D_{-i\kappa}(\zeta e^{-i/\pi/4}) \right| 			&\lesssim	1\\
\left| 	e^{-\frac{i}{4} \zeta^2} \zeta^{-i\kappa} D_{i\kappa}(\zeta e^{-3i\pi/4}) \right| 			&\lesssim	1\\
\left|	e^{-\frac{i}{4} \zeta^2} \zeta^{-i\kappa} D_{-i\kappa-1}(\zeta e^{-i\pi/4} ) \right| 		&\lesssim 	1\\
\left|  e^{-\frac{i}{4} \zeta^2} \zeta^{-i\kappa} D_{i\kappa -1}(\zeta e^{-3\pi i/4}) \right|		&\lesssim 	1
\end{align*}
which are a consequence of \eqref{para-cyl-est}.

To complete the proof, we recall from \cite{CP14} the proof of \eqref{para-cyl-est}. The parabolic 
cylinder function $D_a(z)$ can be expressed in terms of the Whittaker function $W_{k,\mu}(z)$   \cite{WW15}
(see Lemma \ref{lemma:Whittaker} below)
via the formula
\begin{equation}
\label{Whitt-to-par}
D_a(\zeta) = 2^{\frac{1}{4}+\frac{a}{2}} \zeta^{-1/2} W_{\frac{1}{4} + \frac{a}{2},-1/4}(\zeta^2/2)
\end{equation}
while, for $|\arg(z)| < 3\pi /2$, the Whittaker function admits the integral representation
\begin{equation}
\label{Whitt}
W_{\frac{1}{4}+\frac{a}{2},-1/4}(z) = e^{-z/2} z^{\frac{1}{4}+\frac{a}{2}}
		\left[ 	
			1 -\frac	{\Gamma(\frac{3}{2}-a)\Gamma(1-\frac{a}{2})}
						{\Gamma(\frac{1}{2}-\frac{a}{2})\Gamma(-\frac{a}{2})} \frac{1}{z}	
		+ R(a,z)
		\right]
\end{equation}
where
\begin{equation}
\label{Whitt.R}
R(a,z) = 
				\frac{1}{\Gamma(\frac{1}{2}-\frac{a}{2})\Gamma(-\frac{a}{2})}
						\int_{-i\infty-\frac{3}{2}}^{+i\infty-\frac{3}{2}}
								z^\zeta \Gamma(\zeta) \Gamma(-\zeta+\frac{1}{2}-\frac{a}{2})
											\Gamma(-\zeta-\frac{a}{2}) \, d\zeta										
\end{equation}
The computations in \cite[proof of Lemma 3.5]{CP14} show that 
\begin{equation}
\label{Whitt.R.est}
\left| R(a,z) \right| \lesssim |z|^{-3/2}\left(\frac{3}{2}\pi - |\arg(z)|\right)^{-3/2},
\end{equation}
where the implied constant depends only on $c_1$ and $c_2$, if $a=\pm i \kappa$ or $a=\pm i \kappa -1$. This estimate, \eqref{Whitt-to-par}, \eqref{Whitt}, and \eqref{Whitt.R.est} imply \eqref{para-cyl-est}.
\end{proof}

\begin{lemma}
\label{lemma:Whittaker}
The integral representation \eqref{Whitt} holds.
\end{lemma}

\begin{proof}
We begin with the following representation formula from \cite[(13.16.11)]{DLMF}:\footnote{\url{http://dlmf.nist.gov/13.16.E11}}
\begin{equation*}
W_{k,\mu}(z) = \frac{e^{-\frac{1}{2}z}}{2\pi i}
	\int_{-i \infty}^{+i \infty} 
			\frac	{\Gamma(\frac{1}{2}+ \mu + t)\Gamma(\frac{1}{2}-\mu+t)\Gamma(-k-t)}
					{\Gamma(\frac{1}{2}+ \mu - k)\Gamma(\frac{1}{2}-\mu-k)}
			z^{-t} \, dt
\end{equation*}
where the contour separates the poles of $\Gamma(\frac{1}{2}+ \mu + t)\Gamma(\frac{1}{2}-\mu+t)$
from those of $\Gamma(-k-t)$, and $|\arg(z)| < 3\pi/2$.  Thus, taking $k=\frac{1}{2}+\frac{a}{2}$
and $\mu=-\frac{1}{4}$, we obtain
\begin{equation*}
W_{\frac{a}{2}+\frac{1}{4}, -\frac{1}{4}}(z) = \frac{e^{-\frac{1}{2}z}}{2\pi i}
	\int_{-i \infty}^{+i \infty} 
			\frac	{\Gamma(\frac{1}{4} + t)\Gamma(\frac{3}{4}+t)\Gamma(-\frac{a}{2}-\frac{1}{4}-t)}
					{\Gamma(-\frac{a}{2})\Gamma(\frac{1}{2}-\frac{a}{2})}
			z^{-t} \, dt
\end{equation*}
We wish to set $t=\zeta-\left(\frac{1}{4}+\frac{a}{2}\right)$. If $a = \pm i\kappa$ this contour shift can be made without picking up contributions from poles. We recover
\begin{multline*}
W_{\frac{a}{2}+\frac{1}{4}, -\frac{1}{4}}(z) = \frac{e^{-\frac{1}{2}z} z^{\frac{1}{4}+\frac{a}{2}}}{2\pi i}
	\times\\
	\frac{1}{\Gamma(-\frac{a}{2})\Gamma(\frac{1}{2}-\frac{a}{2})}
	\int_{-i\infty}^{+i \infty} 
		\Gamma\left(\zeta-\frac{a}{2}\right) 
		\Gamma\left(\frac{1}{2}+ \zeta -\frac{a}{2}\right)
		\Gamma(-\zeta) z^{-\zeta} \, d\zeta
\end{multline*}
We can now obtain a large-$z$ expansion by shifting the contour to the right. We will pick up poles at $\zeta=0,1,\cdots$ depending on how far we shift. It is easy to compute the residues of the integrand at $\zeta=0$ and $\zeta=1$ using the facts that 
$\Gamma(-\zeta)=\Gamma(1-\zeta)/(-\zeta) = \Gamma(2-\zeta)/(-\zeta(1-\zeta))$. Note that the residues get multiplied by $-1$ in the computations since we shift the contour to the right. We then obtain
\begin{align*}
W_{\frac{a}{2}+\frac{1}{4}, -\frac{1}{4}}(z) 
	&= 	\frac{e^{-\frac{1}{2}z} z^{\frac{1}{4}+\frac{a}{2}}}{2\pi i} \times \\
	&		
			\left( 1 - 
				\frac	{\Gamma(1-\frac{a}{2})\Gamma(\frac{3}{2}-\frac{a}{2})}
						{\Gamma(-\frac{a}{2})\Gamma(\frac{1}{2}-\frac{a}{2})}
						\frac{1}{z}
			\right. \\
	&		-
			\left.
				\frac{1}{\Gamma(-\frac{a}{2})\Gamma(\frac{1}{2}-\frac{a}{2})}
					\int_{-i\infty+\frac{3}{2}}^{+i \infty+\frac{3}{2}}
						\Gamma\left(\zeta-\frac{a}{2}\right) 
						\Gamma\left(\frac{1}{2}+ \zeta -\frac{a}{2}\right)
						\Gamma(-\zeta) z^{-\zeta} \, d\zeta
			\right)
\end{align*}
A trivial change of variable gives \eqref{Whitt.R}.
\end{proof}

%
%

\newpage

\section{Figures}
\label{app:figures}

{
\SixMatrix{The Matrix  $\calR^{(2)}$ for $t>0$, $\pm x>0$}{fig:R.++.+-}
	{\twomat{1}{0}{R_1 e^{-2it\theta}}{1}}
	{\twomat{1}{R_3  e^{2it\theta}}{0}{1}}
	{\twomat{1}{0}{R_4e^{-2it\theta}}{1}}
	{\twomat{1}{R_6e^{2it\theta}}{0}{1}}
}

{
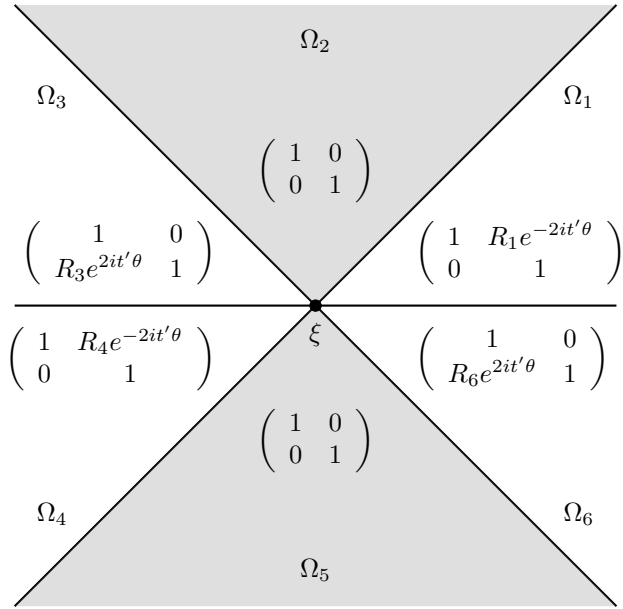
\SixMatrix{The Matrix  $\calR^{(2)}$ for $t<0$, $\pm x>0$ (note that $t'=-t$)}{fig:R.-+.--}
	{\twomat{1}{R_1e^{-2it'\theta}}{0}{1}}
	{\twomat{1}{0}{R_3e^{2it'\theta}}{1}}
	{\twomat{1}{R_4e^{-2it'\theta}}{0}{1}}
	{\twomat{1}{0} {R_6e^{2it'\theta}}{1}}
}

\newpage

\JumpMatrixLeftCut{The Jump Matrix $V^{(2)}_{0}$ for $t>0$, $x>0$}{fig:V20++}
	{\Twomat{1}{0}{\xi \overline{r_\xi}}{1}}
	{\Twomat{1}{-\dfrac{r_\xi}{1-\xi|r_\xi|^2}}{0}{1}}
	{\Twomat{1}{0}{\dfrac{-\xi \overline{r_\xi}}{1-\xi |r_\xi|^2}}{1}}
	{\Twomat{1}{r_\xi}{0}{1}}

\JumpMatrixRightCut{The Jump Matrix $V^{(2)}_{0}$ for $t>0$, $x<0$}{fig:V20+-}
	{\Twomat{1}{0}{\dfrac{\xi\overline{\br_\xi}}{1-\xi|\br_\xi|^2}}{1}}
	{\Twomat{1}{-\br_\xi}{0}{1}}
	{\Twomat{1}{0}{-\xi\overline{\br_\xi}}{1}}
	{\Twomat{1}{\dfrac{\br_\xi}{1-\xi|\br_\xi|^2}}{0}{1}}

\JumpMatrixRightCut{The Jump Matrix $V^{(2)}_0$ for $t<0$, $x>0$}{fig:V20-+}
	{\Twomat{1}{\dfrac{-r_\xi}{1-\xi|r_\xi|^2}}{0}{1}}
	{\Twomat{1}{0}{\xi\overline{r_\xi}}{1}}
	{\Twomat{1}{r_\xi}{0}{1}}
	{\Twomat{1}{0}{\dfrac{-\xi\overline{r_\xi}}{1-\xi|r_\xi|^2}}{1}}

\JumpMatrixLeftCut{The Jump Matrix $V^{(2)}_0$ for $t<0$, $x<0$}{fig:V20--}
	{\Twomat{1}{-\br_\xi}{0}{1}}
	{\Twomat{1}{0}{\dfrac{\xi \overline{\br_\xi}}{1-\xi|\br_\xi|^2}}{1}}
	{\Twomat{1}{\dfrac{\br_\xi}{1-\xi|\br_\xi|^2}}{0}{1}}
	{\Twomat{1}{0}{-\xi \overline{\br_\xi}}{1}}

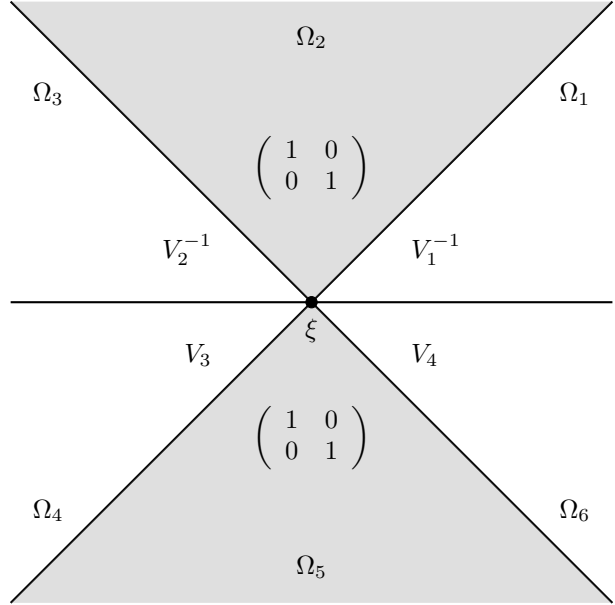
\SixMatrix{The Matrix $P$ in terms of the Jump Matrix $V^{(2)}_0$, where $V_i = \left. V^{(2)}_0 \right|_{\Sigma_i}$}	
	{fig:V.to.P}
	{V_1^{-1}}
	{V_2^{-1}}
	{V_3}
	{V_4}
	
%
%

\FloatBarrier

\bigskip\noindent
 {\bf Acknowledgments.} We thank R. Jenkins and K. McLaughlin for useful discussions, and for sharing with us their recent preprint \cite{BJM16}  with M. Borghese. J. L. and P. P. thank the Department of Mathematics at the University of Toronto and the Fields Institute for hospitality during part of the time that this work was done.


\begin{thebibliography}{99}


\bibitem{BJM16}
Borghese, M., Jenkins, R., McLaughlin, K. T.-R. 
Long-time asymptotic behavior of the focusing nonlinear Schr\"{o}dinger
equation. Preprint, \href{http://arXiv.org/pdf/1604.07436.pdf}{arXiv:1604.07436}.

\bibitem{CJ14}
Cuccagna, S.,  Jenkins, J.
On asymptotic stability of N-solitons of the defocusing nonlinear Schr\"odinger equation. 
Preprint, 
\href{http://arXiv.org/pdf/1410.6887.pdf}{arXiv:1410.6887},
to appear in \emph{Comm.\ Math.\ Phys.}

\bibitem{CP14}
Cuccagna, S.,  Pelinovsky, D. E. 
The asymptotic stability of solitons in the cubic NLS equation on the line. \emph{Appl. Anal.} \textbf{93} (2014), no. 4, 791--822.

\bibitem{DIZ93}
Deift, P. A.; Its, A. R.; Zhou, X. Long-time asymptotics for integrable nonlinear wave equations. \emph{Important developments in soliton theory}, 181--204, Springer Ser. Nonlinear Dynam., Springer, Berlin, 1993.

\bibitem{DZ93}
Deift, P.,  Zhou, X. A steepest descent method for oscillatory Riemann-Hilbert problems. Asymptotics for the MKdV equation. \emph{Ann. of Math.} (2) \textbf{137} (1993), 295--368. 

\bibitem{DZ94}
Deift, P. A.; Zhou, X. Long-time asymptotics for integrable systems. Higher order theory. Comm. Math. Phys., \textbf{165} (1994), no. 1, 175--191.

\bibitem{DZ03}
Deift, P.,  Zhou, X. (2003). Long-time asymptotics for solutions of the NLS equation with initial 
data in a weighted Sobolev space. Dedicated to the memory of J\"urgen K. Moser. 
\emph{Comm. Pure Appl. Math.}, \textbf{56} (2003), 1029--1077.
 
\bibitem{DM08}
Dieng, M., McLaughlin, K D.-T.  Long-time Asymptotics for the NLS equation via dbar methods. 
Preprint,
\href{http://arXiv.org/pdf/0805.2807.pdf}{arXiv:0805.2807},
2008.

\bibitem{DLMF}
NIST Digital Library of Mathematical Functions. 
\url{http://dlmf.nist.gov/}, Release 1.0.11 of 2016-06-08. 
Online companion to \cite{OLBC10}.

 
\bibitem{Do11}
Do, Y. (2011).
A nonlinear stationary phase method for oscillatory Hilbert-Riemann problem.
\emph{Intern. Math. Res. Not.}, \textbf{12}  (2011), 2650--2765.

\bibitem{F00}
Fan, E.  (2000).
Darboux transformation and soliton-like solutions for the Gerdjikov-Ivanov equation.
\emph{J. Physics A}  {33}:6925--6933.

\bibitem{HNU99}
Hayashi, N. Naumkin, P. Uchida, H. 
large-time behavior of solutions for derivative cubic nonlinear Schr\"{o}dinger equations. 
\emph{Publ. Res. Inst. Math. Sci.} 
\textbf{35} (1999), no. 3, 501--513.
 
\bibitem{Its81}
Its, A. R. Asymptotic behavior of the solutions to the nonlinear Schr\"odinger equation, and 
isomonodromic deformations of systems of linear differential equations.  (Russian) 
\emph{Dokl. Akad. Nauk SSSR} \textbf{261} (1981), no. 1, 14--18. English translation in 
\emph{Soviet Math. Dokl.\ } \textbf{24} (1982), no. 3, 452--456.
 
\bibitem{KN78}
Kaup, D. J., Newell, A. C.  An exact solution for a derivative nonlinear Schr\"{o}dinger equation. 
\emph{J. Mathematical Phys.} \textbf{19} (1978), 798--801.

\bibitem{KV97}
Kitaev, A. V.; Vartanian, A. H. Leading-order temporal asymptotics of the modified nonlinear Schr\"odinger equation: solitonless sector. Inverse Problems 
\textbf{13} (1997), no. 5, 1311--1339.

\bibitem{Lee83}
Lee, J.-H. 
Analytic properties of Zakharov-Shabat inverse scattering problem with a polynomial spectral dependence of degree 1 in the potential. Thesis (Ph.D.), 1983, Yale University.

\bibitem{LPS15}
Liu, J.,  Perry, P.,   Sulem, C.  Global existence for the derivative nonlinear Schr\"{o}dinger equation by the method of inverse scattering.
Preprint, 
\href{http://arXiv.org/pdf/1511.01173.pdf}{arXiv:1511.01173},
2015, to appear in \emph{Comm.\ P.\ D.\ E.}

\bibitem{MM08}
McLaughlin, K. T.-R.; Miller, P. D. 
The $\dbar$ steepest descent method and the asymptotic behavior of polynomials orthogonal on the unit circle with fixed and exponentially varying nonanalytic weights. 
\emph{IMRP Int. Math. Res. Pap. }
(2006), Art.\ ID 48673, 1--77.

\bibitem{OLBC10}
F. W. J. Olver, D. W. Lozier, R. F. Boisvert, and C. W. Clark, editors. 
NIST Handbook of Mathematical Functions. 
Cambridge University Press, New York, NY, 2010. 
Print companion to \cite{DLMF}.

\bibitem{PS16}
Pelinovsky,  D.E.,  Shimabukuro, Y.
Existence of global solutions to the derivative NLS equation with the inverse scattering transform method.
Preprint, \href{http://arXiv.org/pdf/1602.02118.pdf}{arXiv:1602.02118}, 2016.

\bibitem{V96}
Varzugin, G. G. 
Asymptotics of oscillatory Riemann-Hilbert problems. 
\emph{J. Math. Phys.}
\textbf{37} (1996), no. 11, 5869--5892.

\bibitem{WW15}
Whittaker, E.T., Watson, G.G.
A course in modern analysis, Cambridge Univ. Press, 1915.

\bibitem{XuFan12}
Xu, J., Fan, E. 
Long-time asymptotic for the derivative nonlinear Schr\"{o}dinger equation with decaying initial value.
Preprint, \href{http://arXiv.org/pdf/1209.4245.pdf}{arXiv:1209.4245}, 2012.

\bibitem{ZM76}
Zakharov, V.E., Manakov, S.V.
Asymptotic behavior of nonlinear wave systems integrated by the inverse scattering method.
\emph{Soviet Physics JETP} {\bf 44} (1976), no. 1, 106--112; translated from \emph{Z. Eksper. Teoret. Fiz.} {\bf 71} (1976), no. 1, 203--215.


\end{thebibliography}
\end{document}